\documentclass[12pt]{amsart}
\usepackage[dvipsnames,usenames]{color}
\usepackage{hyperref}
\usepackage{graphicx}
\usepackage{epsfig}
\usepackage[latin1]{inputenc}
\usepackage{amsmath}
\usepackage{amsfonts}
\usepackage{amssymb}
\usepackage{amsthm}
\usepackage{amscd}
\usepackage{verbatim}
\usepackage{subfig}
\usepackage{pinlabel}
\usepackage{mathtools}
\usepackage{stmaryrd}
\usepackage{enumerate, enumitem}
\usepackage{array,multirow}

\usepackage{graphicx}
\usepackage{thmtools}
\usepackage{thm-restate}
\usepackage[dvipsnames]{xcolor}

\usepackage[textsize=scriptsize]{todonotes}
\usepackage{bm}
\usepackage{tikz}
\usetikzlibrary{cd}
\usetikzlibrary{arrows}
\usetikzlibrary{decorations.pathreplacing}
\usepackage{verbatim}
\usepackage[all]{xy}

\newtheorem*{theorem*}{Theorem}
\newtheorem{theorem}{Theorem}[section]
\newtheorem{lemma}[theorem]{Lemma}

\newtheorem{corollary}[theorem]{Corollary}
\newtheorem{proposition}[theorem]{Proposition}

\theoremstyle{definition}
\newtheorem{definition}[theorem]{Definition}

\newtheorem{remark}[theorem]{Remark}

\newtheorem{example}[theorem]{Example}

\def\it{i^{(t)}_{m,s}}
\def\ib{i^{(b)}_{m,s}}
\def\d{\partial}

\def\R{\mathbb{R}}

\def\Z{\mathbb{Z}}

\def\F{\mathbb{F}}

\def\CZhat{\widehat{\cC}_\Z}
\def\CZ{{\cC}_\Z}

\def\X{\mathbb{X}}
\def\II {\mathcal{I}}

\def\JJ {\mathcal{J}}
\def\dij {\Delta_{\II,\JJ}}

\def\aa {\widetilde{\alpha}}
\def\bb {\widetilde{\beta}}

\def\l {\ell}
\def\co{\colon\thinspace}

\def\x {\mathbf{x}}

\def\w {\omega}

\def\H{\mathcal{H}}

\def\CZhat{\widehat{\mathcal{C}}_\Z}
\def\CZ{{\mathcal{C}}_\Z}

\def\CFKi {\operatorname{CFK}^\infty}

\def\max {\operatorname{max}}
\def\min {\operatorname{min}}
\def\wa{{\widetilde \alpha}}
\def\wb{{\widetilde \beta}}
\DeclarePairedDelimiter{\ceil}{\lceil}{\rceil}
\DeclarePairedDelimiter{\floor}{\lfloor}{\rfloor}
\DeclareMathOperator{\sgn}{sgn}

\def\wH{\widetilde{H}}

\usetikzlibrary{arrows.meta,bending,decorations.markings,intersections} 
\tikzset{
    arc arrow/.style args={%
    to pos #1 with length #2}{
    decoration={
        markings,
         mark=at position 0 with {\pgfextra{%
         \pgfmathsetmacro{\tmpArrowTime}{#2/(\pgfdecoratedpathlength)}
         \xdef\tmpArrowTime{\tmpArrowTime}}},
        mark=at position {#1-\tmpArrowTime} with {\coordinate(@1);},
        mark=at position {#1-2*\tmpArrowTime/3} with {\coordinate(@2);},
        mark=at position {#1-\tmpArrowTime/3} with {\coordinate(@3);},
        mark=at position {#1} with {\coordinate(@4);
        \draw
        (@1) .. controls (@2) and (@3) .. (@4);},
        },
     postaction=decorate,
     }
}

\title{On homology concordance in contractible manifolds and two bridge links}

\author{Hugo Zhou}
\address{School of Mathematics, Georgia Institute of Technology, Atlanta, GA, USA}
\email{hzhou@gatech.edu}

\begin{document}

\maketitle

\begin{abstract}
Let $\CZhat$ be the group consists of manifold-knot pairs $(Y,K)$ modulo homology concordance, where $Y$ is an integer homology sphere bounding an integer homology ball, and let $\CZ$ be the subgroup consisting of pairs $(S^3,K)$.
Dai-Hom-Stoffregen-Truong \cite{Homoconcor} show that the quotient group $\CZhat/\CZ$ admits a $\Z^\infty$-summand.  In this paper, we improve the result by showing that there exists a family $\{(Y,K_m)\}_{m>1 }$ generating  the $\Z^\infty$-summand  where $Y$ is the boundary of a smooth contractible $4$-manifold. In fact, we give a $\Z$-count of such families.

The examples are constructed using a family of knots obtained by blowing down a component of a two-bridge link. They are studied in Jonathan Hales's thesis \cite{jhales}. Using the algorithm due to Ozsv\'{a}th, Szab\'{o} and Hales we give a classification of the knot Floer homology of a larger family of such knots, that might be of independent interest.
\end{abstract}

\section{Introduction}\label{sec:intro}
The integer homology concordance group $\CZhat$ consists of pairs $(Y,K)$ where  $Y$ is an integer homology sphere bounding an integer homology ball and $K$ is a knot in $Y$, where the group operation is induced by the connected sum. Two classes $(Y_1,K_1) $ and $ (Y_2,K_2) $ are equivalent if and only if  $\partial(W,\Sigma)= (Y_1,K_1) \sqcup -(Y_2,K_2),$ where $W$ is an integer homology cobordism and $\Sigma$ a smoothly embedded cylinder in $W$. The subgroup $\CZ$ consists of pairs $(S^3,K).$ A class $(Y,K) \neq 0$ in $\CZhat/ \CZ$ if and only $K$ is not concordant to any knot in $S^3$ in any homology cobordism, or equivalently if and only $K$ does not bound any PL-disk in any homology ball with boundary $Y.$ 

The first nontrival class $(Y,K)\in \CZhat/\CZ$ is found by Levine \cite{nonsur}, building on Akbulut's work \cite{Akbulut}. Hom-Levine-Lidman \cite{homologyconcordance} prove that $\CZhat/ \CZ$ is infinitely-generated and admits a $\Z$-subgroup. Using the infinite family found in \cite{hugo}, Dai-Hom-Stoffregen-Truong \cite{Homoconcor} show  that  $\CZhat/\CZ$ admits a $\Z^\infty$-summand. 

Among each infinite family of manifold-knot pairs which gives rise to the $\Z^\infty$-summand in the literature so far, the manifolds are  always of the form $M_n \#{-M_n}$, such that they bound  homology balls, but not necessarily smooth contractible manifolds. This paper strengthens the existing result by giving an infinite family of knots in the boundary of a smooth contractible manifold  that generates a $\Z^\infty$-summand in $\CZhat/\CZ.$  In fact, we give a $\Z$-count of such families.

Imposing that the homology spheres bound smooth contractible manifolds   removes the homotopic obstruction for the knots to bound PL-disks, and therefore makes it a more interesting and more difficult question. 
\begin{figure}
\labellist
 \pinlabel {$L_n$}  at -10 152
  \pinlabel {$1$}  at 445 210
  \pinlabel {$0$}  at 440 80
 \pinlabel {{$n$}}  at 128 156
  \pinlabel {$n+1$}  at 315 158
\endlabellist
\includegraphics[scale=0.5]{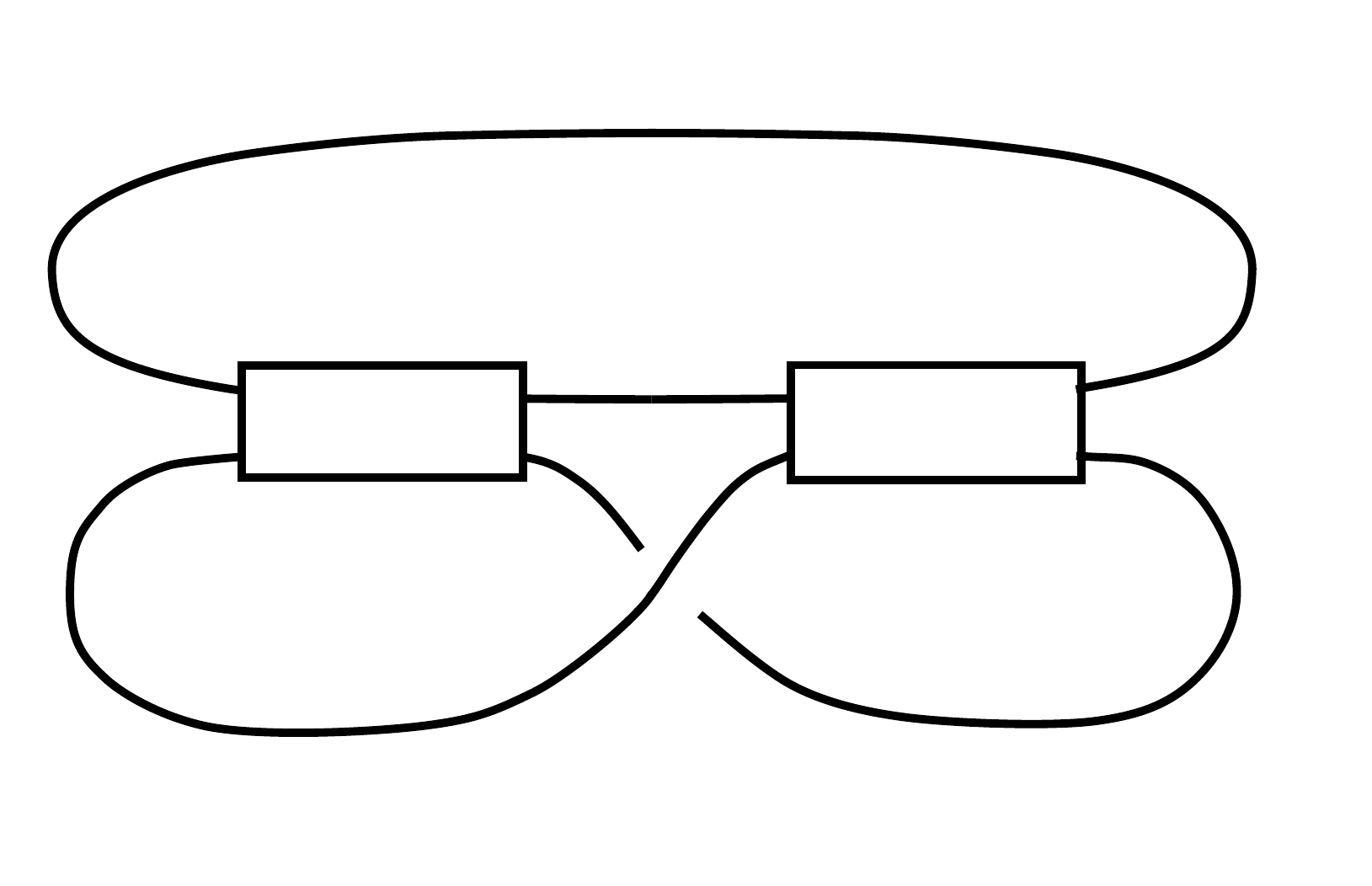}
\caption{The two bridge link $L_n$ whose $(1,0)$-framed surgery yields the integer homology sphere $Y_n$ that bounds a contractible manifold. The number in the box indicates the number of full-twists.  }
\label{fig:twobridge}
\end{figure}
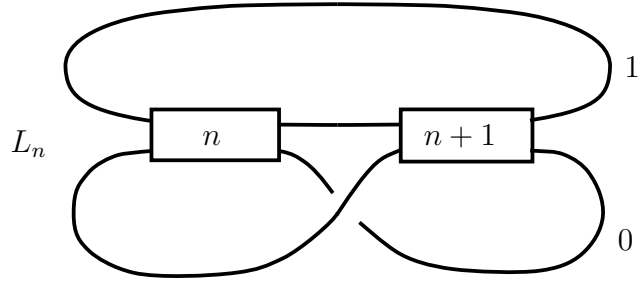

Consider the Mazur type manifold  $Y_n$   with $n\geq 1$ (see \cite{mazur}) as depicted in Figure \ref{fig:twobridge}, obtained from a $(1,0)$-framed surgery on a two-bridge link $L_n$.   From a standard argument that switches the $0$-framed two handle to a dotted circle, we see that $Y_n$ bounds a contractible $4$-manifold. (Both components of $L_n$ are unknotted and they intersect algebraically once.) Let $K_n \subset S^3$ be the  knot obtained from $L_n$ by blowing down the $+1$-framed unknot component. It follows that $S^{3}_{-1}(K_n) = Y_n$ bounds a contractible $4$-manifold. 

Let $\mu^{(2n)}_{m,1}$ be the image of the $(m,1)$-cable of the meridian  in the $-1$-surgery on $K_{2n}$. 

\begin{theorem} \label{thm:main1}
    The family $\{(Y_{2n},\mu^{(2n)}_{m,1})\}_{n>0,m>1}$ generates a $\Z^{\Z\oplus\Z}$ summand in $\CZhat/\CZ$. In particular, for each fixed $n>0$, $\{(Y_{2n},\mu^{(2n)}_{m,1})\}_{m>1}$ generates a $\Z^\Z$ summand in $\CZhat/\CZ$.
\end{theorem}
As abelian groups, $\Z^{\Z\oplus\Z}$ is of course isomorphic to $\Z^{\Z}$.
Here by a $\Z^{\Z\oplus\Z}$ summand we would like to emphasize the existence of a (natural) two-parameter family of linear independent, surjective homomorphisms from $\CZhat/\CZ$ to $\Z$. In this case the two parameters are given by $n$ and $m$. (See Lemma \ref{le:varphij2k} for the homomorphisms.)
 Note that the $\Z^\Z$ summand in \cite{hugo,homologyconcordance}  is generated by knots living in infinitely many different manifolds, while here each family of knots $\{\mu^{(2n)}_{m,1}\}_{m>1}$ lives in the same three manifold $Y_{2n}$, which is the boundary of a smooth contractible $4$-manifold.
 
Theorem \ref{thm:main1} is the direct consequence of the two following theorems. Denoting by $C_k$  the complex isomorphic to $\CFKi(S^3,T_{2,2k+1}).$ First, using the filtered mapping cone formula \cite{zhou2022filtered}, and with the help of the concordance homomorphisms defined in \cite{Homoconcor}, we show the following:
\begin{restatable}{theorem}{zinfty} \label{thm:zinfty}
Suppose a family of knots $\{J_{2k}\}_{k>0}$ satisfies that  $\CFKi(S^3,J_{2k}) \cong C_{2k} \oplus A$, where $H_*(A) = 0$.  
 Then the family $\{(S^3_{-1}(J_{2k}),\mu^{(2k)}_{m,1})\}_{k>0,m>1}$ generates a $\Z^{\Z\oplus \Z}$ summand in $\CZhat/\CZ$, where  $\mu^{(2k)}_{m,1}$ is the image of the $(m,1)$-cable of the meridian  in  the $-1$-surgery on $J_{2k}$.  In particular, for each fixed $k>0,$ $\{(S^3_{-1}(J_{2k}),\mu^{(2k)}_{m,1})\}_{m>1}$ generates a $\Z^{\Z}$ summand in $\CZhat/\CZ$.  
\end{restatable}

Next, the knot Floer complex of the knot family $K_n$ can be explicitly computed as follows.
Ozsv\'{a}th-Szab\'{o}  gives a description in \cite[Section 6.2]{OSknot} for  a genus-one doubly pointed Heegaard diagram of any knot that results from blowing down the $\pm1$-surgery on one component of a two-bridge link. The family of knots $K_n$ is studied carefully in Jonathan Hales's thesis \cite[Section 3]{jhales}. Building on Ozsv\'{a}th-Szab\'{o}'s description, Hales develops  a simple algothrim that outputs a $(1,1)$ diagram of all the knots that arise from blowing down the $\pm1$-surgery on one component of a two-bridge link.  We review the details of this algorithm in Section \ref{sec:blow}.  In particular, he proves the following theorem. 
\begin{restatable}[\cite{jhales}]{theorem}{staircase} \label{thm:staircase}
For $n>0,$ $\CFKi(S^3,K_n) \cong C_n \oplus A$, where $H_*(A)=0$.
\end{restatable}
  See the $k=1$ case of Proposition \ref{prop:class1}   for a more precise statement of Theorem \ref{thm:staircase}. It is clear that Theorem \ref{thm:main1} is the consequence of Theorem \ref{thm:zinfty} and Theorem \ref{thm:staircase}.

\begin{remark}
The family of knots $K_n$ we consider here is the same family used to prove the infinite generation in \cite{homologyconcordance},  where the $d$-invariants  of $1/p$ surgeries on $\mu_{1,1}^{(n)}$ are computed, using the fact that for certain $p$, the resulting manifolds are Brieskorn spheres. Note that this is also similar to the approach employed in \cite{akbulutkarakurt}. Interestingly, note that with the current computation we do not recover the infinite generation result with $\mu_{1,1}^{(n)}$ proved in \cite{homologyconcordance}, as it turns out that $\CFKi(S^3_{-1}(K_n), \mu_{1,1}^{(n)}) $ is locally equivalent to the trivial complex. In order to recover their result, one needs to in addition  consider the nontrivial flip map over $\CFKi(S^3_{-1}(K_n), \mu_{1,1}^{(n)}) $.
\end{remark}

\subsection{Blowing down two bridge links}
We also explore the algorithm due to Ozsv\'{a}th, Szab\'{o}  and Hales. For a rational tangle $[a_1,\cdots,a_\l]$ whose closure is a two-component link, we can always arrange such that $\l$ is odd and each $a_i$ is even when $i$ is odd. (See  Lemma \ref{le:eventwist} and the discussion before that.) Denote by $K^{\pm}([a_1,\cdots,a_\l])$ the knot obtained from blowing down the $\pm 1$-framed upper component. See Figure \ref{fig:intro_rational}. It turns out that we can completely determine the knot Floer homology of $K^{\pm}([a_1,\cdots,a_\l])$ when $\l\leq 3.$ We include this classification result in the paper, with the expectation that these families of knots will generate more future applications.

\begin{figure}[htb!]
\centering
\labellist
  \pinlabel {\small{$a_3$}}  at 130 166
 \pinlabel {\small{$-a_2$}}  at 220 120
  \pinlabel {\small{$a_1$}}  at 320 162
   \pinlabel {$\pm 1$}  at 40 270
\endlabellist
\includegraphics[scale=0.5]{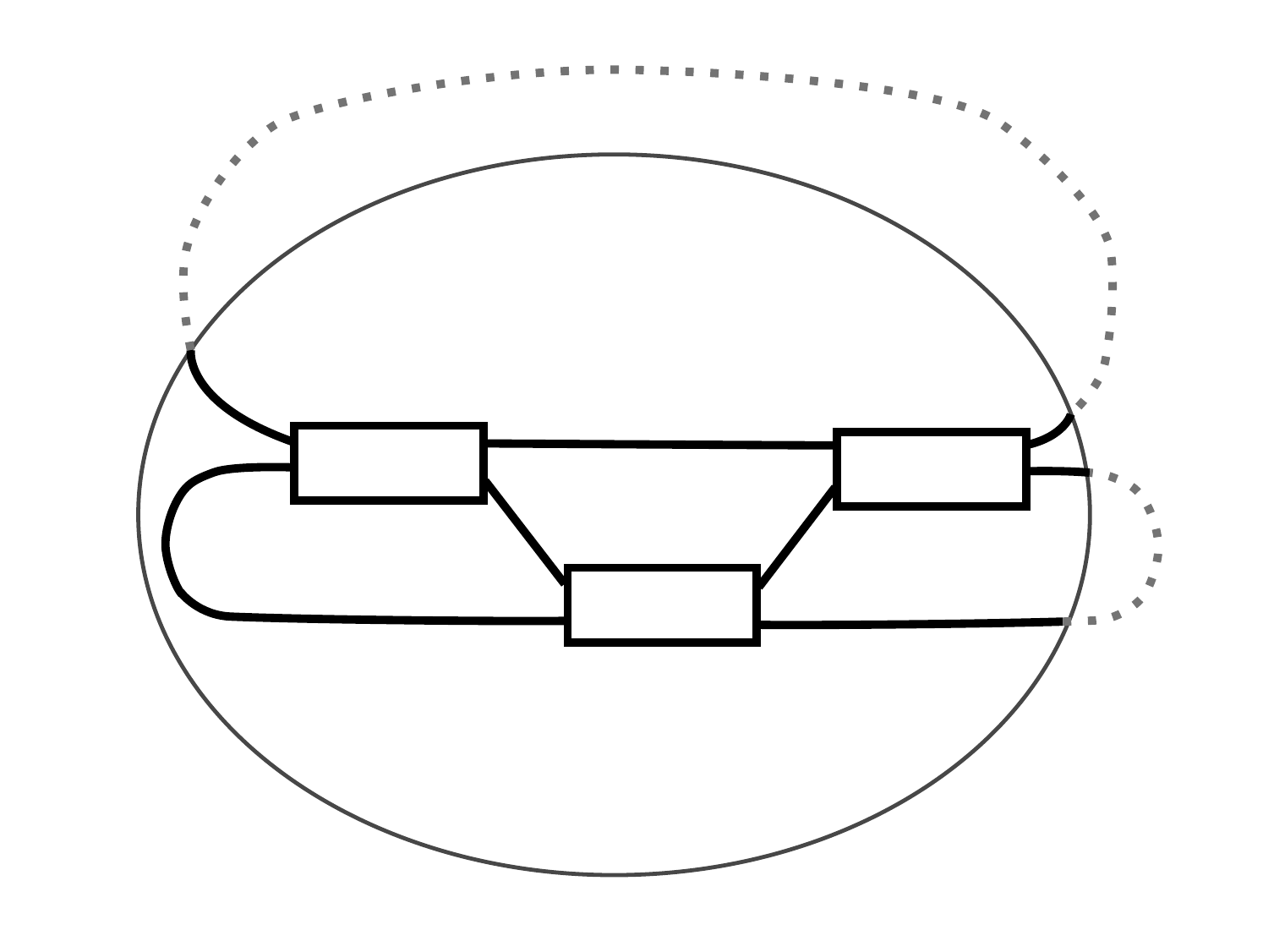}
\caption{Blowing down one component of the closure of a rational tangle $[a_1,a_2,a_3]$ with $a_1$ and $a_3$ even yields the knot $K^{\pm}[a_1,a_2,a_3]$. The numbers in the boxes indicate the number of half-twists. }
\label{fig:intro_rational}
\end{figure}
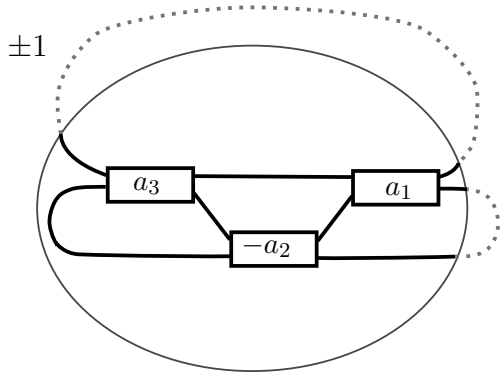

The computation of the invariants has always been a main theme in Heegaard Floer homology. Despite significant development of the theory,  there are still only a limited number of knot families whose knot Floer complexes can be explicitly determined:  L-space knots \cite{OSlens},  thin knots \cite{OSalter} \cite{inathin},   $(1,1)$ almost L-space knots \cite{binns20231} (all determined by the Alexander polynomial) and certain cables of specific knots (using say \cite{hanselman2019cabling}, where the computations are already difficult). 

\begin{theorem}\label{thm:class}
    The knot Floer complex of any knot $K^{\pm}([a_1,\cdots,a_\l])$ with $\l=1,3$ and $a_i$  even for odd $i$ is classified, including Maslov gradings and filtration levels. 
\end{theorem}

The  knot Floer chain homotopy types arise from our examples are novel. For an example, consider the complexes $D_s$ in Definition \ref{def:ds}, the complexes $C_{n,k}$ in Definition \ref{def:cnk} and the complexes $C'_{n,k}$ in Definition \ref{def:c'nk}. Recall $C_n \cong \CFKi(S^3,T_{2,2k+1})$ and let $C_0$ be the complex generated by a single element. We will show that the knot Floer complexes of various knots  $K^{\pm}[n,1,n+k]$ with $n>0$ and $k\geq 0$ consists of direct summands of the above complexes. (For detailed statements see Proposition \ref{prop:class1} and \ref{prop:class2}.) As mentioned earlier, the case for $K^+$ when $k=1$ has already been done by Jonathan Hales. We include it into a framework suitable for a slightly larger family.

Moreover, this classification involves interesting techniques.
For a general sequence $[a_1,b,a_3]$, we observe that changing $b$ by $2$ amounts to performing a full Dehn-twist along an arc $\gamma$. See Figure \ref{fig:halfDehntwist}. This is a local transformation in a neighbourhood of $\gamma$, so we can opt to perform it in the very end. It turns out that each full Dehn-twist amounts to adding a box summand at the ``closest'' generator near $\gamma$. See Figure \ref{fig:sigma2} for the effect on $\CFKi$ and see Definition \ref{def:marked} for a precise definition of the ``closest'' generator. It follows that in order to recover any general  sequence $[a_1,b,a_3]$, we need only consider the case $b= \pm 1$ and $0$ with a \emph{marked basis}. See Proposition \ref{prop:sigma2} for more details.

This method allows us to reduce the full classification to a handful of cases. For an instance, in Section \ref{subsec:class0} we will utilize it to show that the knot Floer complex corresponding to the sequence $[a_1,b,a_3]$ when $b$ is even is isomorphic to $\CFKi(S^3,T_{n,n+1})$ for certain $n$ with a certain number of box summands added to some marked generators.

\begin{remark}
  When $\l >3$, even though a closed formula seems unlikely, using similar methods as in this paper it is still very much  feasible to determine the full knot Floer complex for singular examples of $K^{\pm}([a_1,\cdots, a_\l])$. Moreover, if one opts to compute partial invariants such as the local equivalence class or the $\tau$ invariant,  then I believe the computation should be practical for a larger family of knots, and I expect the result to be interesting as well.   
\end{remark}

\subsection*{Organization} We perform the filtered mapping cone computations  and  prove Theorem \ref{thm:zinfty} in Section \ref{sec:cables}. We review the algorithm by Ozsv\'{a}th, Szab\'{o} and Hales in Section \ref{sec:blow} and prove our classification result in Section \ref{sec:classify}. More detailedly, in Section \ref{subsec:len1}, we prove the length-one case; in Section \ref{subsec:general}, we discuss the length-three case in general and prove some helpful lemmas; in Section \ref{subsec:class0} through \ref{subsec:class4}, we deal with the remaining cases with length-three.
\subsection*{Acknowledgement}
I would like to thank my advisor Jen Hom for encouragement and support. I am grateful to Adam Levine, Fraser Binns and Tye Lidman for helpful conversations. I also want to express my appreciation for the amazing thesis \cite{jhales} by Jonathan Hales. 

\section{Knot Floer homology of the cables of the knot meridian} \label{sec:cables}
In this section we prove Theorem \ref{thm:zinfty}. We will make use of the concordance invariants $\varphi_{i,j}$ defined in \cite{Moreconcor}. We refer the reader to the original paper for the detailed definitions. See \cite[Section 3.2]{Moreconcor} for a worked out example. See \cite[Example 8.1]{zhou2022filtered} for yet another example. Since $\varphi_{i,j}$ factors through the group of knotlike complexes over local equivalence \cite[Proposition 4.42]{Moreconcor}, instead of considering $-1$-surgery on the knot $J_{2k}$ with $\CFKi(S^3,J_{2k}) \cong \CFKi(S^3,T_{4k+1}) \oplus A $, where $H_*(A) = 0,$ it suffices to consider the $-1$-surgery on $T_{4k+1}.$

 Theorem \ref{thm:zinfty} follows from a computational result, Proposition \ref{prop:varphi}. The computational tool is the filtered mapping cone formula \cite{zhou2022filtered}, which combines the work \cite{HL} and \cite{Truong} to give a description of the knot Floer complex of the $(m,1)$-cable of the meridian in the image of a  surgery along a knot. We include a brief review of this filtered mapping cone formula in the following subsection.
\subsection{Preliminaries on the filtered mapping cone formula}
Let $K\subset S^3$ be a knot with genus equal to $g.$ For a given positive integer $m$, let $\mu_{m,1}$ denote the $(m,1)$-cable of the meridian of $K$ in the $-1$-surgery on $K.$
According to \cite[Theorem 1.9]{zhou2022filtered}, the knot Floer complex  $\CFKi(S^3_{-1}(K),\mu_{m,1})$ is filtered chain homotopy equivalent to the doubly-filtered chain complex $X^\infty_m(K)$, defined to be the mapping cone of
\begin{align} \label{eq:x_infinity}
     \bigoplus^{g+m-1}_{s=-g+1}A^\infty_s \xrightarrow{v^\infty_s+h^\infty_s} \bigoplus^{g+m-1}_{s=-g}B^\infty_s,
\end{align}
where each $A^\infty_s$ and $B^\infty_s$ are isomorphic to $\CFKi(S^3,K)$. The map $v^\infty_s\co A_s \to B_s $ is the identity and the map $h^\infty_s\co A_s \to B_{s-1} $ is the reflection along $i=j$ precomposed with $U^s$.  

Let $\II$ and  $\JJ$  be the double filtrations 
on the filtered mapping cone complex $X^\infty_m(K)$. We have
\begin{align}
\intertext{for $[\x,i,j] \in A_{s}$,}
\label{eq: filtration_s3_1}
 \II([\x,i,j]) &= \max\{i,j-s\} \\
 \label{eq: filtration_s3_2}
 \JJ([\x,i,j]) &= \max\{i-m,j-s\} - ms + \frac{m(m+1)}{2} \\
\intertext{and for $[\x,i,j] \in B_{s}$,}
 \label{eq: filtration_s3_3}
 \II([\x,i,j]) &= i \\
  \label{eq: filtration_s3_4}
 \JJ([\x,i,j]) &= i-m  - ms + \frac{m(m+1)}{2}
\end{align}
 It is straightforward to check that for $s<-g+1$, the map $h_s$ induces an isomorphism on the homology; for $s>g+m-1,$ 
the map $v_s(K)$ induces an isomorphism on the homology,  which justifies the truncation of the mapping cone.

The general strategy for computation involves finding a \emph{reduced} basis for $X^\infty_m(K),$ where every term in the differential strictly lowers at least one of the filtrations. This can be achieved through a cancellation process (see for example \cite[Proposition 11.57]{Bordered}) as follows: suppose $\partial x_i = y_i$ $+ $ lower filtration terms, where the double filtration of $y_i$ is the same as $x_i$, then the subcomplex of  $X^\infty_m(K)$ generated by all such $\{x_i,  \partial x_i\}$ is acyclic, and $X^\infty_m(K)$ quotient by this complex is reduced. Alternatively,  one can view the above process as a change of basis, that splits off acyclic summands which individually lie entirely in one double-filtration level.  

\subsection{Computation}
Recall that it suffices to consider the $(m,1)$-cable of the knot meridian in the $-1$-surgery on $T_{4k+1}.$
For $k \geq 1,$ the complex $C_{2k} = \CFKi(S^3,T_{2,4k+1})$ is generated by $a_i$ for $i=1,\cdots,2k$ with coordinate $(0,-2k+2i-1)$ and $b_i$ for $i=1,\cdots,2k+1$ with coordinate $(0,-2k+2i-2)$. The Maslov grading is supported in $b_{2k+1}$ and the differentials are given by
 \[
 \partial a_i = Ub_{i+1} + b_i, \quad \text{for} \hspace{0.5em} i=1,\cdots,2k.
 \]

Via the isomorphism with $\CFKi(S^3,T_{2,4k+1})$, denote the generators in  $A_s$  by $a^{(s)}_i$ for $i=1,\cdots, 2k$ and $b^{(s)}_i$ for $i=1,\cdots, 2k+1$ and the generators in $B_s$ by $a'^{(s)}_i$ and $b'^{(s)}_i$, where $s=-2k+1,\cdots, 2k+m-1.$ 
The differential on the mapping cone is given by
 \begin{align*}
     \partial a^{(s)}_i &= b^{(s)}_i + Ub^{(s)}_{i+1} + a'^{(s)}_i + U^{s+2k+1-2i} (a'^{(s-1)}_{2k-i+1}) \\
     \partial b^{(s)}_i &= b'^{(s)}_i + U^{s+2k+2-2i} (b'^{(s-1)}_{2k-i+2}) \\
     \partial a'^{(s)}_i &= b'^{(s)}_i + Ub'^{(s)}_{i+1}.
 \end{align*}
 We first aim to choose a reduced basis for the complex of $X^\infty_{m} (T_{2,4k+1})$. As a subcomplex of $X^\infty_{m} (T_{2,4k+1})$, each $B_s$ is one dimensional. Indeed, quotienting out $\{a'^{(s)}_i, \d a'^{(s)}_i\}_{1\leq i \leq 2k}$ leaves us with sole generator $b'^{(s)}_{2k+1}$; denote $\beta_s = b'^{(s)}_{2k+1}$ after the change of basis.

 Next, for the complex $A_s$, observe that $a^{(s)}_i$ and $Ub^{(s)}_{i+1}$ are in the same coordinate if $-2k + 2i - 1 \geq s$, and similarly $a^{(s)}_i$ and $b^{(s)}_{i}$ are in the same coordinate if $-2k+2i-1 \leq s-m$.  So we may quotient out 
 $\{a^{(s)}_i, \d a^{(s)}_i \}$ for $i \geq k + \frac{s+1}{2}$ and $i \leq k + \frac{s -m +1}{2}.$
 We have obtained a reduced model of $X^\infty_{m} (T_{2,4k+1})$. As a notational shorthand, let us define
\begin{align*}
    f(m,s) &= \frac{m(m+1)}{2} -ms\\
     \it &=    \min\{ k + \ceil{\frac{s-1}{2}}  , 2k\} \\
     \ib &= \max\{ k +1 + \ceil{\frac{s-m}{2} }, 1 \} 
\end{align*}
such that in the reduced model each $A_s$ is generated by $a^{(s)}_i$ with $\ib \leq i \leq \it$ and  $b^{(s)}_i$ with $\ib \leq i \leq \it +1.$
 The induced differentials on the chain complex are 
 \begin{align}
    \label{eq: diff1} \partial a^{(s)}_i &= b^{(s)}_i + Ub^{(s)}_{i+1} \\
    \label{eq: diff2}  \partial b^{(s)}_i &= U^{2k+1-i}\beta_s + U^{2k+1-i+s} \beta_{s-1}
 \intertext{and the filtration level of the generators are}
     \label{eq: filt1}  \JJ(a^{(s)}_i) &= f(m,s) -2k -s -1 + 2i \\
     \label{eq: filt2} \JJ(b^{(s)}_i) &= f(m,s) -2k -s -2 + 2i  \\
     \label{eq: filt3} \JJ(\beta_s) &= f(m,s+1)\\
     \label{eq: filt4}  \II(a^{(s)}_i) &= \II(b^{(s)}_i) = \II(\beta_s) = 0
 \end{align}
 for $-2k+1 \leq s \leq 2k + m -1$ and suitable $i$ as discussed above.
 For the purpose of computing concordance invariants, we will show the mapping cone can be further truncated. Let
 \begin{align*}
    X^\infty_{m} (T_{2,4k+1}) \langle \ell \rangle = \bigoplus^{\ell}_{s=-\ell+m}A^\infty_s \xrightarrow{v^\infty_s+h^\infty_s} \bigoplus^{\ell}_{s=-\ell+m-1}B^\infty_s.
\end{align*}
Note that under this notation  $X^\infty_{m} (T_{2,4k+1}) = X^\infty_{m} (T_{2,4k+1}) \langle 2k+m-1 \rangle$. 
 \begin{lemma}\label{le: localequi}
 For any $k \geq 1$ and $m\geq 2$, the filtered complex $X^\infty_{m} (T_{2,4k+1}) \langle 2k+m-1 \rangle$  is isomorphic to $ X^\infty_n (T_{2,4k+1}) \langle m-1 \rangle \oplus D $ up to a filtered change of basis, where $H_*(D)=0$.
 \end{lemma} 
\begin{proof}
     We will show that for $m \leq \ell \leq 2k + m -1,$ the complex $X^\infty_{m} (T_{2,4k+1}) \langle \ell \rangle$ is isomorphic to $X^\infty_{m} (T_{2,4k+1}) \langle \ell -1 \rangle \oplus D'$ up to a change of basis, where $H_*(D')=0$. For every such $\ell$, in $X^\infty_{m} (T_{2,4k+1}) \langle \ell \rangle$ perform a change of basis
     \begin{align*}
         \beta_{-\ell + m -1} &\longmapsto \beta_{-\ell + m -1} + U^{\ell -m} \beta_{-\ell + m }\\
         \beta_{\ell} &\longmapsto \beta_{\ell} + U^\ell \beta_{\ell-1}.
\intertext{
      By \eqref{eq: diff2}, as a result the complexes given by
}
     A^\infty_{-\ell+m} &\xrightarrow[]{h_{-\ell+m}}B^\infty_{-\ell+m-1}\\
      A^\infty_{\ell} &\xrightarrow[]{\hspace{0.3em} v_{\ell}\hspace{0.3em}}B^\infty_{\ell}
\end{align*}
 both become summands under the new basis. The change of basis is clearly filtered with respect to the $\II$--filtration. For the $\JJ$--filtration, we compute
 \begin{align*}
     \JJ(\beta_{-\ell + m -1}) - \JJ(U^{\ell -m} \beta_{-\ell + m }) &=  \JJ(\beta_{-\ell + m -1}) - ( \JJ( \beta_{-\ell + m }) - \ell + m )\\
     &=f(m,-\ell + m) - f(m,-\ell + m + 1) + \ell -m\\
     &= \ell \geq 0\\
     \JJ( \beta_{\ell}) - \JJ( U^\ell \beta_{\ell-1}) & =  \JJ( \beta_{\ell}) - \JJ( \beta_{\ell-1}) + \ell\\
     & = f(m, \ell + 1) -  f(m, \ell) + \ell\\
     & =  \ell -m \geq 0.
 \end{align*}
 Therefore the change of basis is filtered.
\end{proof}
We also record the filtration shift between the generators in the complex. 
\begin{definition}\label{def:gradingshift}
    Suppose $U^c y$ is a nontrivial term in $\partial x$, where $c$ is some constant and $x, y$ are both generators. Define 
\begin{equation*}
    \Delta_{\II,\JJ}(x,y) = (\II,\JJ)(x) - (\II,\JJ)(U^c y)
\end{equation*}
and similarly define $\Delta_\II$ and $\Delta_\JJ$.
\end{definition} 
Clearly, $\Delta_{\II,\JJ}(a^{(s)}_i, b^{(s)}_i) = (0, 1)$
   and $\Delta_{\II,\JJ}(a^{(s)}_i, b^{(s)}_{i+1}) = (1, 0)$ for each $\ib \leq i \leq \it$. We also have
\begin{lemma} \label{le: filtrationshift}
For $-2k+1 \leq s \leq 2k + m -1$ and $\ib \leq i \leq \it + 1,$
\begin{align}
    \label{eq:filtrationshift1} \Delta_{\II,\JJ}(b^{(s)}_i,\beta_{s-1}) &= (2k +1 -i +s, i-1)\\
    \label{eq:filtrationshift2} \Delta_{\II,\JJ}(b^{(s)}_i,\beta_{s}) &= (2k +1 -i , m-s+i-1)
\end{align}
\end{lemma}
\begin{proof}
By \eqref{eq: diff2}, \eqref{eq: filt2},\eqref{eq: filt3} and \eqref{eq: filt4}, we compute
    \begin{align*}
         \Delta_{\II,\JJ}(b^{(s)}_i,\beta_{s-1}) &= (\II,\JJ)(b^{(s)}_i) - (\II,\JJ)(U^{2k+1-i+s}\beta_{s-1})\\
         &= (0, f(m,s) - 2k -s -2 + 2i) - (0,f(m,s)) + (2k+1-i+s,2k+1-i+s)\\
         &=(2k +1 -i +s, i-1)\\
         \Delta_{\II,\JJ}(b^{(s)}_i,\beta_{s}) &= (\II,\JJ)(b^{(s)}_i) - (\II,\JJ)(U^{2k+1-i}\beta_{s})\\
         &= (0, f(m,s) - 2k -s -2 + 2i) - (0,f(m,s+1)) + (2k+1-i,2k+1-i)\\
         &=(2k +1 -i , m-s+i-1).
    \end{align*}
\end{proof}
We have all the ingredients to calculate the concordance invariants $\varphi_{i,j}$. This will be done in two steps. First, we translate the complex $X^\infty_n (T_{2,4k+1}) \langle m-1 \rangle$ into the ring $\F[U,V].$ Then we further translate it into the ring $\X$ and perform a change of basis, resulting a standard complex, from which  the invariants $\varphi_{i,j}$ can be readily read off.
\begin{lemma}\label{le:uvcomplex}
    Over the ring $\F[U,V],$ the complex $X^\infty_n (T_{2,4k+1}) \langle m-1 \rangle$ is generated by $\{\beta_s |~ 0\leq s \leq m-1 \} \cup \{ a^{(s)}_{i} |~ 1\leq s \leq m-1, \ib \leq i \leq \it\} \cup \{ b^{(s)}_{i} | ~1\leq s \leq m-1,\ib \leq i \leq \it +1\}$
    with differentials
    \begin{align}
                 \d a^{(s)}_{i}      &= U b^{(s)}_{i+1} + Vb^{(s)}_{i} \\
        \d b^{(s)}_{i} &= U^{\Delta_\II(b^{(s)}_{i},\beta_{s-1})}V^{\Delta_\JJ(b^{(s)}_{i},\beta_{s-1})} \beta_{s-1} + U^{\Delta_\II(b^{(s)}_{i},\beta_{s})}V^{\Delta_\JJ(b^{(s)}_{i},\beta_{s})} \beta_{s}.
    \end{align}
\end{lemma}
\begin{proof}
    This follows trivially from \eqref{eq: diff1}, \eqref{eq: diff2} and the definition of $\Delta_\II$ and $ \Delta_\JJ.$
\end{proof}
We are ready to prove the following proposition. Compare with \cite[Proposition 1.2]{zhou2022filtered}.
\begin{proposition}\label{prop:varphi}
\begin{align*}
    \varphi_{i,j} (X^\infty_n (T_{2,4k+1}) \langle m-1 \rangle) &= \begin{cases}
        -\sum^{s-1}_{1} (\it -\ib + 1)      &\hspace{1em}  (i,j)=(1,0)\\
        1                             &\hspace{1em}  (i,j)\in \{\Delta_{k,m}(s) ~|~1\leq s \leq m-1 \}\\
        0                     &\hspace{1em} \text{otherwise}
    \end{cases} 
    \intertext{where $\Delta_{k,m}(s)$ is given by}
    \Delta_{k,m}(s) &= \begin{cases}
        (k - \ceil{\frac{s-1}{2}}, m + k - \floor{\frac{s+1}{2}})     &\hspace{2em}  s\leq 2k\\
        (0, 2k + m -s)  &\hspace{2em}  s\geq 2k.
    \end{cases}
\end{align*}
\end{proposition}
\begin{proof}
Continuing from Lemma \ref{le:uvcomplex}, we can further translate the complex  $X^\infty_n (T_{2,4k+1}) \langle m-1 \rangle$ into the ring 
\begin{align*}
    \X&=\frac{\F[U_B,\{W_{B,i}\}_{i\in \Z},V_T,\{W_{T,i}\}_{i\in \Z}]}{(U_B V_T, U_B W_{B,i}-W_{B,i+1},V_T W_{T,i}-W_{T,i+1})}
    \intertext{using the maps}
     U&\xmapsto[\hspace{1em}]{} U_B + W_{T,0}\hspace{5em}
    V\xmapsto[\hspace{1em}]{} V_T + W_{B,0}.
\end{align*}
The differentials now reads 
\begin{align*}
    \d a^{(s)}_{i}      &= (U_B + W_{T,0}) b^{(s)}_{i+1} + ( W_{B,0} + V_T )b^{(s)}_{i} \\
     \d b^{(s)}_{i} &= (U_{B}^{\Delta_\II(b^{(s)}_{i},\beta_{s-1})}W_{B,0}^{\Delta_\JJ(b^{(s)}_{i},\beta_{s-1})} + V_T^{\Delta_\JJ(b^{(s)}_{i},\beta_{s-1})}W_{B,0}^{\Delta_\II(b^{(s)}_{i},\beta_{s-1})})  \beta_{s-1} \\
     &+ (U_B^{\Delta_\II(b^{(s)}_{i},\beta_{s})}W_{B,0}^{\Delta_\JJ(b^{(s)}_{i},\beta_{s})} + V_T^{\Delta_\JJ(b^{(s)}_{i},\beta_{s})}W_{T,0}^{\Delta_\II(b^{(s)}_{i},\beta_{s})} ) \beta_{s}.
\end{align*}
Note that each term in the previous coefficient in $\F[U,V]$ becomes two terms in the above expression, one in the ideal $(U_B)$ and one in the  the ideal $(V_T)$.
We perform the change of basis
\begin{align}
   \label{eq: changebasis1} \widetilde{b}^{(s)}_i &= \begin{cases}
      b^{(s)}_{i} + U^{-1}_B W_{B,0} ( b^{(s)}_{i-1} ) &\quad \text{if} \hspace{0.5em} i=i^{(t)}_{m,s} + 1\\
     b^{(s)}_{i} + U^{-1}_B W_{B,0} (b^{(s)}_{i-1}) + V^{-1}_T W_{T,0}  (b^{(s)}_{i+1}) &\quad \text{if} \hspace{0.5em}  i^{(b)}_{m,s}<i<i^{(t)}_{m,s} + 1\\
     b^{(s)}_{i} + V^{-1}_T W_{T,0} (b^{(s)}_{i+1})  &\quad \text{if} \hspace{0.5em} i=i^{(b)}_{m,s},
     \end{cases}\\
    \label{eq: changebasis2} \bb_s &= \begin{cases}
     \beta_s + V_T^{-s}W_{T,0}^{m-s}\beta_{s-1}    & \hspace{1.3em} \text{if}  \hspace{0.5em} s=0\\
     \beta_s + V_T^{-s}W_{T,0}^{m-s}\beta_{s-1}   + U_B^{s-m}W_{B,0}^{s}\beta_{s+1} & \hspace{1.3em} \text{if}  \hspace{0.5em} 1\leq s \leq m-2 \\
      \beta_s + U_B^{s-m}W_{B,0}^{s}\beta_{s+1}    & \hspace{1.3em} \text{if}  \hspace{0.5em} s=m-1
    \end{cases}
    \intertext{which simplifies the differentials to}
   \label{eq: finaldiff1} \d a^{(s)}_{i}      &= U_B  \widetilde{b}^{(s)}_{i+1} +  V_T  \widetilde{b}^{(s)}_i \\
    \label{eq: finaldiff2} \d  \widetilde{b}^{(s)}_i &=  \begin{cases}
   V_T^{\Delta_\JJ(b^{(s)}_{i},\beta_{s})}W_{T,0}^{\Delta_\II(b^{(s)}_{i},\beta_{s})} \bb_{s}, &\hspace{0.5em} i=\it + 1 \\
     U_B^{\Delta_\II(b^{(s)}_{i},\beta_{s-1})}W_{B,0}^{\Delta_\JJ(b^{(s)}_{i},\beta_{s-1})}  \bb_{s-1}, &\hspace{0.5em} i=\ib\\
    0,   &\text{ otherwise}.
         \end{cases}
\end{align}
Note that $\Delta_{\II,\JJ}(b^{(s)}_i,\beta_{s-1}) - \Delta_{\II,\JJ}(b^{(s)}_i,\beta_{s}) = (s,s-m)$ by \eqref{eq:filtrationshift1} and \eqref{eq:filtrationshift2}  which justifies the change of basis \eqref{eq: changebasis2}. We have obtained a standard complex \cite[Definition 5.1]{Moreconcor}, where the sequence of vector length of each $V$--edges (starting from $\bb_0$) is as follows 
\begin{align*}
\Big( \underbrace{\overbrace{-(1,0),\cdots,-(1,0)}^{(\it -\ib + 1) \text{ copies}}, \Delta_{\II,\JJ}(b^{(s)}_{\it + 1}, \beta_s),}_{1\leq s \leq m-1} \cdots \Big). 
\end{align*}
To finish the proof, it suffices to show $\Delta_{k,m}(s)  = \Delta_{\II,\JJ}(b^{(s)}_{\it + 1}, \beta_s)$ for $1\leq s \leq m-1.$ Recall that $\it =    \min\{ k + \ceil{\frac{s-1}{2}}  , 2k\}.$ Therefore
\begin{align*}
    \it + 1 &= \begin{cases}
        k + \ceil{\frac{s+1}{2}}  &\hspace{3em}  s \leq 2k\\
        2k+1 &\hspace{3em}  s \geq 2k.
    \end{cases}
    \intertext{By \eqref{eq:filtrationshift2}, we compute}
    \Delta_{\II,\JJ}(b^{(s)}_{\it + 1}, \beta_s) &= (2k+1 -k -  \ceil{\frac{s+1}{2}}, m-s+k + \ceil{\frac{s+1}{2}} -1  )\\
    &=(k - \ceil{\frac{s-1}{2}}, m + k - \floor{\frac{s+1}{2}})   \hspace{3em}  \text{if } s\leq 2k;\\
     \Delta_{\II,\JJ}(b^{(s)}_{\it + 1}, \beta_s) &= (2k+1 -(2k+1), m-s+2k +1 -1) \\
     &= (0,2k + m -s)    \hspace{10em}  \text{if } s\geq 2k.
\end{align*}   
\end{proof}
 
\begin{lemma} \label{le:varphij2k}
    Suppose $\CFKi(S^3,J_{2k}) \cong \CFKi(S^3,T_{4k+1}) \oplus A$ for $k>0,$ where $H_*(A) = 0.$ Let $\mu^{(2k)}_{m,1}$ denote the image of the $(m,1)$-cable of the median in the $-1$-surgery on $J_{2k}.$ Then we have
    \begin{align*}
        \varphi_{i,j}(S^3_{-1}(J_{2k}),\mu^{(2k)}_{m,1}) = \begin{cases}
            1  &\hspace{3em} (i,j) = (k, m+k-1)\\
            0   &\hspace{3em} i>k \quad \text{or } \quad j>m+k-1.
        \end{cases}
    \end{align*}
\end{lemma}
\begin{proof}
    Since $\CFKi(S^3,J_{2k}) \cong \CFKi(S^3,T_{4k+1}) \oplus A$, by the filtered mapping cone formula   $\CFKi(S^3_{-1}(J_{2k}),\mu^{(2k)}_{m,1}) \cong \CFKi(S^3_{-1}(T_{4k+1}),\mu'^{(2k)}_{m,1}) \oplus A'$. By Lemma \ref{le: localequi}
    \[
    \varphi_{i,j}(S^3_{-1}(J_{2k}),\mu^{(2k)}_{m,1}) =   \varphi_{i,j} (X^\infty_n (T_{2,4k+1}) \langle m-1 \rangle).
    \]
    According to Proposition \ref{prop:varphi}, the value of $\varphi_{i,j}$ is determined by the sequence $\Delta_{k,m}(s)$ for $1\leq s \leq m-1.$
    Note that $\Delta_{k,m}(1) = (k, m+k-1),$ and as $s$ increases by $1$ either the first or the second entry of $\Delta_{k,m}(s)$ decreases by $1$. The result immediately follows.
\end{proof}
\begin{proof}[Proof of Theorem \ref{thm:zinfty}]
Lemma \ref{le:varphij2k} shows that the $\varphi_{k,m+k-1}$ are linear independent, and
    the homomorphism
    \begin{align*}
    \bigoplus_{k>0,m>1} \varphi_{k,m+k-1} \colon \CZhat/\CZ &\longrightarrow \bigoplus_{k>0,m>1} \Z
   \intertext{
    is surjective. In particular, for each fixed $k>0$ the homomorphism 
    }
    \bigoplus_{m>1} \varphi_{k,m+k-1} \colon \CZhat/\CZ &\longrightarrow \bigoplus_{m>1} \Z
     \end{align*}
    is surjective.
\end{proof}

\section{Blowing down two-bridge links}\label{sec:blow}
We now turn to the algorithm due to Ozsv\'{a}th, Szab\'{o}  and Hales.
Let $R=R_1\cup R_2$ be a two-bridge link, where $R_1$ and $R_2$ are the two link components. One can view $R$ as two arcs $A_1$ and $A_2$ on a fixed $S^2$ with two trivial over-arcs $C_1$ and $C_2$ each intersecting  $S^2$  transversely at end points, such that $R_i = A_i \cup C_i.$ Since $R_2$ is unknotted,  $\epsilon$-framed surgery on  $R_2$ with $\epsilon = \pm 1$  recovers $S^3$. Let $\Sigma$ be the genus one surface obtained from $S^2$ by attaching an one-handle along the arc $C_2,$ and let $\mu_2$ be the meridian of the one-handle. Ozsv\'{a}th-Szab\'{o} gives a description for a genus-one doubly pointed Heegaard diagram for the knot $R_1$ in $S^3_\epsilon(R_2) =S^3$  as follows.
\begin{proposition}{\cite[Proposition 6.3]{OSknot}}
For $\epsilon = \pm 1, $ let $\alpha = C_2 \cup ($trivial arc in $S^2)$, $\beta = A_2 \cup C_2 \cup  k\mu_2$, where $k$ is chosen such that $\alpha \cdot \beta = \epsilon,$ and put $z$ and $w$ basepoints at the two end points of $C_1.$ Then $(\Sigma, \alpha, \beta, z, w)$ represents $(S^3_{\epsilon}(R_2),R_1).$ 
\end{proposition}
\begin{proof}
Clearly $(\Sigma, \alpha, \beta)$ represents $S^3_{\epsilon}(R_2).$ Connecting $z$ to $w$ in the complement of $\alpha$ traces out $C_1$ while connecting $w$ to $z$ in the complement of $\beta$ traces out $A_1$.
\end{proof}
In particular, it follows that $(S^3_{\epsilon}(R_2),R_1)$ is a $(1,1)$ knot. Next, view the two-bridge link $R$ as the numerator closure of a rational tangle $p/q$, with $p$ even and $q$ odd. We follow the convention in \cite{rationaltangle}. Every rational tangle can be obtained from the trivial tangle by performing
\begin{itemize}
    \item the vertical right-handed half-twist $\tau$, and
    \item the horizontal right-handed half-twist $\sigma.$
\end{itemize}
 Specifically in the case when the numerator closure is a link, $\tau^2$ and $\sigma$ suffice to generate the rational tangle. \footnote{In fact $\tau^2$ and $\sigma^2$ suffice, but at the expense of increasing the length of the continued fraction. For our purpose we will use $\tau^2$ and $\sigma$.} This can be seen from a simple lemma regarding continued fractions, as follows. For integers $a_1, \cdots, a_\ell$, denote the continued fraction by  $[a_1,\cdots, a_\ell]$ as follows:
\[
[a_1,\cdots, a_\ell]=a_1+\cfrac{1}{a_2+\cfrac{1}{\cdots+\cfrac{1}{ a_{\ell-1}+\cfrac{1}{a_\ell}}}}
\]

\begin{lemma} \label{le:eventwist}
If $p$ is even and $q$ is odd, we can arrange such that $p/q = [a_1, \cdots, a_\ell]$ where $\ell$ is odd, and  $a_i$ is even when $i$ is odd. 
\end{lemma}
\begin{proof}
Set $(p_0,q_0)=(p,q),$  and for each $i\geq 1,$  we will recursively choose integers $(a_i,p_i,q_i)$ with $p_i$ and $q_i$ coprime, such that
 \begin{align}\label{eq:recursive}
         \frac{p_{i-1}}{q_{i-1}}=a_i + \frac{q_i}{p_i}.
    \end{align} 
Since ${p_{i-1}}/{q_{i-1}}$ is between two consecutive integers, we  choose $a_i$ simply to be the closest even integer (resp.~the closest integer) when $i$ is odd (resp.~when $i$ is even) and choose $p_i,q_i$ according to \eqref{eq:recursive}. It follows that $|p_i| = |q_{i-1}|$ and
 \begin{align*}
        p_{i-1} = a_i q_{i-1} + q_{i} \qquad \text{mod } 2.
    \end{align*}
    From these one can inductively show that for all $0\leq i \leq n$ we have
 \begin{align*}
        (p_i,q_i) &= \begin{cases}
         (\text{even},\text{odd}) \hspace{2em}   & i \hspace{0.5em} \text{even}\\
         (\text{odd},\text{even})  & i \hspace{0.5em} \text{odd} .
      \end{cases}
    \end{align*}
    Note that $|q_i| < |p_i| = |q_{i-1}|$, therefore this process will terminate in finite number of steps and
    that the last term $a_n$ must have $n$ odd, since $q_n=0$ which is an even number.
\end{proof}
\begin{figure}
\begin{minipage}{.5\linewidth}
\subfloat[]{
\centering
\labellist
  \pinlabel {\small{$a_3$}}  at 100 222
 \pinlabel \rotatebox{-90}{\small{$-a_2$}}  at 100 135
  \pinlabel {\small{$a_1$}}  at 210 152
\endlabellist
\includegraphics[scale=0.5]{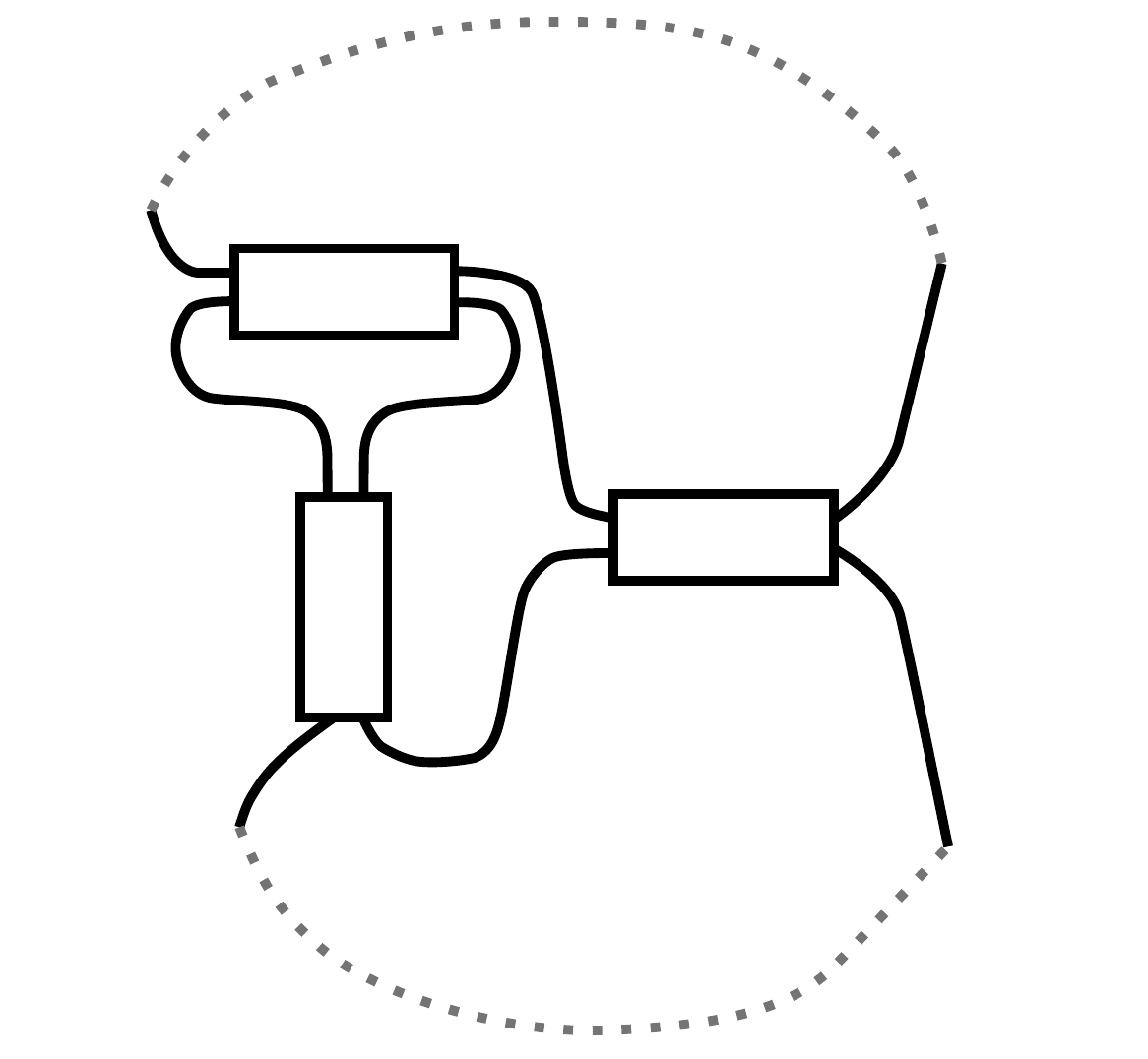}
}
\end{minipage}%
\begin{minipage}{.5\linewidth}
\subfloat[]{
\centering
\labellist
  \pinlabel {\small{$a_3$}}  at 130 166
 \pinlabel {\small{$-a_2$}}  at 220 120
  \pinlabel {\small{$a_1$}}  at 320 162
   \pinlabel {$\pm 1$}  at 40 270
\endlabellist
\includegraphics[scale=0.5]{rationaltangle2}\label{subfig:rationaltangle2}
}
\end{minipage}
\caption{On the left, a rational tangle corresponding to $p/q = [a_1,a_2,a_3]$, with the numerator closure indicated by the dotted arcs.  On the right, for even $a_1$ and $a_3$, the knot $K^{\pm}[a_1,a_2,a_3]$. The numbers in the boxes indicate the number of half-twists. }
\label{fig:rationaltangle}
\end{figure}

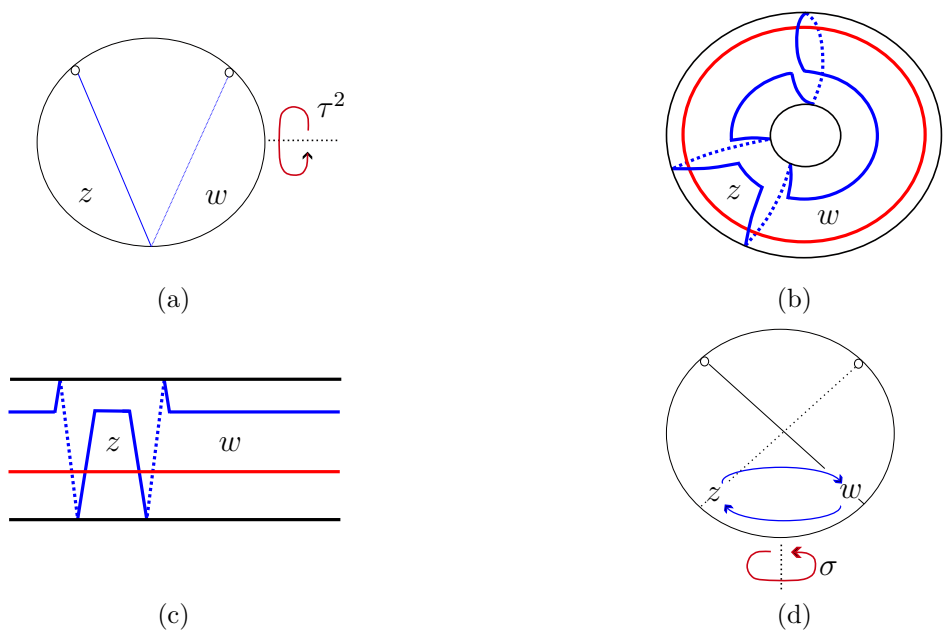
\begin{figure}
\begin{minipage}{.5\linewidth}
\centering
\subfloat[]{
\labellist
 \pinlabel { $z$}  at 46 35
 \pinlabel { $w$}  at 104 34
 \pinlabel { $\tau^2$}  at 155 75
\endlabellist
\includegraphics[scale=0.85]{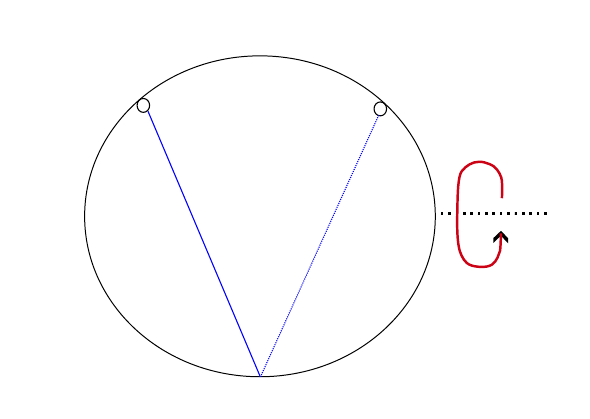}\label{subfig:vt}
}
\end{minipage}%
\begin{minipage}{.5\linewidth}
\centering
\subfloat[]{
\labellist
 \pinlabel { $z$}  at 58 36
 \pinlabel { $w$}  at 100 25
\endlabellist
\includegraphics[scale=0.85]{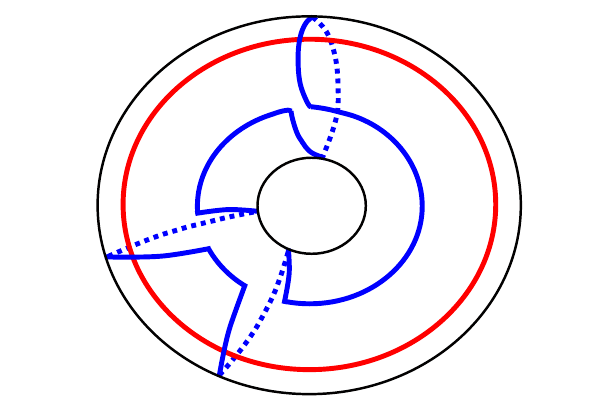}\label{subfig:torus1}
}
\end{minipage}\\
\begin{minipage}{.5\linewidth}
\centering
\subfloat[]{
\labellist
 \pinlabel { $z$}  at 58 65
 \pinlabel { $w$}  at 110 65
\endlabellist
\includegraphics[scale=0.85]{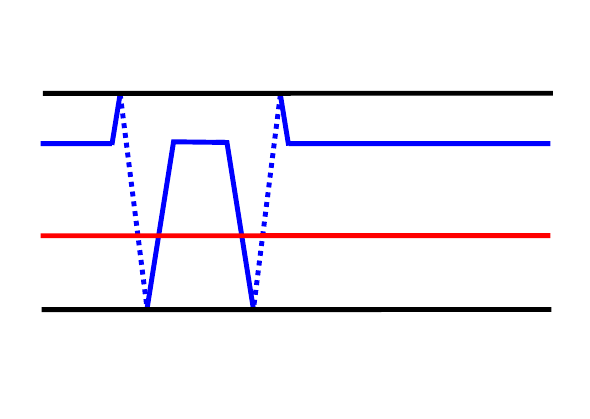}\label{subfig:torus2}
}
\end{minipage}%
\begin{minipage}{.5\linewidth}
\centering
\subfloat[]{
\labellist
 \pinlabel { $z$}  at 49 43
 \pinlabel { $w$}  at 110 44
  \pinlabel { $\sigma$}  at 100 10
\endlabellist
\includegraphics[scale=0.85]{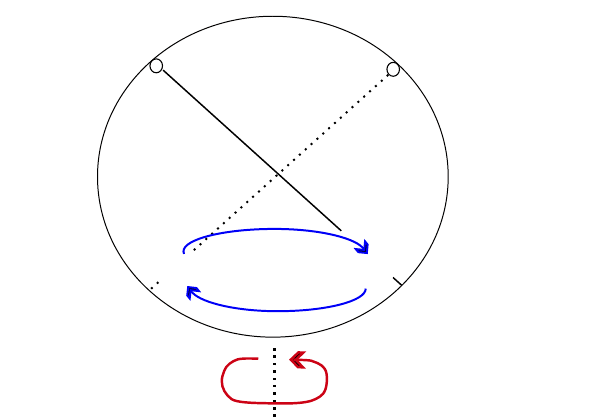}\label{subfig:ht}
}
\end{minipage}
\caption{Figure \protect\subref{subfig:vt} and \protect\subref{subfig:torus1} demonstrate the result of applying  $\tau^2$ on the four-punctured sphere and $\Sigma= S^3 \cup ($one-handle$)$ supposing $\epsilon = +1$, respectively. In Figure \protect\subref{subfig:ht} the two black arcs illustrate the effect of $\sigma.$}
\label{fig:actions}
\end{figure}
\begin{definition}
    Denote by $K^{\pm}([a_1, \cdots, a_\ell])$ the knot obtained by doing $\pm 1$ surgery on the upper strand of the numerator closure of tangle $p/q = [a_1, \cdots, a_\ell]$, respectively.  One can also write $K^{\pm}(p/q)$. See Figure \ref{fig:rationaltangle} \subref{subfig:rationaltangle2}.
\end{definition}

\subsection{A diagrammatic algorithm} \label{subsec:algo}
The following algorithm is due to Jonathan Hales \cite{jhales}.
Let us consider the  effect of action $\tau^2$ and $\sigma$ on the Heegaard diagram $(\Sigma,\alpha, \beta, z, w)$ inside $\epsilon=\pm 1$ surgery. We depict the result of these actions in Figure \ref{fig:actions}. Note that in Figure \ref{fig:actions} \subref{subfig:torus1} and \ref{fig:actions} \subref{subfig:torus2}, in order to preserve the framing of $\beta$ curve, for each added full twist we need to ``double-back" one time such that the intersection points cancel in pairs with sign, and the framing of the $\beta$ curve remains to be $\epsilon$. 

In order to compute the Heegaard Floer homology of a $(1,1)$ knot, the standard treatment is to lift the genus-one Heegaard diagram $(\Sigma,\alpha,\beta,z,w)$ to the universal cover $\R^2$, where  the bigon counts are explicit. Therefore we would like to study the actions of $\tau^{2n}$ and $\sigma$ on the universal cover of $\Sigma$. The following is interpreted from  \cite[Lemma 3.2.3]{jhales}:

For $n\in \Z_{>0}$,  $\tau^{2n}$ has the following effect: in the covering $S^1 \times \R^2$, arrange such that the lifts of the $z$ basepoint lie on a vertical line $\ell_z := \{3/4\}\times \R^2$ (similarly can define $\ell_w$) and suppose $\tilde{\beta}$ is any lift of $\beta$; at each intersection $\tilde{\beta} \cap \ell_z,$ perform an upwards finger move on   $\tilde{\beta}$ to include $n$ more lifts of the $z$ basepoints on $\ell_z$. Equivalently, $\tau^{2n}$ can also be interpreted as performing a downwards finger move at $\ell_w$ to include $n$ lifts of the $w$ basepoints. Parallelly, for $n\in \Z_{>0}$,  $\tau^{-2n}$ has the effect of a downwards finger move to include $n$ lifts of $z$ basepoints or an upwards finger move to include $n$ lifts of $w$ basepoints.  Further lift this to the universal cover $\R^2.$ For an example, see Figure \ref{fig:examplelen1}  in Example \ref{ex:1} and Figure \ref{fig:11diagram}.

On the other hand, note that $\sigma$ is a local action. See Figure \ref{fig:actions} \subref{subfig:ht}. In the universal cover, $\sigma$~(resp. $\sigma^{-1}$) corresponds to performing a clockwise~(resp. counterclockwise) half-Dehn twist around each lift of $z$ and $w$ basepoints. See the middle step of Figure \ref{fig:11diagram}. As a convention in this paper, after each action of $\sigma$ we also switch the role of $z$ and $w$ basepoints , such that $z$ is on the left and $w$ is on the right. The switching induces a chain homotopy equivalence of the resulting Heegaard Floer complex. We include the switching such that the $z$ basepoint is consistently on the left and $w$ basepoint is consistently on the right in each lift. 

According to Lemma \ref{le:eventwist}, each rational tangle whose numerator closure is a two-component link can be obtained from the trivial tangle by applying  $\tau^2$ and $\sigma$ iteratively. Therefore we have described an algorithm that produces the $(1,1)$ diagram of any knot that arises from blowing down a two-bridge link. Here by a $(1,1)$ diagram, we mean the universal cover of a genus-one doubly pointed Heegaard diagram, where we fix a preferred parametrization. We have outlined a proof for the following theorem:
\begin{theorem}[\cite{jhales}]
The actions  $\tau^2$ and $\sigma$  provide an explicit description of the $(1,1)$ diagram of any knot that arises from blowing down a two-bridge link.
\end{theorem}

\section{Classifying $K^{\pm}(p/q)$ with continued fraction length $\leq 3$} \label{sec:classify}
 Using the algorithm due to Ozsv\'{a}th, Szab\'{o} and Hales, the goal of this section is to give a complete classification of $\CFKi(S^3,K^{\pm}(p/q))$ for  $p/q = [a_1,\cdots,a_\ell]$ where $\ell= 1$ or $3$ and $a_i$ is even for each odd $i$. Note that 
\begin{align}
     K^{-}(p/q) = -K^{+}(-p/q) = - K^{+}([-a_1, \cdots, -a_\ell ]).\label{eq:mirror}
\end{align}
As a road map for this section, we will first consider the case $\l=1$ in the next subsection. Then, in Section \ref{subsec:general}, we will discuss some general facts about the $\l=3$ case, reducing it to five subcases. Finally we will prove each one of these subcases in Section \ref{subsec:class0} through \ref{subsec:class4}, completing the classification. All five subcases and their corresponding conclusions are recorded in Proposition \ref{prop:summary} for the reader's convenience.
\subsection{Case $K^{\pm}([a_1])$} \label{subsec:len1}
The following proposition together with \eqref{eq:mirror} classify all the $K^{\pm}([a_1])$.
\begin{proposition}\label{prop:length1}
    For $n>0$
    \begin{align}
      \label{eq:lone1}  K^{+}([2n]) &= - T_{n-1,n}\\
        K^{+}([-2n]) &= -T_{n,n+1}
    \end{align}
    and $K^{+}([0])$ is the unknot. 
\end{proposition}
\begin{proof}
    We will only prove \eqref{eq:lone1}. The other case is similar and left for the reader. Recall that the torus knot $-T_{n,n-1}$  is the braid closure of $(\w_{n-1}\cdots\w_1)^{n-1}$ where $\w_i$ is the braid group element that exchanges the $i$-th and $(i+1)$-th strand, with the crossing convention given by Figure\ref{fig:braid}\subref{subfig:braid2}. Note that a left handed full twist of the first $k$ strands has a presentation of
    $(\w_{k-1}\cdots\w_1)^{k}.$
    
    As depicted in Figure\ref{fig:braid}\subref{subfig:braid3}, the knot $K^{+}([2n])$ is the braid closure of $\w_1\w_2\cdots \w_{n-1} (\w_{n-2}\cdots \w_1)^{n-1}.$ Therefore it suffices to show that as elements of the braid group 
    \begin{align}
        (\w_{n-1}\cdots\w_1)^{n-1} &= \w_1\w_2\cdots \w_{n-1} (\w_{n-2}\cdots \w_1)^{n-1}.\label{eq:toprove}
        \intertext{From the braid relation}
        \w_i \w_{i+1} \w_i &= \w_{i+1} \w_i \w_{i+1},
        \intertext{it is straightforward to see that for $1\leq k<i\leq k+j,$}
        \w_i (\w_{k}\w_{k+1} \cdots \w_{k+j}) &= (\w_{k}\w_{k+1} \cdots \w_{k+j}) \w_{i-1}.\label{eq:relation}
    \end{align}
    We proceed by induction. For each $1\leq i\leq n-1$, we claim that
     \begin{align}
     (\w_{n-1}\cdots \w_1)^{n-1} &= (\w_{n-1}\cdots \w_1)^{n-1-i}(\w_{n-i}\cdots\w_{n-1})(\w_{n-2}\cdots \w_1)^{i}.\label{eq:inductionhypothesis}
        \end{align}
  This is obviously true for $i=1.$ Suppose this is true for some $1\leq i \leq n-2$, by using \eqref{eq:relation}, we have
    \begin{align*}
        (\w_{n-1}\cdots \w_1)^{n-1} &= (\w_{n-1}\cdots \w_1)^{n-1-i}(\w_{n-i}\cdots\w_{n-1})(\w_{n-2}\cdots \w_1)^{i}\\
        &=(\w_{n-1}\cdots \w_1)^{n-2-i} \w_{n-1}\cdots\w_{n-i}(\w_{n-i-1}\w_{n-i}\cdots\w_{n-1})\w_{n-i-2}\cdots\w_1 (\w_{n-2}\cdots \w_1)^{i}\\
         &=(\w_{n-1}\cdots \w_1)^{n-2-i}(\w_{n-i-1}\w_{n-i}\cdots\w_{n-1})(\w_{n-2}\cdots\w_{n-i-1}) \w_{n-i-2}\cdots\w_1 (\w_{n-2}\cdots \w_1)^{i}\\
         &= (\w_{n-1}\cdots \w_1)^{n-1-(i+1)}(\w_{n-i-1}\cdots\w_{n-1})(\w_{n-2}\cdots \w_1)^{i+1}
    \end{align*}
Thus we have proved \eqref{eq:inductionhypothesis}. Taking $i = n-1$ in \eqref{eq:inductionhypothesis} yields \eqref{eq:toprove}.
\end{proof}

\begin{figure}[htb!]
\begin{minipage}{1\linewidth}
\centering
\subfloat[The knot $K^{+}({[2n]})$.]{
\begin{tikzpicture}
\begin{scope}[thin, black!0!white]
          \draw  (-5, 0) -- (5, 0);
      \end{scope}
    \node at (0,0){\includegraphics[scale=0.3]{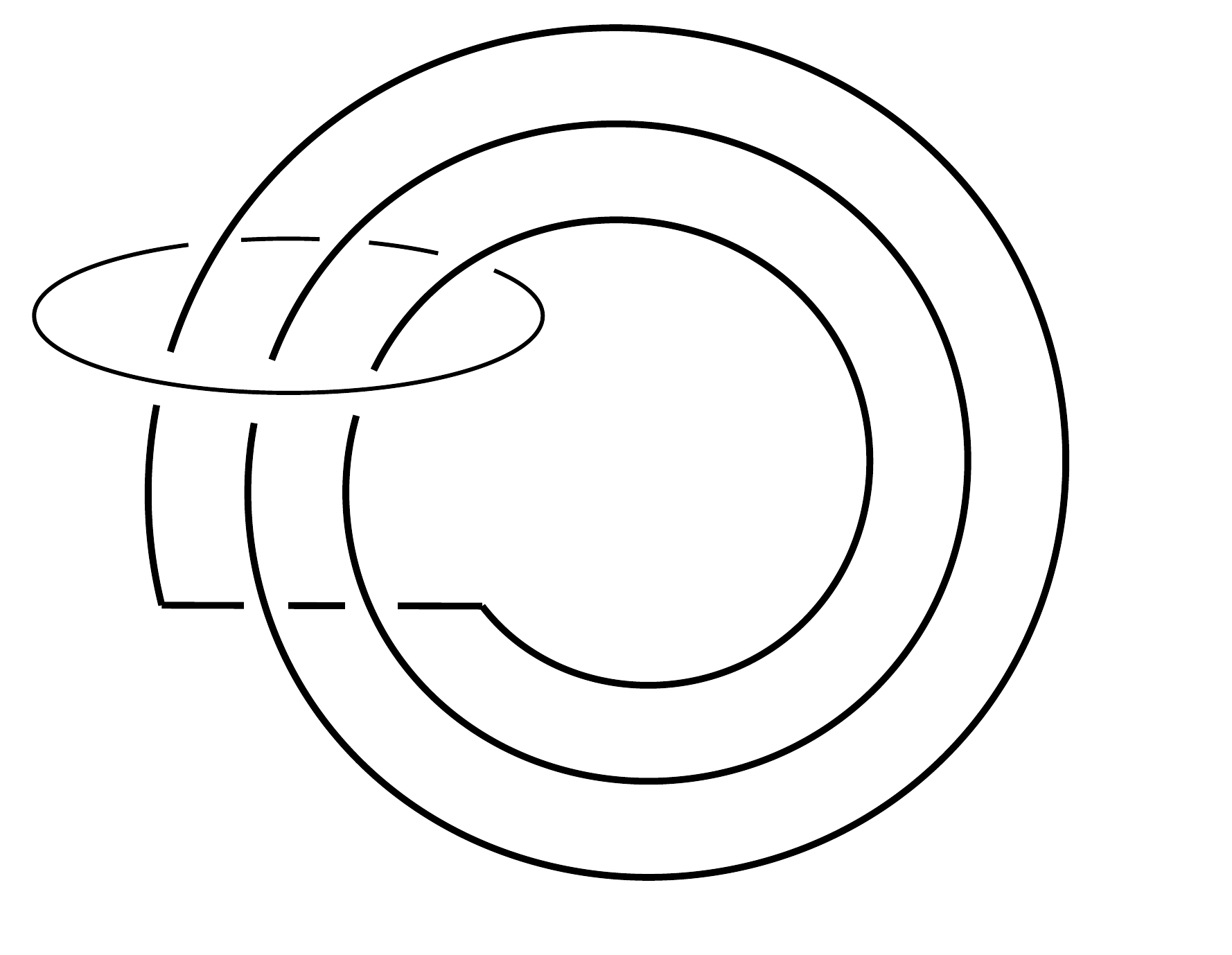}\label{subfig:braid1}};
     \draw [decorate,decoration={brace,amplitude=4pt},xshift=0.5cm,yshift=0pt]
      (0,2) -- (0,1) node [midway,right,xshift=0.1cm] {$n$-strands};
      \node at (-2.5,1.2){$+1$};
\end{tikzpicture}
}
\end{minipage}\\
\centering
\begin{minipage}{.5\linewidth}
\centering
\subfloat[The braid group element $\w_i$.]{
\centering
\begin{tikzpicture}
\begin{scope}[thin, black!0!white]
          \draw  (-4, 0) -- (4, 0);
      \end{scope}
    \node at (0,0){\includegraphics[scale=0.6]{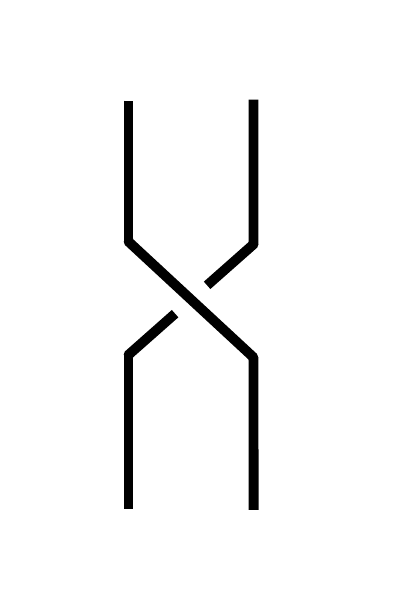}\label{subfig:braid2}};
 \node at (-0.5,1.5){$i$};
      \node at (0.5,1.5){$i+1$};
\end{tikzpicture}
}
\end{minipage}%
\begin{minipage}{.5\linewidth}
\centering
\subfloat[A braid presentation of $K^{+}({[2n]})$.]{
\centering
\begin{tikzpicture}
\begin{scope}[thin, black!0!white]
          \draw  (-5, 0) -- (5, 0);
      \end{scope}
    \node at (0,0){\includegraphics[scale=0.5]{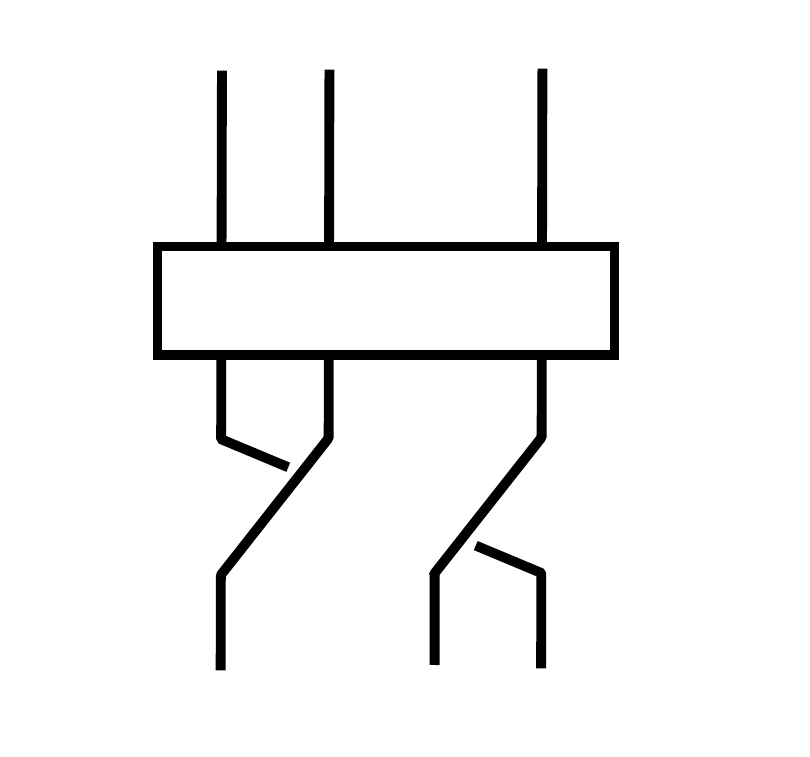}\label{subfig:braid3}};
     \draw [decorate,decoration={brace,amplitude=4pt},xshift=0cm,yshift=0pt]
      (-1,1.6) -- (0.8,1.6) node [midway,above,yshift=0.1cm] {$n$-strands};
      \node at (-0.1,0.4){$-1$};
      \node at (0.2,1){$\cdots$};
       \node at (-0.25,-0.8){$\cdots$};
\end{tikzpicture}

}
\end{minipage}
\caption{}
\label{fig:braid}
\end{figure}

\begin{example}\label{ex:1}
    Consider $K^{+}([-2n])$ for $n>0.$ Applying the algorithm by Ozsv\'{a}th, Szab\'{o} and Hales, starting from a $(1,1)$ diagram for the unknot $K^{+}([0])$, where the $\beta$ curve has slope $-1$, we equivariantly perform downwards finger moves to include $n$ copies of the $z$ basepoint.   The resulting $(1,1)$ diagram is  depicted in Figure \ref{fig:examplelen1}. For a chosen lift $\aa$ of $\alpha$, we mark the intersection points of $\aa\cap \bb$ from left to right by $x_1, x_2, \cdots$  in order. There are $2n-1$ intersection points in total, and for each $ i \in \{1,\cdots, n-1\}$, there is a bigon from $x_{2i-1}$ to $x_{2i}$ with $i$ copies of $z$ and  a bigon from $x_{2i+1}$ to $x_{2i}$ with $n-i$ copies of $w$; there are no other bigons. Therefore we conclude that  $\CFKi(S^3,K^{+}([-2n]))$ is generated by $x_1, \cdots x_{2n - 1}$ with the differentials
    \begin{align*}
         \d x_{2i-1} &=  x_{2i-2} + x_{2i}  \hspace{5em} \ 
         \intertext{and the filtration shifts}
         \dij (x_{2i-1},x_{2i-2}) &= (n-i+1,0)  \\
         \dij (x_{2i-1},x_{2i}) &= (0,i)  
    \end{align*}
    for $i=0,\cdots n$, where we take $x_{0} = x_{2n} = 0.$ For the definition of $\dij$ see Definition \ref{def:gradingshift}. 

    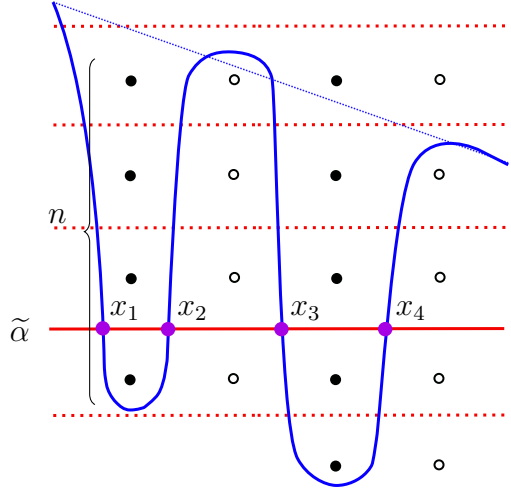
\begin{figure}[htb!]
        \centering
        \begin{tikzpicture}
            \node at (0,0) {\includegraphics[scale=0.4]{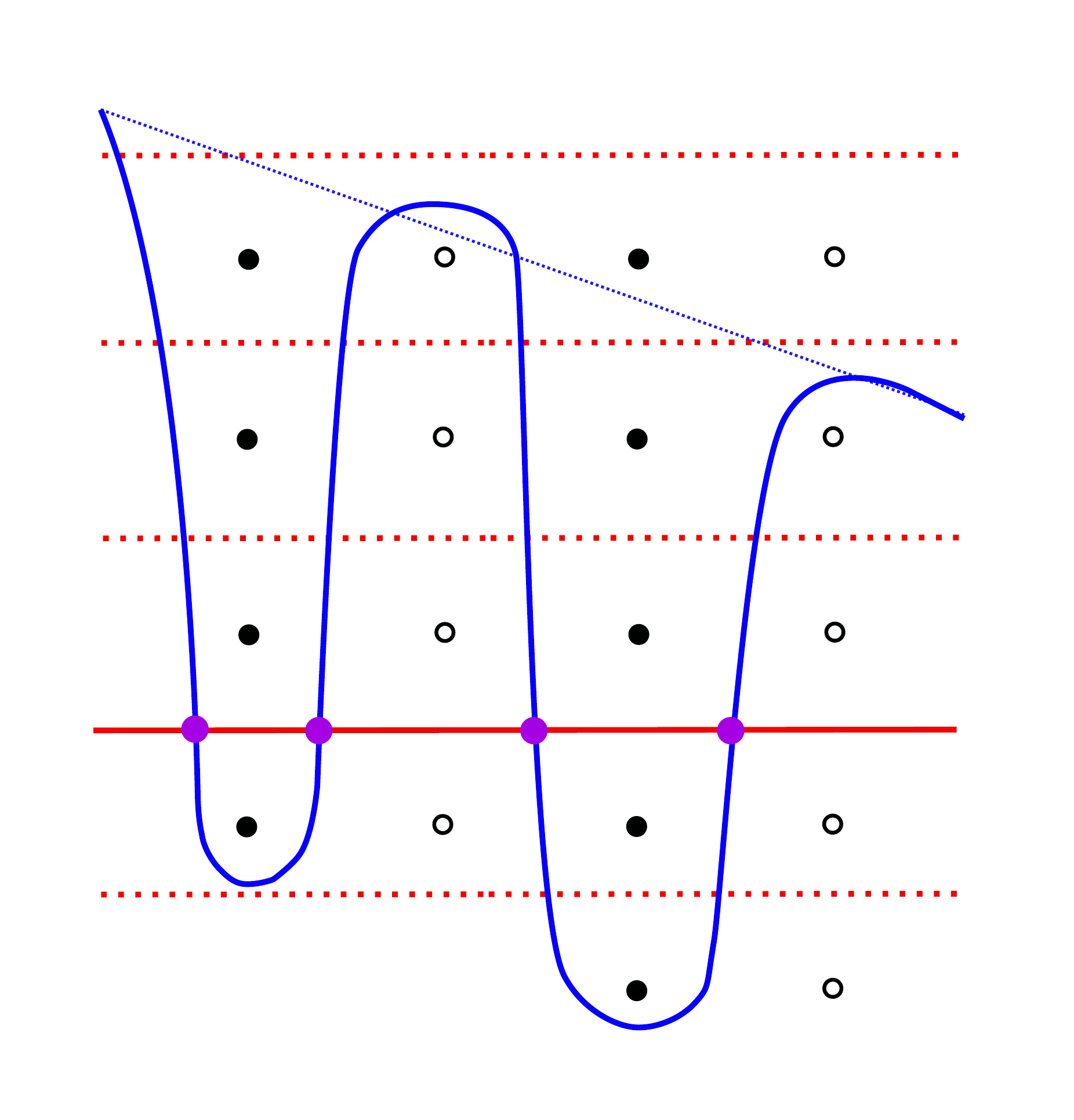}};
             \draw [decorate,decoration={brace,amplitude=4pt,mirror},xshift=0cm,yshift=0pt]
      (-2.5,2.4) -- (-2.5,-2.2) node [midway,yshift=0.2cm,xshift=-0.5cm] {$n$};
      \node at (-3.5,-1.2) {$\aa$};
      \node at (-2.1,-0.95) {$x_1$};
       \node at (-1.2,-0.95) {$x_2$};
       \node at (0.3,-0.95) {$x_3$};
       \node at (1.7,-0.95) {$x_4$};
        \end{tikzpicture}        
        \caption{The $(1,1)$ diagram for $K^{+}([-2n])$ with $n>0$, where the solid dots indicate (lifts of) the $z$ basepoint and the hollow dots indicate (lifts of) the $w$ basepoint.}
        \label{fig:examplelen1}
    \end{figure}
    From Proposition \ref{prop:length1} we already know $K^{+}([-2n]) = -T_{n,n+1}$. So this provides another way of computing the knot Floer complex of $T_{n,n+1}$ torus knots independent of the structure theorem of L-space knot \cite{OSlens} and the computation of their Alexander polynomials.

    On the other direction, without knowing $K^{+}([-2n]) = -T_{n,n+1}$, just by looking at their $(1,1)$ diagrams in Figure \ref{fig:examplelen1}, one can also show they are L-space knots via \cite[Theorem 1.2]{GLV} (the $(1,1)$ diagram is coherent) and this gives a short cut for computing their Alexander polynomials. 
\end{example}
\subsection{General results on the case $\l=3$}\label{subsec:general} From now on we fix the length of the continued fraction to be $3.$ Note that $K^{\pm}([2n_1,0,2n_2]) = K^{\pm}([2n_1 + 2n_2])$, therefore the results here also apply to the knots in the previous section. We will end up reducing the classification to the case when $a_2 = \pm 1$ or $0$, so for the following lemma we consider some relations between this subclass of knots.
\begin{lemma}\label{le:relation}
For any integer $b$, $n$, $n_1$ and $n_2$, we have the following
    \begin{align}
       \label{eq:relation1} K^{\pm}([2n_1,b,2n_2]) &= K^{\pm}([2n_2,b,2n_1]) \\
       \label{eq:relation2} K^{\pm}([2n_1,1,2n_2]) &= K^{\pm}([2(n_1+1),-1,2(n_2+1)])\\
       \label{eq:relation3} K^{\pm}([0,\pm 1,2n]) &= K^{\pm}([2n])\\
        \label{eq:relation4} K^{\pm}([2,-1,2n]) &= K^{\pm}([2n-2])\\
        \label{eq:relation5}  K^{\pm}([-2,1,2n]) &= K^{\pm}([2n+2])\\
         \label{eq:relation6}  K^{+}([2n_1,1,2n_2]) &=  -K^{-}([-2n_1,-1,-2n_2]).
    \end{align}
\end{lemma}
\begin{proof}
   The relation \eqref{eq:relation1} can be seen by rotating the paper plane by $180$ degrees along a vertical axis.\\
   To show \eqref{eq:relation2}, we compute
   \begin{align*}
       [2(n_1+1),-1,2(n_2+1)]&=2(n_1+1)+\cfrac{1}{-1+\cfrac{1}{2(n_2+1)}}\\
      & = \frac{-4(n_1+1)(n_2+1) + 2(n_1 +1) + 2(n_2+1)}{1-2(n_2+1)}\\
      & = \frac{4n_1n_2+2n_1+2n_2}{2n_2+1}\\
      & =  [2n_1,1,2n_2].
   \end{align*}
   The relation \eqref{eq:relation3} is clear from the link diagram. Relations \eqref{eq:relation4} and \eqref{eq:relation5} are straightforward from the continued fraction. Relation \eqref{eq:relation6} comes from mirroring.
\end{proof}

\subsubsection{Full Dehn twists} \label{subsubsec:fullDehntwists}
Let $\H^{\pm}([2n_1,b,2n_2])$ denote the lifted Heegaard diagram in $S^1\times \R$ obtained by applying the action  $\tau^{2n_2}\sigma^b\tau^{2n_1}$ over the $\mp 1$ sloped curve $\bb$. By default, we also fix a preferred lift $\aa$ of $\alpha$ and $\bb$ of $\beta$. According to the algorithm discussed in Section \ref{subsec:algo}, with the above data, $\H^{\pm}([2n_1,b,2n_2])$ induces a basis $B$ for the knot Floer complex $\CFKi(S^3,K^{\pm}([2n_1,b,2n_2]))$ 

When  $b=\pm1$ or $0$,   
after performing $\sigma^b\tau^{2n_1}$, connect each pair of the $z$ and $w$ basepoint with a horizontal line segment $\gamma$.  After potentially pulling tight $\bb,$ we see that $\gamma$ intersects $\bb$ at most once. See Figure \ref{fig:halfDehntwist}\subref{subfig:halfDehntwist0} for the case when $b=1$; the other two cases are similar. The action of $\sigma^2$ is given by the local transformation depicted in Figure \ref{fig:halfDehntwist}\subref{subfig:halfDehntwist12}. Since this transformation is defined inside a small neighbourhood of $\gamma$, we may reverse the order of $\sigma^2$ and $\tau^{2n_2}$.
\begin{definition}
    For an arc $\gamma$, define $\rho_+$ to be the conformal transformation in a neighbourhood of $\gamma$ depicted in Figure \ref{fig:halfDehntwist}\subref{subfig:halfDehntwist12}. This is to be understood up to  rotation in $\R^2.$  We can similarly define $\rho_-$, given by a reflection of $\rho_+$ along a vertical axis. The action $\rho_\pm$ is trivial if $\bb \cap \gamma_{\pm} = \varnothing.$
\end{definition}
\begin{remark}
    A useful perspective is to view $\rho_{+}$ as the reverse of the counterclockwise finger move depicted in Figure \ref{fig:halfDehntwist}\subref{subfig:halfDehntwist4}, where we move both the basepoint and the $\bb$ curve. Equivalently one can also perform a counterclockwise finger move on the other basepoint to achieve the same result. Similar is for  $\rho_{-}$, with the only difference being that the finger move is clockwise.
\end{remark}
We stress that $\rho_\pm$ preserves the angle between $\bb$ and $\gamma.$ For example $\rho_- \circ \rho_+$ is not well defined.  Nevertheless, 
this allows us to reduce any $\H^{\pm}([2n_1,b,2n_2])$ to the case when $b=\pm 1$ or $0.$

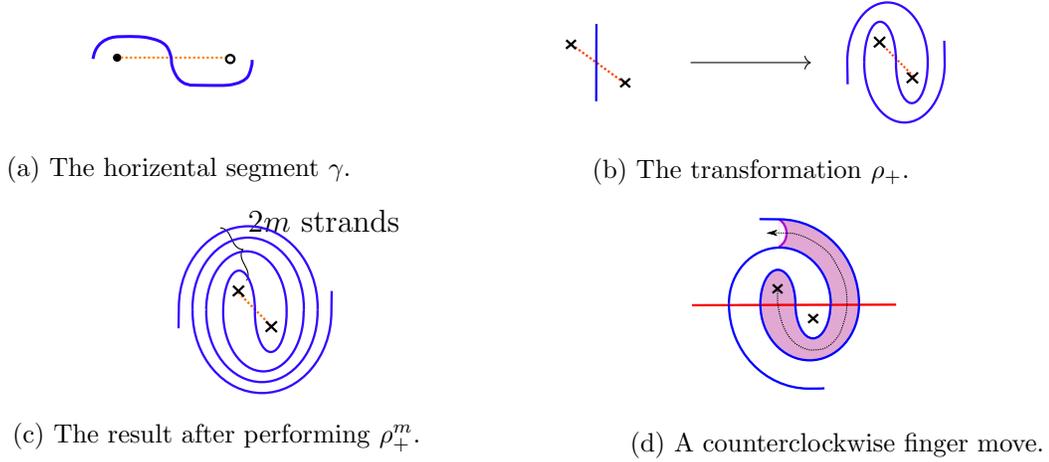
\begin{figure}[htb!]
\begin{minipage}{.4\linewidth}
\centering
\subfloat[The horizental segment $\gamma.$]{
\begin{tikzpicture}
\begin{scope}
    \begin{scope}[thin, black!0!white]
	  \draw  (-4, 0) -- (4,0);
      \end{scope}
\end{scope}
    \node at (0,0) {\includegraphics[scale=1.5]{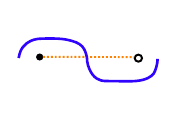}};
\end{tikzpicture}
\label{subfig:halfDehntwist0}
}
\end{minipage}%
\begin{minipage}{.6\linewidth}
\centering
\subfloat[The transformation $\rho_+.$]{
\begin{tikzpicture}
    \begin{scope}[thin, black!0!white]
	  \draw  (-5, 0) -- (5,0);
      \end{scope}
 \draw [->] (-0.8, 0) -- (0.8,0);
    \node at (-2,0) {\includegraphics[scale=1.5]{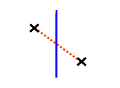}};
   \node at (2,0) {\includegraphics[scale=1.5]{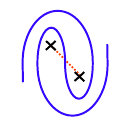}};
\end{tikzpicture}
\label{subfig:halfDehntwist12}
}
\end{minipage}\\
\begin{minipage}{0.5\linewidth}
\centering
\subfloat[The result after performing $\rho_+^{m}.$]{
\begin{tikzpicture}
\begin{scope}
    \begin{scope}[thin, black!0!white]
	  \draw  (-5, 0) -- (4,0);
      \end{scope}
\end{scope}
  \draw [decorate,decoration={brace,amplitude=4pt},xshift=0cm,yshift=0pt]
      (-0.45,1.05) -- (-0.1,0.35) node [midway,right,xshift=0.05cm,yshift=0.45cm] {$2m$ strands};
    \node at (0,0) {\includegraphics[scale=1.5]{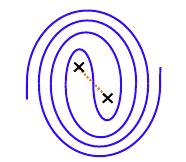}};
\end{tikzpicture}
\label{subfig:halfDehntwist3}
}
\end{minipage}%
\begin{minipage}{0.5\linewidth}
\centering
\subfloat[A counterclockwise finger move.]{
\begin{tikzpicture}
    \begin{scope}[thin, black!0!white]
	  \draw  (-4, 0) -- (5,0);
      \end{scope}
    \node at (0,0) {\includegraphics[scale=1.5]{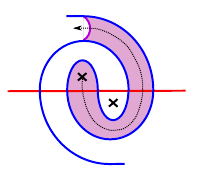}};
\end{tikzpicture}
\label{subfig:halfDehntwist4}
}
\end{minipage}
\caption{In above diagrams, the $\times$ symbol is used to indicate either basepoint.}
\label{fig:halfDehntwist}
\end{figure}

For the following, set $\epsilon = \sgn(b), m=\floor{\epsilon b/2}$ and $b' = b -2\epsilon m$. In order to understand $\H^{\pm}([2n_1,b,2n_2])$, it suffices to understand $\H^{\pm}([2n_1,b',2n_2])$ and the action $\rho_{\epsilon}^m$ over certain arcs (the image of all the $\gamma$ under $\tau^{2n_2}$). 
\begin{definition}\label{def:gamma}
When $b'=1$ or $0$, define $\gamma_{\epsilon}$ to be a straight line segment of slope $-2n_2$ from a $z$ basepoint to a $w$ basepoint in $\H^{\pm}([2n_1,1,2n_2])$. 

When $b'=-1$,  define $\gamma_-$ to be  a straight line segment of slope $-2(n_2-1)$ from a $z$ basepoint to a $w$ basepoint in $\H^{\pm}([2(n_1-1),1,2(n_2-1)])$.

We abuse the notation and let $\gamma_\pm$ also denote the line segments with the same slope in $\H^{\pm}([2n_1,b,2n_2])$.
\end{definition}
\begin{lemma} \label{le:sigma2} Given $n_1,b$ and $n_2,$ for $m$ and $b'$ as above,
\begin{itemize} 
    \item  when $b\geq 0$,  $\H^{\pm}([2n_1,b,2n_2])$ is obtained from $\H(K^{\pm}([2n_1,b',2n_2]))$ by performing the local transformation $\rho_+^{m}$ over all $\gamma_{+}$;
    \item  when $b\leq 0$,  $\H^{\pm}([2n_1,b,2n_2])$ is obtained from $\H(K^{\pm}([2n_1,b',2n_2]))$ by performing the local transformation $\rho_-^{m}$ over all $\gamma_{-}$.
\end{itemize}
\end{lemma}
\begin{proof}
    The case when $b'=0$ or $1$ follows from the fact that $\tau^{2n_2}\sigma^{b}\tau^{2n_1} = \sigma^{2m}\tau^{2n_2}\sigma^{b'}\tau^{2n_1}$ and the image of $\gamma$ under $\tau^{2n_2}$ is $\gamma_+$. Consider the case when $b'=-1$.
   We already know that
   \[
 K^{\pm}([2n_1,-1,2n_2]) = K^{\pm}([2(n_1-1),1,2(n_2-1)]).
\]
In fact, one can check that after pulling tight $\bb$,  $\H^{\pm}([2n_1,-1,2n_2])$ and $\H^{\pm}([2(n_1-1),1,2(n_2-1)])$ represent the same diagram. This justifies the definition of $\gamma_-$ in this case and the rest of the proof follows.
\end{proof}
The map defined in Lemma \ref{le:sigma2} is equivariant; denote them by $\sigma^{\pm 2m}$ respectively. They induce maps on the knot Floer chain complex (also denoted by $\sigma^{\pm 2m}$), which turn out to be encoded by certain $+$ and $-$ markings over the basis elements. 

We describe the process of assigning  $+$ markings, the $-$ case is parallel.
\begin{definition}\label{def:marked}
For a given  $\gamma_{+}$, after pulling tight $\bb$, let $a_{+}=\gamma_{+} \cap \aa $ and  $b_{+}=\gamma_{+} \cap \bb$. Both $a_{+}$ and $b_{+}$ are unique if they exist, as can be seen from the local picture Figure \ref{fig:halfDehntwist}\subref{subfig:halfDehntwist0}. 
From Figure \ref{fig:halfDehntwist}\subref{subfig:halfDehntwist4} we can see that $\rho_{+}$ does not affect the knot Floer complex if $\aa \cap \gamma_{+} = \varnothing.$ We give a  $+$  marking to the ``closest'' intersection point $p$ of $\aa \cap \bb$ to $a_+$ or $b_+$. Concretely, $p$ is the unique intersection point that forms a $(\alpha,\beta,\gamma)$ triangle with $a_{+}$ and $b_{+}.$
We call $p$ a $+$ marked point.

 For the chain complex $C=\CFKi(S^3,K^{\pm}([2n_1,b,2n_2]))$ with the basis $B$ induced by $\aa$ and $\bb,$  assign markings for every $\gamma_{+}$. This results in  a $+$ marked basis $B_+.$ Note that it is possible for a basis element to be assigned more than one $+$ markings.

 Similarly define intersection points $a_-, b_-, $ the $ -$ marked points and $-$ marked basis $B_-.$
\end{definition}
Let $D_1$ be the filtered chain complex over $\F[U,U^{-1}]$ generated by $x_0,x_1,x_2$ and $y$ with differentials (each with length one) as below
\[
\xymatrix{
x_{0} \ar[d] & x_1 \ar[d] \ar[l]   \\
y  & x_2 \ar[l] 
}
\]
Namely, $D_1$ is a length one box summand.
\begin{lemma}\label{le:sigma2+}
    For the knot Floer complex with marked basis  $(\CFKi(K^{\pm}([a_1,a_2,a_3])),B_+)$ where $a_1$ and $a_3$ even and $a_2 \geq 0,$ the map \[
    \sigma^2 \co (\CFKi(K^{\pm}([a_1,a_2,a_3])),B_+) \longrightarrow (\CFKi(K^{\pm}([a_1,a_2+2,a_3])),B'_+)\] 
    is defined as follows. For each $+$ marked point $p$, add $\l$ copies of $D_1$ summands, where $\l$ is the number of $+$ markings of $p,$ such that for each $D_1$ summand, $y$ and $p$ share the same filtration level and Maslov grading. Remove all the previous markings and  assign a $+$ marking to $x_1$ of each added $D_1$ summand. This gives the new marked basis $B'_+$. 
\end{lemma}
\begin{proof}
    Given a intersection point $b_+ = \gamma_+ \cap \bb$ for some $\gamma_+,$ let $p$ be the $+$ marked point induced by $\gamma_+.$ In a small neighbourhood of $\gamma_+,$ the transformation $\rho_+$ is given by Figure \ref{fig:sigma2}\subref{subfig:sigma2_0}, up to rotation in $\R^2.$ In the resulting diagram, label the intersection points by $p_0, c, a, b$ and $p_1$ from left to right (as we will see, this assignment also only matters up to the symmetry from the middle). They generate a complex depicted in Figure \ref{fig:sigma2}\subref{subfig:sigma2_1}. In particular, we can perform a change of basis $\{a, b, c, p_0, p_1\} \mapsto \{a, b, c, p_0 + p_1, p_i\}$ where $i=0$ or $1.$ Both choice of the change of basis splits off a summand generated by $\{a, b, c, p_0 + p_1\}$ isomorphic to $D_1.$  We then identify either $p_0$ or $p_1$ with $p$. For $i=0,1,$  clearly $p_0+p_1$ share the same filtration level and Maslov grading with $p_i$; moreover  observe that  $p_i$ inherits all the bigons of $p$ (either incoming or outgoing, with the same filtration shifts). Adopting the perspective of Figure \ref{fig:halfDehntwist}\subref{subfig:halfDehntwist4}, performing a counterclockwise finger move will undo the transformation $\rho_+$ on the diagram level, and remove a box summand on the chain complex level. The choice of basepoint with which we perform the finger move corresponds to identifying $p$ with $p_0$ or $p_1.$  

    Over the diagram  $\H^{\pm}([a_1,a_2+2,a_3])$, following $\bb$ and record the sequence of $b_+$ in order.  Perform the finger move in  Figure \ref{fig:halfDehntwist}\subref{subfig:halfDehntwist4} near each $b_+$ (this amounts to removing an interval of $\bb$ near $b_+$ and gluing it back). The resulting diagram is $\H^{\pm}([a_1,a_2,a_3]).$ This proves the statement regarding the chain complex. For the new marked basis $B'_+,$ simply notice that in the local picture  Figure \ref{fig:sigma2}\subref{subfig:sigma2_0}, the intersetion point $a$ forms a triangle with $a_+$ and $b_+.$
\end{proof}
\begin{figure}[htb!]
\begin{minipage}{0.5\linewidth}
\centering
\subfloat[]{
\begin{tikzpicture}
    \begin{scope}[thin, black!0!white]
	  \draw  (-6, 0) -- (3,0);
     \end{scope}
     \node at (-3.2,-0.2) {\tiny $p$};
          \node at (0.28,0.15) {\tiny $b$};
         \node at (-0.75,0.15) {\tiny$c$};
         \node at (0.95,-0.15) {\tiny$p_1$};
           \node at (-0.28,-0.15) {\tiny$a$};
                 \node at (-1.15,-0.15) {\tiny$p_0$};
    \node at (0,0) {\includegraphics[scale=2]{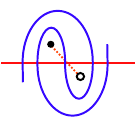}};
     \draw [thick,-latex'] (-2.2, 0) -- (-1.6, 0);
     \node at (-3,0) {\includegraphics[scale=2]{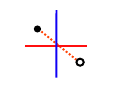}};
\end{tikzpicture}
\label{subfig:sigma2_0}
}
\end{minipage}%
\begin{minipage}{0.5\linewidth}
\centering
\subfloat[]{
\begin{tikzpicture}
    \begin{scope}[thin, black!0!white]
	  \draw  (-5, 0) -- (5,0);
      \end{scope}
       \begin{scope}[thin, black!30!white]
	  \foreach \i in {-1,...,1}
   {
              \draw  (-1.5, \i) -- (1.5, \i);
      	 \draw  (\i, -1.5) -- (\i,1.5);       
       }
      \end{scope}
   \filldraw (-0.5, 0.5) circle (2pt) node[] {};
      	\filldraw (0.5, 0.5) circle (2pt) node[] {};
       \filldraw (0.5, -0.5) circle (2pt) node[] {};
       \filldraw (-0.5, -0.5) circle (2pt) node[] {};
           \filldraw (-0.8, -0.8) circle (2pt) node[] {};
      \draw [thick,-latex'] (0.5, 0.5) -- (0.5, -0.4);
      \draw [thick,-latex'] (0.5, 0.5) -- (-0.4, 0.5);
        \draw [thick,-latex'] (0.5, -0.5) -- (-0.4, -0.5);
         \draw [thick,-latex'] (-0.5, 0.5) -- (-0.5, -0.4);
           \draw [thick,-latex'] (0.5, -0.5) -- (-0.7, -0.8);
         \draw [thick,-latex'] (-0.5, 0.5) -- (-0.8, -0.7);
         \node at (0.7,0.7) {$a$};
         \node at (-0.7,0.7) {$b$};
         \node at (0.7,-0.7) {$c$};
           \node at (-0.2,-0.2) {$p_1$};
                 \node at (-1,-1) {$p_0$};
\end{tikzpicture}
\label{subfig:sigma2_1}
}
\end{minipage}\\
\begin{minipage}{1\linewidth}
\centering
\subfloat[The effect of the map $\sigma^{2m}$.]{
\begin{tikzpicture}
    \begin{scope}[thin, black!0!white]
	  \draw  (-6, 0) -- (6,0);
      \end{scope}
            \filldraw (-0.7, -0.7) circle (2pt) node[] {};
   \filldraw (-0.5, 0.5) circle (2pt) node[] {};
      	\filldraw (0.5, 0.5) circle (2pt) node[] {};
       \filldraw (0.5, -0.5) circle (2pt) node[] {};
       \filldraw (-0.5, -0.5) circle (2pt) node[] {};

      \draw [thick,-latex'] (0.5, 0.5) -- (0.5, -0.4);
      \draw [thick,-latex'] (0.5, 0.5) -- (-0.4, 0.5);
        \draw [thick,-latex'] (0.5, -0.5) -- (-0.4, -0.5);
         \draw [thick,-latex'] (-0.5, 0.5) -- (-0.5, -0.4);

            \filldraw (1, 2) circle (2pt) node[] {};
      	\filldraw (2, 2) circle (2pt) node[] {};
       \filldraw (2, 1) circle (2pt) node[] {};
       \filldraw (1, 1) circle (2pt) node[] {};

      \draw [thick,-latex'] (2, 2) -- (2, 1.1);
      \draw [thick,-latex'] (2, 2) -- (1.1, 2);
        \draw [thick,-latex'] (2, 1) -- (1.1, 1);
         \draw [thick,-latex'] (1,2) -- (1, 1.1);

\node [rotate = 45] at (0.75,0.75) {$\cdots$};
  \draw [decorate,decoration={brace,amplitude=4pt},xshift=0cm,yshift=0pt]
      (-0.65,0.55) -- (1.1,2.2) node [midway,left,xshift=0cm, yshift=0.4cm] {$m$ copies};
\end{tikzpicture}
\label{subfig:sigma2_2}
}
\end{minipage}
\caption{}
\label{fig:sigma2}
\end{figure}
\begin{lemma}\label{le:sigma2-}
    For the knot Floer complex with marked basis  $(\CFKi(K^{\pm}([a_1,a_2,a_3])),B_-)$ where $a_1$ and $a_3$ even and $a_2 \leq 0,$ the map 
    \[
     \sigma^{-2} \co (\CFKi(K^{\pm}([a_1,a_2,a_3])),B_-) \longrightarrow (\CFKi(K^{\pm}([a_1,a_2-2,a_3])),B'_-)\] 
    is defined as follows. For each $+$ marked point $p$, add $\l$ copies of $D_1$ summands, where $\l$ is the number of $-$ markings of $p,$ such that for each $D_1$ summand, $x_1$ and $p$ share the same filtration level and Maslov grading. Remove all the previous markings and  assign a $-$ marking to $y$ of each added $D_1$ summand. This gives the new marked basis $B'_-$. 
\end{lemma}
\begin{proof}
    This is parallel to Lemma \ref{le:sigma2+}.
\end{proof}
Practically, we only need to consider the complex with marked basis $(\CFKi(K^{\pm}([2n_1,b',2n_2])),B_\pm)$ with $b'=\pm 1$ and $0$ and consider the map $\sigma^{\pm2m}.$ In general it is not difficult to determine $B_{\pm}$ using Definition \ref{def:gamma} and \ref{def:marked}. In particular, we have proved the following.
\begin{proposition} \label{prop:sigma2}
Up to chain homotopy equivalence,
\[
 \CFKi(S^3,K^{\pm}([2n_1,b,2n_2])) \cong \CFKi(S^3,K^{\pm}([2n_1,b',2n_2])) \oplus \bigoplus^{N}_{i=1} D_1,
\]
   where $b' = b -2\epsilon m$ for $\epsilon = \sgn(b), m=\floor{\epsilon b/2}$, $N$ is the number of makrings in the marked basis $B_{\epsilon}$ and $D_1$ is the length one box complex. 
\end{proposition}
For any $n_1$ and $n_2$, by Definition \ref{def:gamma}, $\gamma_+$ and $\gamma_-$ coincide for $\CFKi(K^{\pm}(S^3,[2n_1,0,2n_2]))$, so $B_+ = B_-.$ In this case it is easy to check that
    the map $\sigma^2$ in Lemma \ref{le:sigma2+}   and the map $\sigma^{-2}$ in Lemma \ref{le:sigma2-} induce the same map on the level of filtered chain homotopy type, even though they induce different basis in the image. In other words, we have the following.
    \begin{lemma} \label{le:2brelation}
Up to filtered chain homotopy equivalence, for $n_1,n_2,b \in \Z,$
\[
\CFKi(K^{\pm}(S^3,[2n_1,2b,2n_2])) \cong \CFKi(K^{\pm}(S^3,[2n_1,-2b,2n_2])).
\]
\end{lemma}    
    Note that marked basis $B_{+}$ and $B_{-}$ are generally different for $\CFKi(K^{\pm}([2n_1,1,2n_2]))$ since   $\gamma_{+}$ and $\gamma_{-}$ have different slopes.
    
In view of Lemma \ref{le:relation},\ref{le:sigma2+}  \ref{le:sigma2-} and \ref{le:2brelation} 
we can reduce the $\l=3$ case to  some subcases as follows.  Suppose we are given  $[a_1,a_2,a_3]$  with $a_1, a_3 \in 2\Z$ and $a_2\in \Z$. If $a_2$ is even, then it suffice to consider $a_2>0$, and by mirroring we only need to consider $K^+$. If $a_2$ is odd, by mirroring we can guarantee that $a_1 >0$. Next reducing $a_2$ to the case $a_2=\pm 1$, we can further restrict to the case $a_2=1$ and consider both $B_+$ and $B_-$ on $\CFKi(S^3,K^{\pm}([a_1,1,a_3]))$. Finally note that we can require $a_3 \neq -2$.

In summary, we have the following five cases, that are the topics of the next five subsections, respectively. We record the conclusion of each subsection here for the reader's convenience. This, together with the $\l=1$ case completes the proof of Theorem \ref{thm:class}.
\begin{proposition}\label{prop:summary} 
According to the discussion above, the case $\l=3$ is fully classified by the following cases:
   for $n_1, n_2 \in \Z$ and integer $b>0,$
    \begin{enumerate}
    \item $K^{+}([2n_1,2b,2n_2])$, Corollary \ref{cor:2nsum};
\end{enumerate}
 for $n_1,n_2>0,$
    \begin{enumerate}[resume]
    \item $K^{+}([2n_1,1,2n_2])$, Proposition \ref{prop:class1};
     \item $K^{-}([2n_1,1,2n_2])$, Proposition \ref{prop:class2};
     \end{enumerate}
  and   for $n_1>0,n_2>1,$
      \begin{enumerate}[resume]
      \item $K^{+}([2n_1,1,-2n_2])$, Proposition \ref{prop:class3};
      \item $K^{-}([2n_1,1,-2n_2])$, Proposition \ref{prop:class4}.
\end{enumerate} 
\end{proposition}

\subsection{Case $K^{+}([2n_1,2b,2n_2])$ with $b>0$}\label{subsec:class0}
 Since $K^{+}([2n_1,0,2n_2]) = K^{+}([2n])$ for $n=n_1 + n_2,$ with $K^{+}([2n])$ already classified in Subsection \ref{subsec:len1}, the only extra data we require is a set of  markings on $\CFKi(S^3,K^{+}([2n]))$ (which as complexes of L-space knots,  admits a unique basis). 

\subsubsection{Marked basis for $K^{+}([-2n])$ with $n>0$ }
Fix $-n = n _1 + n_2$ and
 let $n_2 = -k $ for $k\in \Z$.  We seek to decide the marked basis corresponding to $ K^{+}([-2(n-k),2b,-2k])$ for $b>0$.  By Definition \ref{def:gamma} $\gamma_+$ is of slope $k.$  Revisit Figure \ref{fig:examplelen1}: in a $S^1\times \R$ slice, denote the lifts of $\alpha$ that intersect a chosen lift $\bb$ of $\beta$ by $\alpha_1$ through $\alpha_{n}$ from bottom to top. Fix a lift $\aa$ of $\alpha$.  It takes $n$ iterations  to cover all the intersections of $\aa \cap \bb,$ with $\aa$ being identified with $\alpha_\l$ for $1\leq \l \leq n$ in each iteration. In total there are $2n-1$ generators in the complex $\CFKi(S^3,K^{+}([-2n])) = \CFKi(S^3,T_{n,n+1}).$ Every iteration covers $2$ intersection points except the last one, which covers $1$ intersection point. 

 The only intersection points that are marked are the ones in the middle of each $S^1\times \R$ slice. We focus on the portion of $\bb$ that travels between $z$ and $w$ basepoints, which is isotopic to a slope $n$ line segment. The question of how many markings each intersection point receives is a completely combinatorial one: for each $\alpha_\l$ we need only count the number of line segments of slope $k$ that intersect both line segments of slope $n$ and slope $0$. 
 \begin{figure}[htb!]
     \centering
     \begin{tikzpicture}
         \node at (0,0) {\includegraphics[scale=0.3]{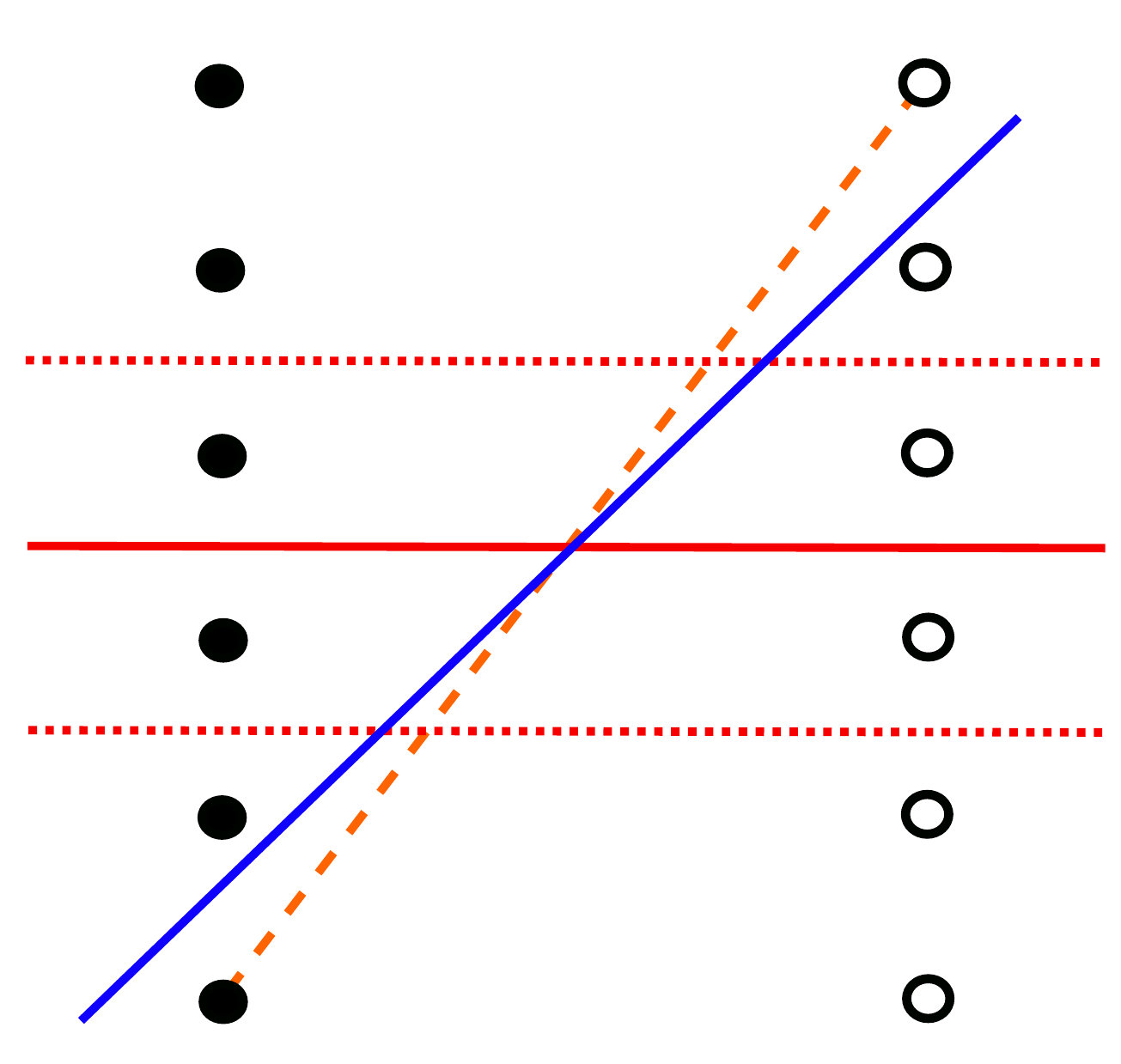}};
         \node [text=red] at (-2.2,-0.8) {$\alpha_1$};
          \node [text=red] at (-2.2,0) {$\alpha_2$};
               \node [text=red] at (-2.2,0.7) {$\alpha_3$};
                \node [text=Orange] at (0.5,1) {$\gamma_+$};
                    \node [text=blue] at (1.8,1.5) {$\bb$};
    \node  at (-1.5,-1.9) {\small $-1$};
      \node  at (-1.5,-1.1) {\small $0$};
      \node  at (-1.5,-0.4) {\small $1$};
     \node  at (-1.5,0.3) {\small $2$};
     \node  at (-1.5,0.95) {\small $3$};
     \node  at (-1.5,1.65) {\small $4$};
     \end{tikzpicture}
     \caption{The $\gamma_+$ line segment with slope $n+1$ intersects the $n$ sloped $\bb$ line segment once.}
     \label{fig:lineseg1}
 \end{figure}
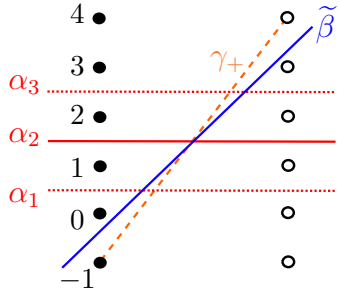
 
 In the below proposition, note that $\CFKi(S^3,T_{n,n+1})$ admits a unique basis, and by the symmetry it does not matter from which end we start counting the generators.
\begin{proposition}\label{prop:2ncase1}
   For $n,b>0$ and $k\in \Z \setminus \{0,n\}$, corresponding to $ K^{+}([-2(n-k),2b,-2k])$, the $(2\l)$-th generator in $\CFKi(S^3,T_{n,n+1})$ receives $m(n,k,\l)$ markings for $1\leq \l \leq n-1$,  where 
   \begin{align}
     \label{eq:m} m(n,k,\l) = \begin{cases}
           k-n \hspace{5em} &\hspace{2em} k \geq n+1\\
           n-\l + \min\{\l-k,0\} + \min\{k+\l-n,0\} &\hspace{2em} 1\leq k \leq n-1\\
           -k  \hspace{5em} &\hspace{2em} k \leq -1.
       \end{cases}
   \end{align}
\end{proposition}
\begin{proof}
For $1\leq \l \leq n-1,$ in the $\l$-th iteration, $\alpha_\l$ is identified with $\aa$ and the intersection point depicted in Figure \ref{fig:lineseg1} is  the $(2\l)$-th generator. Label the height of the  $z$ and $w$ basepoints by $j$, such that $\alpha_\l$ is between $j=\l$ and $j=\l -1$.

   When $k\geq n+1,$ $\gamma_+$ intersects the line segment of slope $n$ if and only if it starts from $j\leq -1$ and ends at $j \geq n.$ There are $k-1 -n + 1 = k-n$ of such line segments in total.

   When $k\leq -1,$ $\gamma_+$ intersects the line segment of slope $0$ if and only if it starts from $j\geq \l$ and ends at $j \leq \l -1.$ There are $\l - 1 - (\l + k) +1   = -k$ of such line segments in total.

   When $1\leq k\leq n-1,$ $\gamma_+$ intersects both the line segments  if only if it starts between $0\leq j \leq \l -1$ and ends between $\l\leq j \leq n -1.$ In other words, we need only find the length of the interval $[k,k+\l-1]\cap [\l,n-1],$ which is given by
   \begin{align*}
       \min\{k+\l-1&,n-1\} - \max\{\l,k\} + 1 \\
       &= n-\l + \min\{\l-k,0\} + \min\{k+\l-n,0\}.
   \end{align*}
\end{proof}
Note that the expression $m(n,k,\l)$ is symmetric under the transformations
\begin{align*}
    (n,k,\l) &\longmapsto  (n,k,n- \l)  \\
    (n,k,\l) &\longmapsto (n,n- k,\l).
\end{align*}
\subsubsection{Marked basis for $K^{+}([2n])$ with $n>0$ } This is similar to the previous case.
Fix $n = n _1 + n_2$ and
 let $n_2 = -k $ for $k\in \Z$.  We seek to decide the marked basis corresponding to $ K^{+}([2(n+k),2b,-2k])$ for $b>0$.   Similarly label the lifts of $\alpha$ by $\alpha_1$ through $\alpha_{n-1}$ from bottom to top.   It takes $n-1$ iterations  to cover all the intersections of $\aa \cap \bb,$ with $\aa$ being identified with $\alpha_\l$ for $1\leq \l \leq n-1$ in each iteration. In total there are $2n-3$ generators in the complex $\CFKi(S^3,K^{+}([2n])) = \CFKi(S^3,T_{n,n-1}).$ Every iteration covers $2$ intersection points except the last one, which covers $1$ intersection point. 

 The only intersection points are the ones in the middle of each $S^1\times \R$ slice and the portion of $\bb$ that travels between $z$ and $w$ basepoints is of slope $-n$. By Definition \ref{def:gamma} $\gamma_+$ is of slope $k.$ The proof of the next proposition is similar to the previous case and left to the reader as an exercise.
 \begin{proposition}\label{prop:2ncase2}
   For $n,b>0$ and $k\in \Z \setminus \{0,-n\}$, corresponding to $ K^{+}([2(n+k),2b,-2k])$, the $(2\l -1)$-th generator in $\CFKi(S^3,T_{n,n-1})$ receives $m'(n,k,\l)$ markings for $1\leq \l \leq n-1$,  where 
   \begin{align}
     \label{eq:m'} m'(n,k,\l) = \begin{cases}
           -k-n \hspace{5em} &\hspace{2em} k \leq -n-1\\
           n-\l + \max\{\l-n-k,0\} + \max\{k+\l,0\} &\hspace{2em} -n+1\leq k \leq -1\\
           k  \hspace{5em} &\hspace{2em} k \geq 1.
       \end{cases}
   \end{align}
\end{proposition}
\subsubsection{Marked basis for $K^{+}([0])$ }
Again let $n_2 = -k$ for $k\in \Z.$
\begin{definition}\label{def:c0}
    Let $C_0$ be the complex generated by one element. 
\end{definition}
We seek to decide the marked basis in $C_0 $  corresponding to $K^{+}([2k,2b,-2k])$ for $b>0$. Similar analysis as from the previous sections shows that the unique generator in $C_0$ is assigned $|k|$ markings for $k\in \Z$. This concludes the discussion of the case $K^{+}([2n_1,2b,2n_2])$ with $b>0$.
In particular, Proposition \ref{prop:2ncase1}, \ref{prop:2ncase2} and above imply the following.
\begin{corollary}\label{cor:2nsum}
    For any $n_1,n_2 \in \Z_{\neq 0}$  and $b>0,$ up to chain homotopy equivalence $\CFKi(S^3,K^{+}([2n_1,2b,2n_2]))$ is given by
    \begin{itemize}
        \item  if $n_1 + n_2 =0$,
        \[
        C_0 \oplus b|n_2| D_1;
        \]
        \item if $n_1 + n_2 >0$, denoting $n= n_1 + n_2$ and $ k=-n_2,$
        \[
        \CFKi(S^3,T_{n,n-1}) \oplus b \left(\sum_{\l=1}^{n-1} m'(n,k,\l)\right) D_1,
        \]
        where $m(n,k,\l)$ is given by \eqref{eq:m'};
        \item if $n_1 + n_2 <0$, denoting $n= - (n_1 + n_2)$ and $ k=-n_2,$
        \[
        \CFKi(S^3,T_{n,n+1}) \oplus b\left(\sum_{\l=1}^{n-1} m(n,k,\l)\right) D_1,
        \]
        where $m(n,k,\l)$ is given by \eqref{eq:m}.
    \end{itemize}
\end{corollary}

\subsection{Case $K^{+}([2n_1,1,2n_2])$ with $n_1,n_2>0$}\label{subsec:class1}
  From now on let us write $K^{+}_{n_1,n_2}$ for the knot $K^{+}([2n_1,1,2n_2])$ for the simplicity of the notation. See Figure \ref{fig:Ln1n2} for a depiction of the knot $K^{+}_{n_1,n_2}$.  Since  the corresponding rational tangle has a presentation $[2n_1,1,2n_2]$,  by Jonathan Hales's algorithm we shall consider the action $\tau^{2n_2}\sigma \tau^{2n_1}$ on the $(1,1)$ diagram. The entire procedure is shown in Figure \ref{fig:11diagram}.

\begin{figure}[htb!]
\labellist
  \pinlabel {$1$}  at 445 210

 \pinlabel {{$n_1$}}  at 128 156
  \pinlabel {$n_2$}  at 315 158
\endlabellist
\includegraphics[scale=0.5]{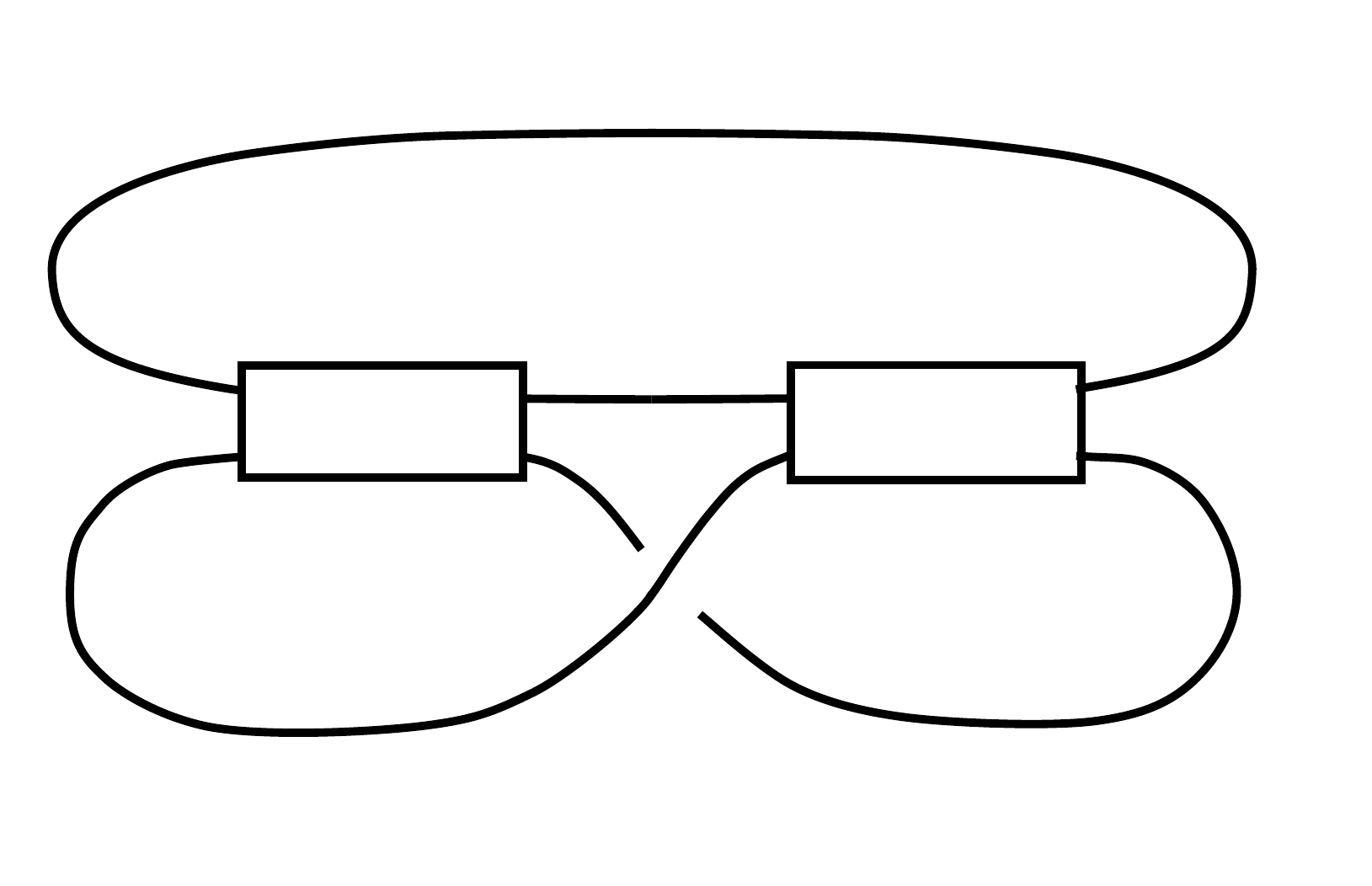}
\caption{The knot  $K^{+}_{n_1,n_2}$. Numbers in the boxes indicate the number of full-twists.  }
\label{fig:Ln1n2}
\end{figure}
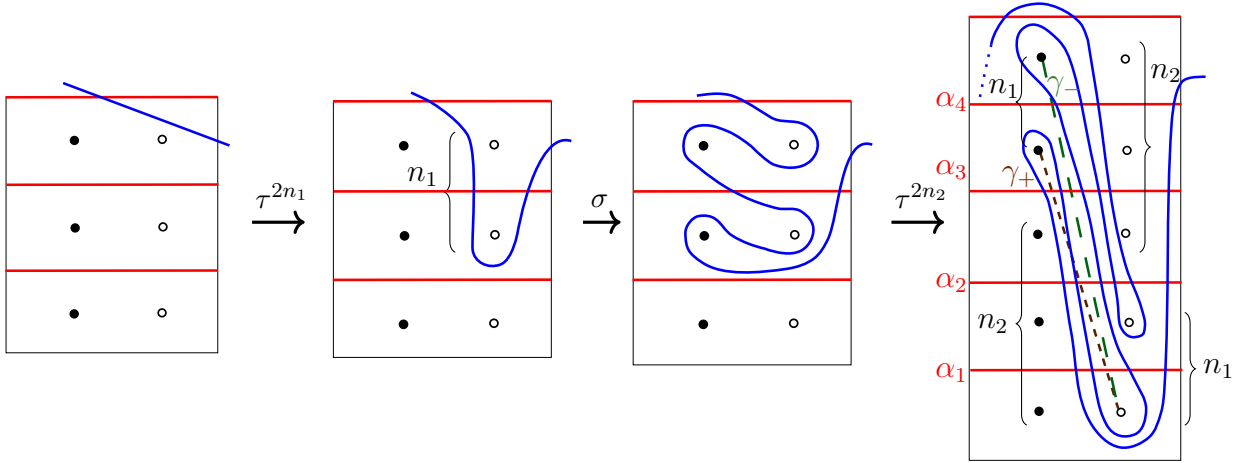
\begin{figure}
\begin{tikzpicture}
 \begin{scope}[thin, black!0!white]
          \draw  (0, 0) -- (18, 0);
      \end{scope}
\centering
       \draw [decorate,decoration={brace,amplitude=4pt,mirror},xshift=0.5cm,yshift=0pt]
      (3.5,1.2) -- (3.5,-0.4) node [midway,left,xshift=-0.08cm,yshift=6pt] {$n_1$};
\node at (-0.5,0.1) {\includegraphics[scale=0.34]{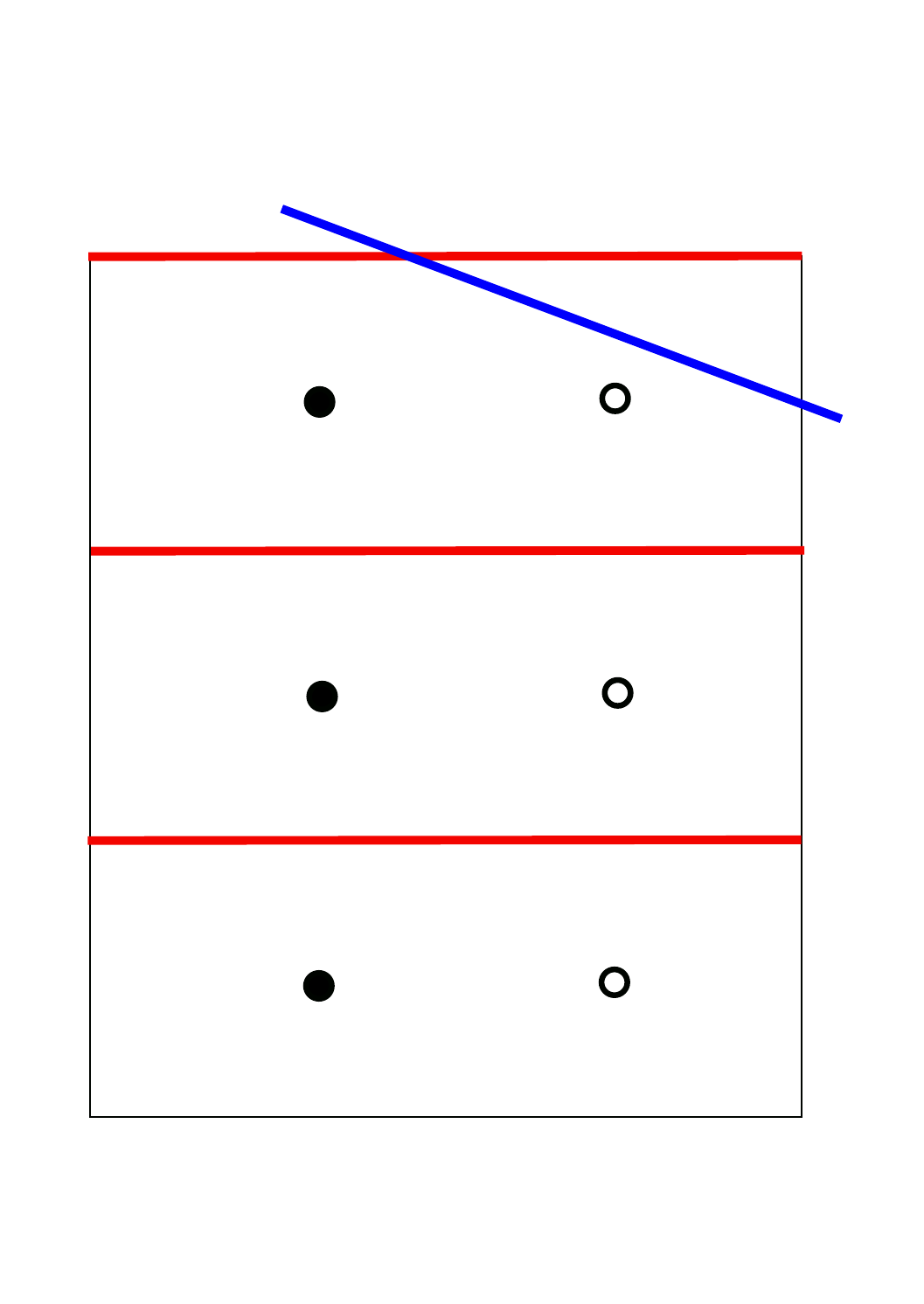}};
\node at (1.7,0.35) {\small{$\tau^{2n_1}$}};
\draw [->,very thick] (1.3,0) -- (2,0);
\node at (4,0) {\includegraphics[scale=0.35]{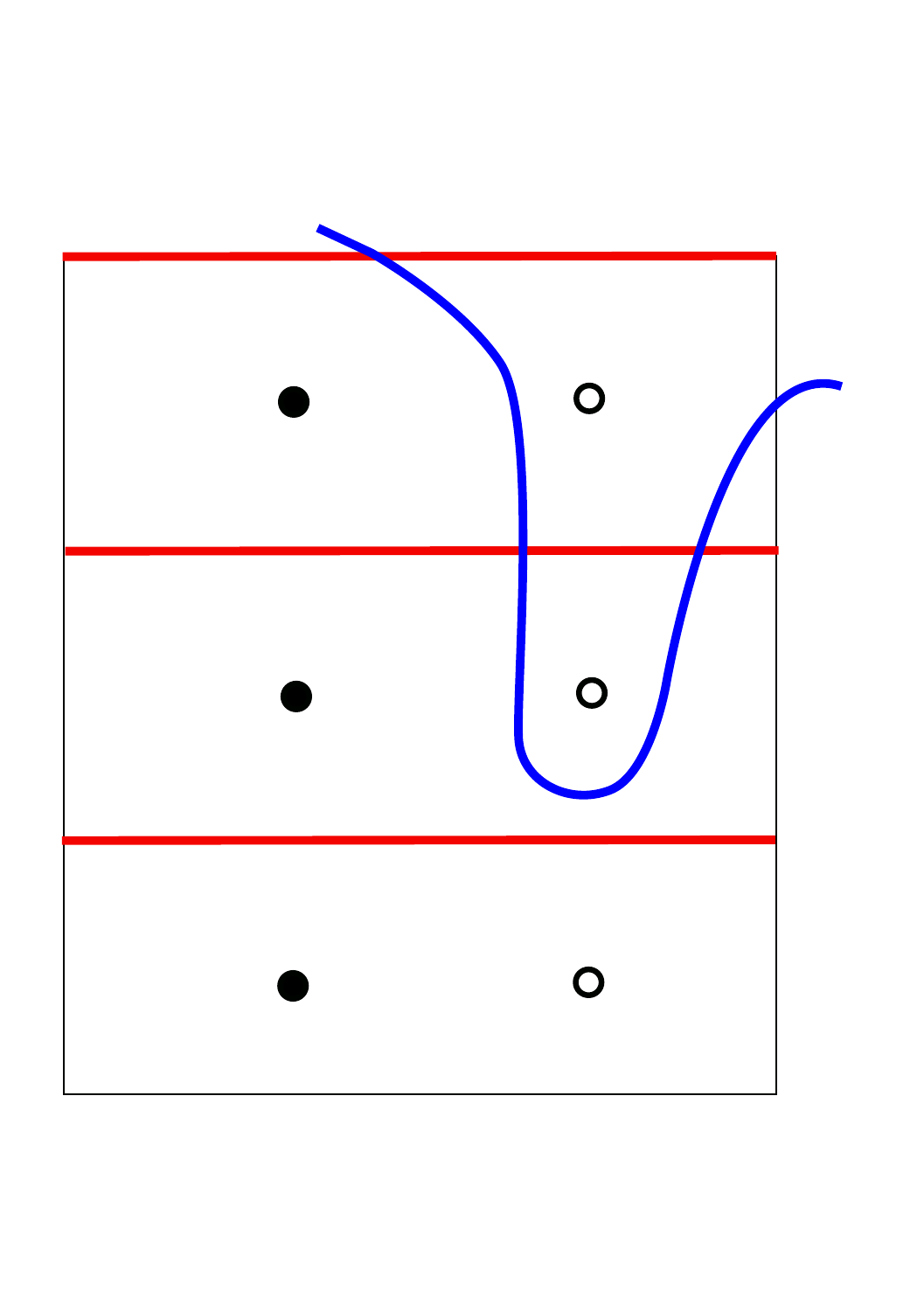}};
\node at (5.92,0.28) {\small{$\sigma$}};
\draw [->,very thick] (5.7,0) -- (6.2,0);
\node at (8,0) {\includegraphics[scale=0.35]{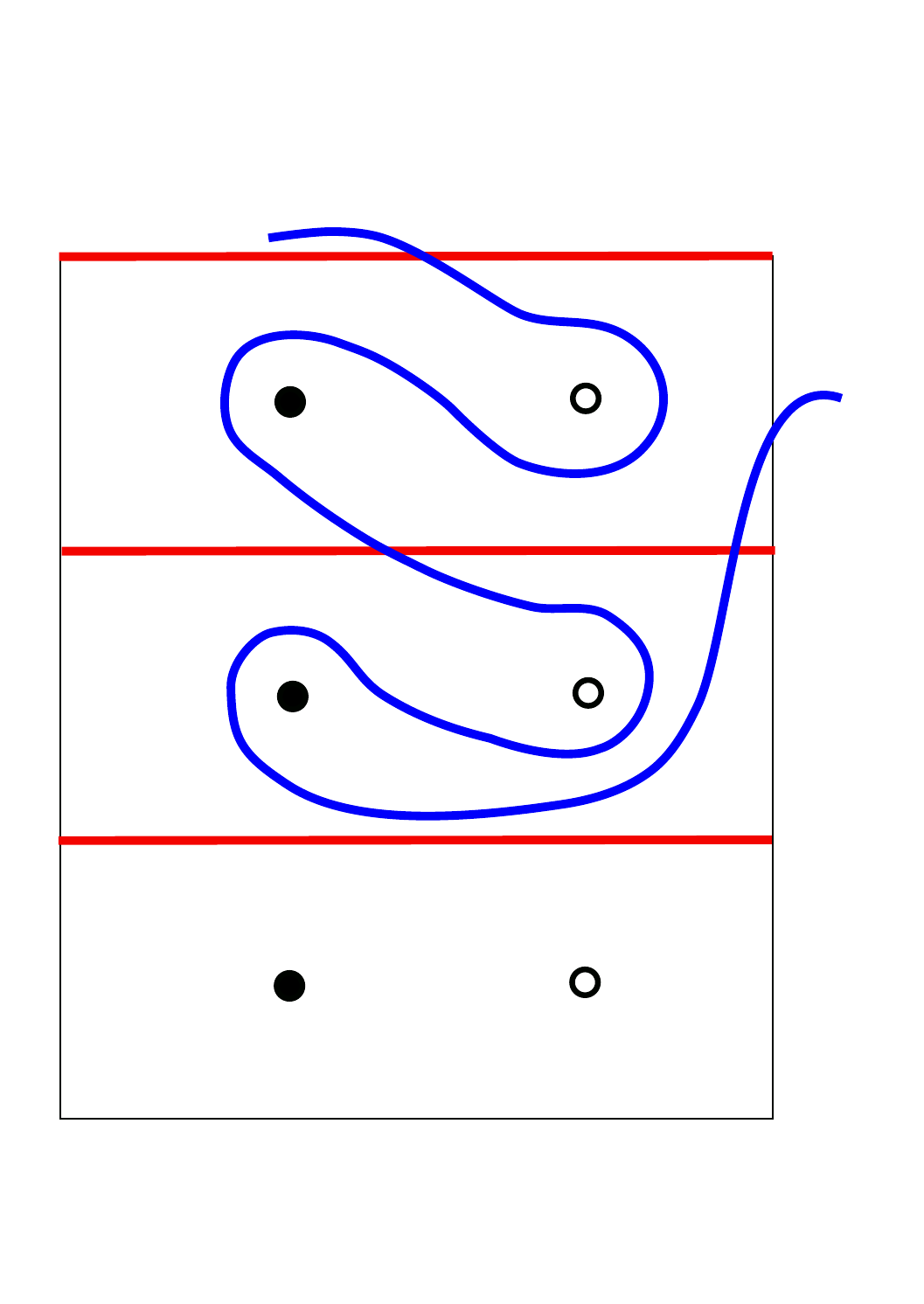}};
\node at (10.2,0.35) {\small{$\tau^{2n_2}$}};
\draw [->,very thick] (9.8,0) -- (10.5,0);
\node at (12.5,0) {\includegraphics[scale=0.34]{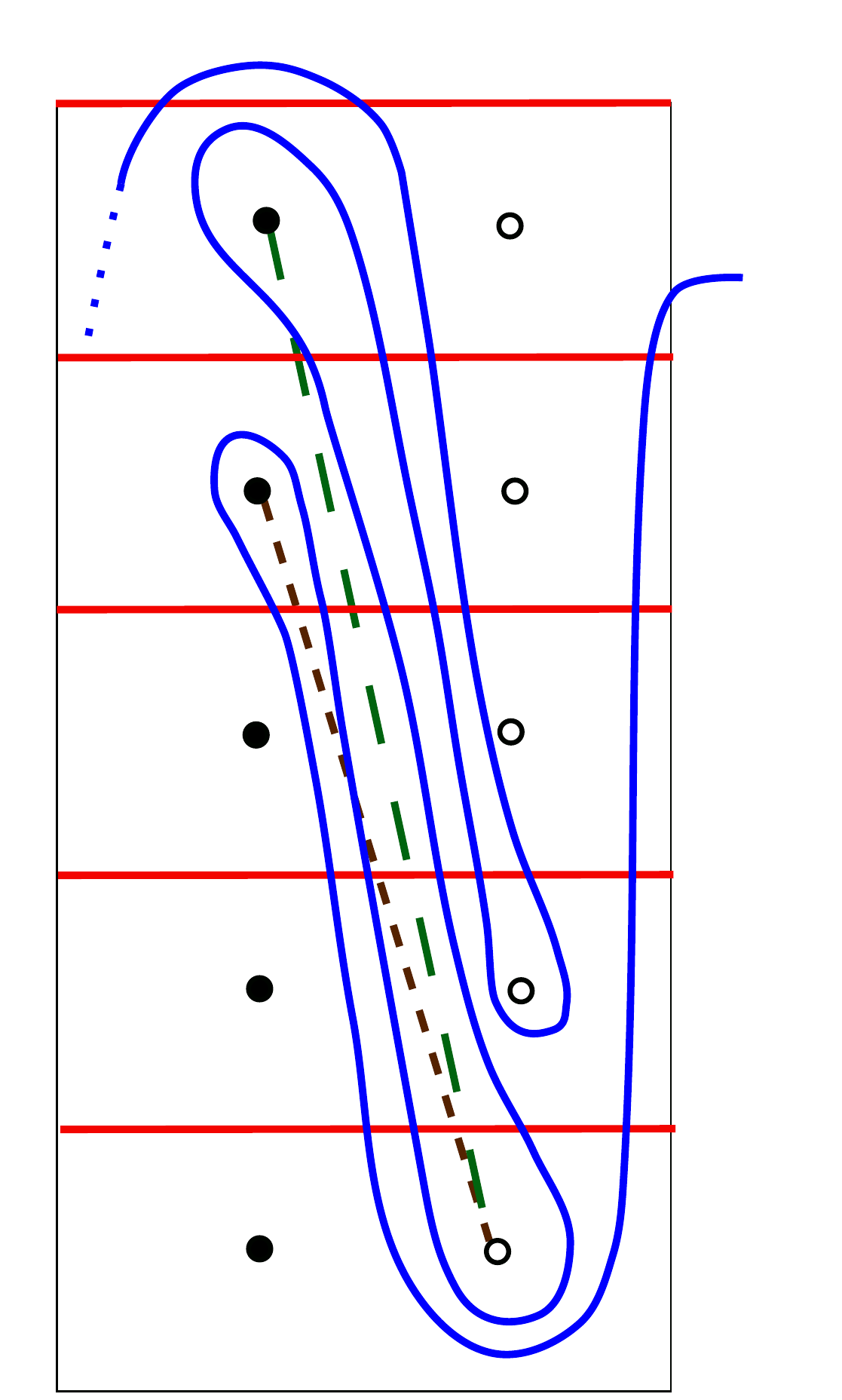}};
 \draw [decorate,decoration={brace,amplitude=4pt},xshift=0.5cm,yshift=0pt]
      (13.2,-1.2) -- (13.2,-2.7) node [midway,right,xshift=0.1cm] {$n_1$};
 \draw [decorate,decoration={brace,amplitude=4pt},xshift=0.5cm,yshift=0pt]
      (12.6,2.4) -- (12.6,-0.4) node [midway,right,xshift=0cm,yshift=28pt] {$n_2$};
       \draw [decorate,decoration={brace,amplitude=4pt,mirror},xshift=0.5cm,yshift=0pt]
      (11.1,2.2) -- (11.1,1) node [midway,left,xshift=0.05cm,yshift=5pt] {$n_1$};
       \draw [decorate,decoration={brace,amplitude=4pt,mirror},xshift=0.5cm,yshift=0pt]
      (11.1,0) -- (11.1,-2.7) node [midway,left,xshift=-0.1cm] {$n_2$};
 \node at (12.1,1.8) [text=OliveGreen] {$\gamma_-$};
 \node at (11.5,0.6) [text=Brown] {$\gamma_+$};
 
 \node at (10.6,1.6) [text=red] {$\alpha_4$};
\node at (10.6,0.7) [text=red] {$\alpha_3$};
  \node at (10.6,-0.8) [text=red] {$\alpha_2$};
      \node at (10.6,-2) [text=red] {$\alpha_1$};
\end{tikzpicture}
\caption{The action $\tau^{2n_2}\sigma \tau^{2n_1}$  when $n_1 =2 $ and $n_2 = 3,$  starting from the slope $-1$. The solid dots indicate (lifts of) the $z$ basepoint and the hollow dots indicate (lifts of) the $w$ basepoint. The slope of  $\gamma_+$ is   $-n_2$ and the slope of  $\gamma_-$ is   $-(n_2+1).$  }
\label{fig:11diagram}
\end{figure}

Fix a lift $\wa$ (resp.~$\wb$) of the $\alpha$ (resp.~$\beta$) curve. Consider the final diagram in Figure \ref{fig:11diagram} and suppose $\wb$ travels from left to right.  Observe that in one iteration $\wb$ pass through $n_1+n_2-1$ consecutive lifts of $\alpha$ (after isotoping away bigons without basepoint if necessary); denote them $\alpha_1, \cdots, \alpha_{n_1+n_2-1}$ from bottom to top. To include all intersection points of $\wa \cap \wb$, $n_1+n_2-1$ iterations are needed. Following $\wb$, denote the diagram of each iteration by $H(s)$ for $1\leq s \leq n_1+n_2-1$; call it the $s$-th \emph{block}. Note that in $H(s),$ $\wa$ is identified with $\alpha_s.$

To further determine the marked points, by Definition \ref{def:gamma} each $\gamma_+$ is slope $-2n_2$, and viewing $\H^+([2n_1,1,2n_2])$ as $\H^+([2(n_1+1),-1,2(n_2+1)])$, each $\gamma_-$ is of slope $-2(n_2+1).$

\begin{lemma}
    For $1\leq s \leq n_1+n_2-1$, each block $H(s)$ corresponds to the diagram depicted in Figure \ref{fig:block}. 
\end{lemma}
\begin{proof}
    This can be readily read off from the final diagram in Figure \ref{fig:11diagram}.
\end{proof}
\begin{remark}
   To obtain Figure \ref{fig:block}, we are allowed to isotope the curves in the universal cover in the complement of the basepoints and other curves, and to move the basepoints around as long as they do not cross the curves. Since both these operations preserve the bigons, Figure \ref{fig:block} can be used to compute the knot Floer complex. We need to pay some extra attention to keep track of $\gamma_+$ and $\gamma_-.$
\end{remark}

\begin{figure}
\begin{minipage}{.5\linewidth}
\centering
\subfloat[The first block $H(1).$]{
\begin{tikzpicture}
\begin{scope}
    \begin{scope}[thin, black!0!white]
	  \draw  (-4, 0) -- (4,0);
      \end{scope}
\end{scope}
    \node at (0,0) {\includegraphics[scale=1.5]{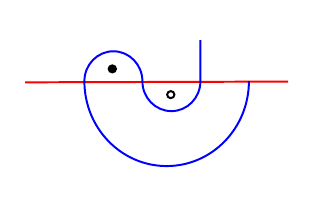}\label{subfig:block1}};
\end{tikzpicture}
}
\end{minipage}%
\begin{minipage}{.5\linewidth}
\centering
\subfloat[The last block $H(n_1+n_2-1).$]{
\begin{tikzpicture}
\begin{scope}
    \begin{scope}[thin, black!0!white]
	  \draw  (-5, 0) -- (5,0);
      \end{scope}
\end{scope}
    \node at (0,0) {\includegraphics[scale=1.5]{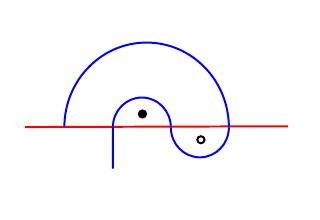}\label{subfig:block3}};
\end{tikzpicture}
}
\end{minipage}\\
\begin{minipage}{\linewidth}
\centering
\subfloat[The $s$-th block $H(s)$ with $2\leq s \leq n_1+n_2-2.$]{
\begin{tikzpicture}[scale=1.2]
\begin{scope}
    \begin{scope}[thin, black!0!white]
	  \draw  (-5, 0) -- (5,0);
      \end{scope}
\end{scope}
\node at (0.42,1.1) {$\cdots$};
\node at (-1.42,1.1) {$\cdots$};
\node at (-0.59,-1.05) {$\cdots$};
\node at (1.4,-1.04) {$\cdots$};

 \draw [decorate,decoration={brace,amplitude=4pt},xshift=0 cm,yshift=0.5pt]
      (-0.15,1.25) -- (0.9,1.25) node [midway,above,xshift=0.5cm,yshift=2pt] {\tiny$\max\{n_2-s,0\}$};
     \draw [decorate,decoration={brace,amplitude=4pt},xshift=0 cm,yshift=0.5pt]
    (-2,1.25) -- (-0.95,1.25) node [midway,above,xshift=-0.2cm,yshift=2pt] {\tiny$\max\{n_1-s,0\}$};

 \draw [decorate,decoration={brace,amplitude=4pt},xshift=0 cm,yshift=0.5pt]
      (-1.7,0.45) -- (0.7,0.45) node [midway,above,xshift=-0.1cm,yshift=1pt] {\tiny$\min\{n_1,s\} + \min\{n_2,s\} -s$};
\draw [decorate,decoration={brace,amplitude=4pt,mirror},xshift=0 cm,yshift=-0.5pt]
      (-0.85,-0.36) -- (1.59,-0.36) node [midway,below,xshift=-0.1cm,yshift=-1pt] {\tiny$\min\{n_1,s\} + \min\{n_2,s\} -s$};

     \draw [decorate,decoration={brace,amplitude=4pt,mirror},xshift=0 cm,yshift=-0.5pt]
      (0.89,-1.1) -- (1.9,-1.1) node [midway,below,xshift=0 cm,yshift=-2pt] {\tiny$\max\{s-n_1,0\}$};
     \draw [decorate,decoration={brace,amplitude=4pt,mirror},xshift=0 cm,yshift=-0.5pt]
    (-1.15,-1.1) -- (-0.1,-1.1) node [midway,below,xshift=-0.2cm,yshift=-2pt] {\tiny$\max\{s-n_2,0\}$};
    \node at (-2.1,-0.4) [text=OliveGreen] {$\gamma_-$};
    \node at (0.1,-0.2) [text=OliveGreen] {$\gamma_-$};
 \node at (-0.7,0.2) [text=Brown] {$\gamma_+$};
\node at (2,-0.2) [text=OliveGreen] {$\gamma_-$};
 \node at (1,0.2) [text=Brown] {$\gamma_+$};
    \node at (0,0) {\includegraphics[scale=1.8]{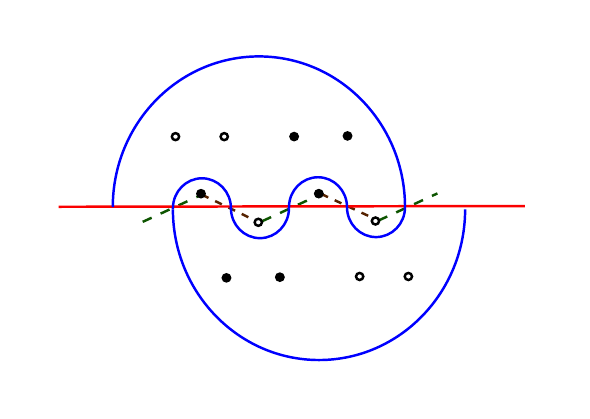}\label{subfig:block2}};
\end{tikzpicture}
}
\end{minipage}%
\caption{}
\label{fig:block}
\end{figure}
We now abuse the notation and denote the complex generated by $H(s)$ in $\CFKi(S^3,K^{+}_{n_1,n_2})$ also by $H(s)$, for $1 \leq s \leq n_1+n_2-1.$ Let the chain complex $G(1) = \CFKi(S^3,K^{+}_{n_1,n_2})$, and $G(i+1)=G(i)/(H(i)\cup H(n_1+n_2-i))$ for $1\leq i \leq \floor{\frac{n_1+n_2-1}{2}}.$ Observe that $H(s)$ and $H(n_1+n_2-s)$ are both subcomplexes in $G(s).$ 

We note that although suppressed from the notation, both $G(s)$ and $H(s)$ depend on $s,n_1,n_2$ and $k.$

\begin{definition}\label{def:ds}
Define a filtered  chain complex $D_s$ over $\F[U,U^{-1}]$ for $s \geq 1$ to be the complex generated by $x_i$ with $0\leq i \leq 2s$ and $y$,  where differentials are given by
\begin{align*}
    \partial x_{2i+1} &= x_{2i} + x_{2i+2}  \hspace{2em} 0 \leq i \leq s-1 \\
    \partial x_{2i} &= y  \hspace{6em} 0 \leq i \leq s 
    \intertext{with the filtration shifts}
    \dij (x_{2i+1}, x_{2i}) &= (1,0)\\  
    \dij (x_{2i+1}, x_{2i+2}) &= (0,1) \hspace{5em} 0 \leq i \leq s-1 \\
    \dij (x_{2i}, y) &= (i,s-i) \hspace{3.5em} 0 \leq i \leq s.
\end{align*}
We say $D_s$ is supported in Maslov grading $a$ and filtration level $(i,j)$ if $y$ is supported in Maslov grading $a$ and filtration level $(i,j)$.
\end{definition}
Alternatively, $D_s$ is the complex $C_s$ with the addition of a generator $y$ and the differentials from $x_{2i}$ to $y$ for $0\leq i \leq s$ with above filtration shifts. 
The complex $D(3)$ is shown below as an example.
\[
\xymatrix{
x_{0} \ar[ddd] & x_1 \ar[d] \ar[l]  & & \\
  & x_2 \ar[ldd] & x_3 \ar[d] \ar[l] & \\
  & & x_4 \ar[lld] & x_5 \ar[d] \ar[l]\\
  y & & & x_6 \ar[lll]
}
\]

\begin{lemma}\label{le:hs}
    When $s \leq \min\{n_1,n_2, \floor{\frac{n_1+n_2-1}{2}}\},$ there exists a filtered change of basis $T$ of $G(s),$ such that in the image of $T$,   $H(s)$   becomes a  summand, and moreover    $H(s) \cong D_s$. 
\end{lemma}
\begin{proof}
    Consider two consecutive blocks $H(s)$ and $H(s+1)$ in $G(s)$ for $1\leq s \leq \min\{n_1,n_2, \floor{\frac{n_1+n_2-1}{2}}\}.$  See Figure \ref{fig:consecutive}. Denote the intersection points of $\wb \cap \wa$ by $x_0, \cdots, x_{2s}, y , x'_{0},\cdots, x'_{2s'}, y'$ in order, where $s'= \min\{n_1,s+1\} + \min\{n_2,s+1\} -s-1$, which is the number of small inner arcs in each half plane in $H(s+1).$ There are two cases to discuss: when $s= \min \{n_1, n_2\}$ or $s<\min \{n_1, n_2\}.$
    \begin{itemize}
        \item Suppose $s= \min \{n_1, n_2\}$, then $s' = s.$ Observe that the only arrows connected to $H(s)$ are given by 
        \[
        \partial x'_{2i} = y + y'  \hspace{3em} 0 \leq i \leq s.
        \]
        Perform change of basis 
        \begin{align*}
            x'_{i} \mapsto x'_{i} + x_{i} \hspace{3em} 0 \leq i \leq 2s 
        \end{align*}
        splits off $H(s)$ as a summand. Since the filtration shifts are
        \begin{align}
        \dij (x_{2i}, y) &= (i,s-i)  \\
            \dij (x'_{2i}, y) &= (i,s-i) + (\max \{n_1 -s-1, 0\}, \max \{n_2 -s-1, 0\})   \label{eq:filshift1}               \end{align}     
        for $0 \leq i \leq s,$ this change of basis is filtered. 
        \item Suppose $s< \min \{n_1, n_2\}$, then $s' = s+1.$ Similarly, the only arrows connected to $H(s)$ are given by 
        \[
        \partial x'_{2i} = y + y'  \hspace{3em} 0 \leq i \leq s+1.
        \]
        Perform change of basis 
        \begin{align*}
            x'_{i} &\mapsto x'_{i} + x_{i} \hspace{3em} 0 \leq i \leq 2s \hspace{2em} \  \\ 
            x'_{2(s+1)} &\mapsto x'_{2(s+1)} + x_{2s}
        \end{align*}
        splits off $H(s)$ as a summand. Since the filtration shifts are
        \begin{align}
        \dij (x_{2i}, y) &= (i,s-i)   \hspace{8 em}  \text{for } 0 \leq i \leq s,  \quad \text{and} \\
        \dij (x'_{2i}, y) &= (i,s+1-i) + (n_1 -s-1, n_2 -s-1) \hspace{0.5em}     \end{align}     
        for $0 \leq i \leq s+1,$ this change of basis is filtered. 
    \end{itemize}
    And it is clear from the diagram that each $H(s)$ generates a $D_s$ summand (after quotienting out the top outmost arc).
\end{proof}
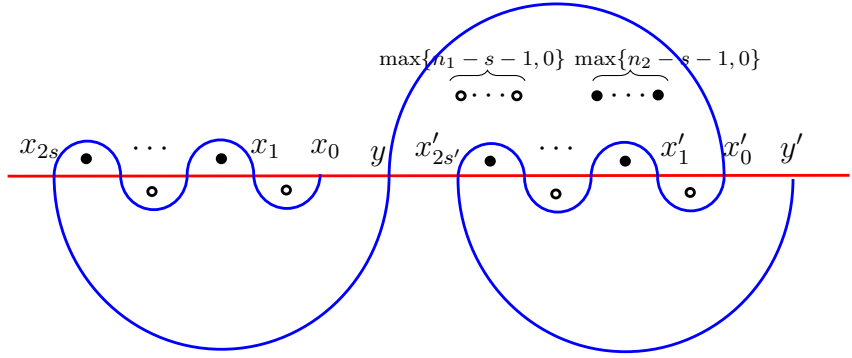
\begin{figure}[htb!]
   \centering
\begin{tikzpicture}
\node at (2.86,1.1) {$\cdots$};
\node at (1,1.1) {$\cdots$};

\node at (-1.2,0.4) {$x_0$};
\node at (-2 ,0.4) {$x_1$};
\node at (-3.5 ,0.4) {$\cdots$};
\node at (-5 ,0.4) {$x_{2s}$};
\node at (-0.5 ,0.3) {$y$};

\node at (5 ,0.4) {$y'$};
\node at (4.28 ,0.4) {$x'_0$};
\node at (3.45 ,0.4) {$x'_1$};
\node at (1.9 ,0.4) {$\cdots$};
\node at (0.3 ,0.4) {$x'_{2s'}$};

\draw [decorate,decoration={brace,amplitude=4pt},xshift=0cm,yshift=0pt]
      (0.45,1.3) -- (1.45,1.3) node [midway,above,xshift=-0.2 cm] {\tiny$\max \{n_1 -s-1, 0\}$};
      \draw [decorate,decoration={brace,amplitude=4pt},xshift=0cm,yshift=0pt]
      (2.35,1.3) -- (3.35,1.3) node [midway,above,xshift=0.5 cm] {\tiny$\max \{n_2 -s-1, 0\}$};

    \node at (0,0) {\includegraphics[scale=1.5]{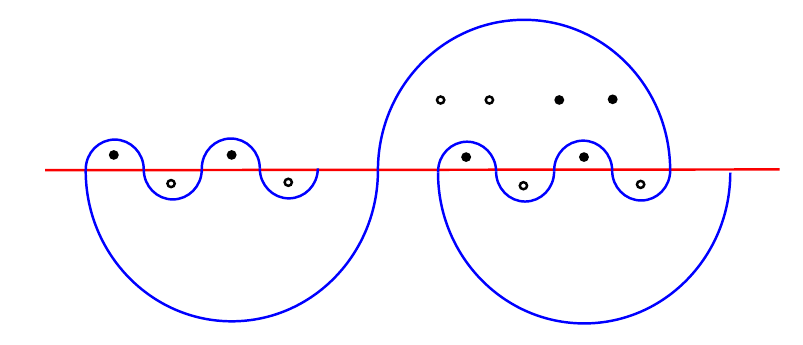}};
\end{tikzpicture}
\caption{Two consecutive blocks $H(s)$(left) and $H(s+1)$(right) in $G(s)$ for $1\leq s \leq \min\{n_1,n_2, \floor{\frac{n_1+n_2-1}{2}}\}.$ There could be other basepoints (at most more of each kind) in $H(s+1)$ in the lower half plane.}
\label{fig:consecutive}
\end{figure}
\begin{lemma}\label{le:hminus}
    When $s \leq \min\{n_1,n_2, \floor{\frac{n_1+n_2-1}{2}}\},$ there exists a filtered change of basis $T'$ of $G(s),$ such that in the image of $T'$,   $H(n_1+n_2-s)$   becomes a  summand, and moreover    $H(n_1+n_2-s) \cong D_s$. 
\end{lemma}
\begin{proof}
    This follows from Lemma \ref{le:hs} and the symmetry of the knot Floer complex. Equivalently one can run the similar argument as in the proof of Lemma \ref{le:hs} again.
\end{proof}

\begin{lemma}\label{le:simul}
    When $s \leq \min\{n_1,n_2\} -1,$ there exists a filtered change of basis such that $G(s)\cong G(s+1) \oplus H(s) \oplus H(n_1+n_2-s) \cong G(s+1) \oplus 2D_s.$  
\end{lemma}
\begin{proof}
    When $s \leq \min\{n_1,n_2\} -1,$ $H(s)\cap H(n_1+n_2-1) = \varnothing.$ By Lemma \ref{le:hs} and \ref{le:hminus}, one simply apply the change of basis $T'\circ T$, and in the resulting complex $H(s)$ and $ H(n_1+n_2-1)$ both become summands.
\end{proof}
These lemmas allow us to completely determine the complex $\CFKi(S^3,K^{+}_{n_1,n_2}).$ Note that due to \eqref{eq:relation1},   we may assume $n_1 \leq n_2.$ Therefore from now on we only consider  knots of the form $K^{+}_{n,n + k}$ with $n>0$ and $k\geq 0$. 
\begin{proposition}\label{prop:n1n2} For $n>0$ and $k\geq 0,$ we have
    \[
    \CFKi(S^3,K^{+}_{n,n + k})\cong G(n) \oplus 2\left(\bigoplus^{n-1}_{s=1}D_s \right).
    \]
\end{proposition}
\begin{proof}
    This follows from Lemma \ref{le:simul}. Starting from $G(1) \cong \CFKi(S^3,K^{+}_{n_1,n_2}),$ we keep splitting off pairs of summands of $D_s$ when  $s \leq \min\{n_1,n_2\} -1.$
\end{proof}
Therefore we need only to determine the quotient complex $G(n) \cong \bigcup_{j=0}^{k} H(n+j)$ in $\CFKi(S^3,K^{+}_{n,n + k}).$ Note that the top outmost arc in $H(n)$ and the bottom  outmost arc in $H(n+k)$ are quotient out. We discuss several cases for different $k.$
\begin{itemize}
    \item When $k=0,$ $G(n) \cong C_n.$
    \item When $k=1,$ we can use either Lemma \ref{le:hs} or \ref{le:hminus} to split off a $D_n$ summand. The remaining summand $G(n)/D_n \cong C_n.$
    \item When $k=2,$ we can use Lemma \ref{le:hs} and \ref{le:hminus} to split off a pair of $D_n$ summands. The remaining summand $G(n+1) \cong C_n.$
    \item When $k>3,$ we again can use Lemma \ref{le:hs} and \ref{le:hminus} to split off a pair of $D_n$ summands. The remaining summand is $G(n+1).$ 
\end{itemize}
\begin{definition}\label{def:cnk}
For $n>0, k\geq 2$, define  $C_{n,k}$ to be the complex generated by $\{x^{(j)}_{i} ~|~ 1\leq j \leq k-1, 0\leq i \leq 2n \} \cup \{ y_j ~|~ 1\leq i \leq k-2\}$. For the simplicity for the following let $y_0 = y_{k-1} = 0$. For $1 \leq j \leq k-1$, the differentials are given by 
\begin{align*}
     \partial x^{(j)}_{2i+1} &= x^{(j)}_{2i} + x^{(j)}_{2i+2} \hspace{2em} 0 \leq i \leq n-1 \\
    \partial x^{(j)}_{2i} &= y_{j} + y_{j-1}  \hspace{3em} 0 \leq i \leq n
    \intertext{with the filtration shifts}
    \dij (x^{(j)}_{2i+1},x^{(j)}_{2i} ) &= (1,0)\\
    \dij (x^{(j)}_{2i+1},x^{(j)}_{2i+2} ) &= (0,1)  \hspace{7em} 0 \leq i \leq n-1\\
    \dij (x^{(j)}_{2i},y_{j} ) &= (i+j,n-i)\\
    \dij (x^{(j)}_{2i},y_{j-1} ) &= (i,k-j+n-i)  \hspace{2em} 0 \leq i \leq n.
\end{align*} 
Note that $C_{n,2}=C_n.$ See Figure \ref{fig:adblock}\subref{subfig:cnk} for an example when $n=2$ and $k=4$.
\end{definition}
\begin{lemma}\label{le:cnk}
    When $k\geq 2,$ in $\CFKi(S^3,K^{+}_{n,n + k})$ the quotient complex $G(n+1) \cong C_{n,k}.$
\end{lemma}
\begin{proof}
    The quotient complex $G(n+1) \cong \bigcup_{j=1}^{k-1} H(n+j)$ in $\CFKi(S^3,K^{+}_{n,n + k})$. We have drawn the block $H(n+j)$ for $1\leq j \leq k-1$ in Figure \ref{fig:adblock}\subref{subfig:adblock}. Again   the top outmost arc in $H(n+1)$ as well as the bottom outmost arc in $H(n+k-1)$ are quotiented out. Denoting the generators in $H(n+j)$  by $x^{(j)}_{2n},  \cdots x^{(j)}_{0}, y_j$ from left to right in order, note that for each $j$ the $2n+1$ generators $x^{(j)}_{2n},  \cdots x^{(j)}_{0}$ form a staircase, and for each $0\leq i \leq n$ the generator $x^{(j)}_{2i}$ has a differential to $y_j$ and $y_{j-1}$ (taking $y_0 = y_{k-1} =0$).  The filtration shifts can be readily computed by counting the number of basepoints in the bigons.
\end{proof}
\begin{figure}[htb!]
\begin{minipage}{.5\linewidth}
\subfloat[]{
   \centering
\begin{tikzpicture}\label{subfig:adblock}
\node at (-0.5,1.3) {$\cdots$};
\node at (0.46,-1.1) {$\cdots$};
 \draw [decorate,decoration={brace,amplitude=4pt},xshift=0cm,yshift=0pt]
      (-1,1.5) -- (0,1.5) node [midway,above,xshift=0 cm] {$k-j$};
       \draw [decorate,decoration={brace,amplitude=4pt},xshift=0cm,yshift=0pt]
      (-1.7,0.5) -- (0.6,0.5) node [midway,above,xshift=0 cm] {$n$};
       \draw [decorate,decoration={brace,amplitude=4pt,mirror},xshift=0cm,yshift=0pt]
      (-0.05,-1.25) -- (0.85,-1.25) node [midway,below,xshift=0 cm] {$j$};

    \node at (0,0) {\includegraphics[scale=1.5]{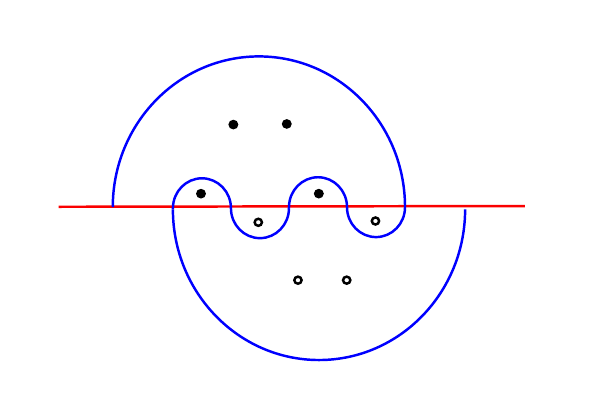}};
\end{tikzpicture}
}
\end{minipage}%
\begin{minipage}{.5\linewidth}
\subfloat[]{
    \centering
    \begin{tikzpicture}\label{subfig:cnk}
      \begin{scope}[thin, black!40!white]  
      \foreach \i in {-1,...,5}
   {
   \draw  (-0.75, 0.5+\i) -- (5.75, 0.5+\i);
     \draw  (0.5+\i, -0.75) -- (0.5+\i, 5.75);
       }
      \end{scope}
         \filldraw (5, 0) circle (2pt) node[] (){};
         \node at (5,-0.3) {\tiny$x^{(1)}_{4}$};
          \filldraw (5, 1) circle (2pt) node[] (){};
           \node at (5.3,0.7) {\tiny$x^{(1)}_{3}$};
           \filldraw (4, 1) circle (2pt) node[] (){};
              \node at (4,0.7) {\tiny$x^{(1)}_{2}$};
            \filldraw (4, 1.8) circle (2pt) node[] (){};
              \node at (4.3,1.5) {\tiny$x^{(1)}_{1}$};
             \filldraw (3, 1.8) circle (2pt) node[] (){};
               \node at (3.1,1.5) {\tiny$x^{(1)}_{0}$};
                \filldraw (4, 2.2) circle (2pt) node[] (){};
                    \node at (4.4,2.3) {\tiny$x^{(2)}_{4}$};
                   \filldraw (4, 3) circle (2pt) node[] (){};
                   \node at (4.4,3.3) {\tiny$x^{(2)}_{3}$};
                   \filldraw (3, 3) circle (2pt) node[] (){};
                                      \node at (3.4,3.3) {\tiny$x^{(2)}_{2}$};

                    \filldraw (0, 5) circle (2pt) node[] (){};
                     \node at (-0.3,5) {\tiny$x^{(3)}_{0}$};
               \filldraw (1, 5) circle (2pt) node[] (){};
                    \node at (0.7,5.3) {\tiny$x^{(3)}_{1}$};
           \filldraw (1, 4) circle (2pt) node[] (){};
            \node at (0.7,4) {\tiny$x^{(3)}_{2}$};
            \filldraw ( 1.8,4) circle (2pt) node[] (){};
                   \node at (1.6,4.3) {\tiny$x^{(3)}_{3}$};
             \filldraw ( 1.8,3) circle (2pt) node[] (){};
                         \node at (1.5,3.1) {\tiny$x^{(3)}_{4}$};
                \filldraw ( 2.2,4) circle (2pt) node[] (){};
                   \node at (2.6,4.3) {\tiny$x^{(2)}_{0}$};
                   \filldraw ( 3,4) circle (2pt) node[] (){};
                    \node at (3.4,4.3) {\tiny$x^{(2)}_{1}$};

  \filldraw (2.2, 0) circle (2pt) node[] (){};
                     \node at (1.9,0) {\small$y_1$};
    \filldraw (0, 2.2) circle (2pt) node[] (){};
      \node at (0,1.9) {\small$y_2$};

     \draw  (2.2, 0) -- (5, 0);
     \draw  (2.2, 0) -- (4, 1);
     \draw  (2.2, 0) -- (3, 1.8);
         \draw  (2.2, 0) -- (4, 2.2);
     \draw  (2.2, 0) -- (3, 3);
       \draw  (2.2, 0) -- (2.2, 4);

         \draw  (0,2.2) -- (0, 5);
     \draw  (0,2.2) -- (1, 4);
     \draw  (0,2.2) -- ( 1.8,3);
         \draw  (0,2.2) -- (4, 2.2);
     \draw  (0,2.2) -- (3, 3);
       \draw  (0,2.2) -- (2.2, 4);

         \draw  (1,5) -- (0, 5);
         \draw  (1,5) -- (1, 4);
         \draw  (1,4) -- (1.8, 4);
          \draw  (1.8,3) -- (1.8, 4);

          \draw (3,4) -- (2.2, 4);
            \draw (3,4) -- (3, 3);
             \draw (4,3) -- (3, 3);
             \draw (4,3) -- (4, 2.2);

                   \draw  (5,1) -- (5, 0);
         \draw  (5,1) -- (4, 1);
         \draw  (4,1) -- (4,1.8);
          \draw  (4,1.8) -- (3,1.8);
        
    \end{tikzpicture}
    }   
\end{minipage}
\caption{On the left, the  block $H(n+j)$ with $1\leq j\leq k-1$. On the right, the complex $C_{2,4}.$}
\label{fig:adblock}
\end{figure}
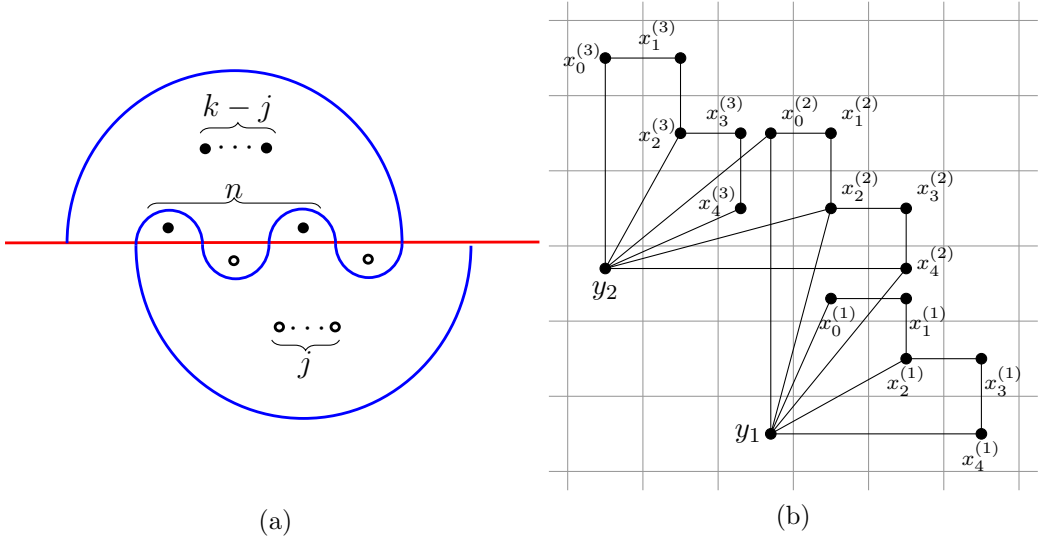
The following proposition describes the knot Floer complex $ \CFKi(S^3,K^{+}_{n,n+k})$ for $n>0, k\geq 0$. The $k=1$ case of Proposition \ref{prop:class1} is the content of \cite[Thoerem 3.3.14]{jhales}, which describes the knot Floer complex of the family of knots $K_n,$ which in turn implies Theorem \ref{thm:staircase}.  (To be precise, \cite[Thoerem 3.3.14]{jhales} describes the knot Floer complex of the mirror $-K_n$.)  
\begin{proposition}\label{prop:class1}
    For $n>0, k\geq 0$, up to homotopy equivalence the knot Floer complex $ \CFKi(S^3,K^{+}_{n,n+k})$  is given by the following. 
    \begin{itemize}
        \item When $k=0,$ 
        \[
        \CFKi(S^3,K^{+}_{n,n})\cong C_n \oplus 2\left(\bigoplus^{n-1}_{s=1}D_s \right);
        \]
        \item When $k=1,$ 
        \[
        \CFKi(S^3,K^{+}_{n,n+1})\cong  C_n \oplus D_n \oplus 2\left(\bigoplus^{n-1}_{s=1}D_s\right);
        \]
        \item When $k=2,$ 
        \[
        \CFKi(S^3,K^{+}_{n,n+2})\cong C_{n} \oplus 2\left(\bigoplus^{n}_{s=1}D_s\right);
        \]
        \item When $k\geq 2,$ 
        \[
        \CFKi(S^3,K^{+}_{n,n+k})\cong C_{n,k} \oplus 2\left(\bigoplus^{n}_{s=1}D_s\right).
        \]
    \end{itemize}
    Moreover, there exists a choice of basis, such that when $k=0,$ for each $1\leq s \leq n-1$ both $D_s$ are supported in filtration level $(f_{n,s},f_{n,s}).$ When $k>0,$ for $1\leq s \leq n$, each pair of $D_s$ is supported in filtration levels
    \begin{align*}
          &(f_{n,s},  f_{n,s}) + \left(\frac{(k-1)(k-2)}{2}, -(n-s+1)k +1\right) \hspace{2em} and\\
       &(f_{n,s},  f_{n,s}) + \left(-(n-s+1)k +1, \frac{(k-1)(k-2)}{2}\right), 
    \end{align*}
   respectively, except for when $k=1$ the single copy of $D_n$ is supported in $(0,0),$ where 
    \begin{align}
       f_{n,s} =  -\frac{(n-s)(n-s-1)}{2}.
    \end{align}
    Under this basis, each summand $D_s$ is supported in Maslov grading $-1$;  each $0$-graded generator has one $-$ marking and each $+1$-graded generator has one $+$ marking.
\end{proposition}
\begin{proof}
    The statements regarding the homotopy equivalence type of $ \CFKi(S^3,K^{+}_{n,n+k})$ follow from Proposition \ref{prop:n1n2}, Lemma \ref{le:cnk} and the discussion between. We are left with the statements regarding the Maslov grading, the marked basis and the filtration levels, which can be proved by examining  the process of splitting off $D_s$ summands more closely.

   Since we have that for $s\leq n,$
   \begin{align*}
        \dij (x^{(s)}_{2i}, y_{s-1}) &= (i,s+1-i) + (n -s, n + k -s) \\
        \dij (x^{(s)}_{2i}, y_s) &= (i,s+1-i) 
   \end{align*}
   the filtration shift from $y_s$ to $y_{s-1}$ (which supports $D_s$ and $D_{s-1}$ respectively) is  
   \begin{align}
       -(n -s, n + k -s)  \hspace{5em} s\leq n. \label{eq:suppshift}
   \end{align}
For a fixed basis $B$ of a complex $C$, we say $B$ is supported in the filtration level $(a,b)$ if the generator with the lowest $i$ (resp.~$j$) filtration in $B$ is in filtration level $a$ (resp.~$b$).

When $k=0,$ fix a basis such that the summand $C_n$ is supported in filtration level $(0,0).$ By \eqref{eq:filshift1} both copies of $D_{n-1}$ are supported in $(0,0).$ By \eqref{eq:suppshift}, for $1\leq s \leq n-1$ the copy of $D_s$ that comes from $H(s)$ is supported in filtration
\begin{align*}
    \left( -\sum_{\l=0}^{n-s-1} \l, -\sum_{\l=0}^{n-s-1} \l \right) = (f_{n,s},f_{n,s}).
\end{align*}
The filtration levels of the remaining complex follows from the symmetry of the knot Floer complex. When $k=1,$ fix a basis such that the single copy of $D_n$  is supported in $(0,0).$ (The summand $C_n$ is supported in $(0,0)$ as well.) The filtration levels in this case follows in the exact same way as above. 

When $k\geq 2,$ fix a basis such that the summand $C_{n,k}$ is supported in filtration level $(0,0).$ It is straightforward to check that $x^{(1)}_{0}$ is supported in the filtration level
\begin{align*}
    \left( \sum_{\l=1}^{k-2} \l, n \right) = \left(\frac{(k-1)(k-2)}{2},n \right).
\end{align*}
Then by \eqref{eq:filshift1} the copy of $D_n$ that comes from $H(n)$ is supported in
$(\frac{(k-1)(k-2)}{2},-k+1)$. The filtration levels of the remaining complex follows in the same way as above.

 For the statement regarding the Maslov grading, observe (for example, from Figure \ref{fig:consecutive}) that every $y_s$ generator in some $H(s)$ is supported in the same Maslov grading $t$, and every $x^{(s)}_0$ generator in some $H(s)$ is supported in the same Maslov grading $t+1.$  On the other hand, the homogeneous element $\sum_{s} x^{(s)}_0$ is a generator of $H_*(\CFKi(S^3,K^{+}_{n,n+k})) \cong HF^\infty(S^3)$ and therefore is supported in Maslov grading $0.$ It follows that $t=-1.$
   The statement regarding the marked basis can be readily read off by Figure \ref{fig:block} \subref{subfig:block2}, using Definition \ref{def:marked}.
\end{proof}
\subsection{Case $K^{-}([2n_1,1,2n_2])$ with $n_1,n_2>0$}\label{subsec:class2}
  Similarly let us write $K^{-}_{n_1,n_2}$ for the knot $K^{-}([2n_1,1,2n_2])$. We still consider the action $\tau^{2n_2}\sigma \tau^{2n_1}$ on the $(1,1)$ diagram, with the difference being the starting slope is $+1$. The resulting $(1,1)$ diagram is depicted in Figure \ref{fig:block'}\subref{subfig:block'0}. Compare with the final diagram in  Figure \ref{fig:11diagram}. The process for determining $\CFKi(S^3,K^{-}_{n_1,n_2})$ is almost identical to the process described in the previous section, so we will only include  the key steps, being much less elaborate.

To start, we can similarly define the $s$-th block $H'(s)$ for $1\leq s \leq n_1+n_2-1$; also let the chain complex $G'(1) = \CFKi(S^3,K^{-}_{n_1,n_2})$, and $G'(i+1)=G'(i)/(\wH'(i)\cup \wH'(n_1+n_2-i))$, where $\wH'(i) = H'(i) \setminus \{ \text{right end point} \} $ and $\wH'(n_1+n_2-i) = H'(n_1+n_2-i) \setminus \{ \text{left end point} \} $ for $1\leq i \leq \floor{\frac{n_1+n_2-1}{2}}.$  We add the apostrophe to differentiate from the complexes defined in the previous section. 

\begin{lemma}
    For $1\leq s \leq n_1+n_2-1$, each block $H'(s)$ corresponds to the diagram depicted in Figure \ref{fig:block'} \subref{subfig:block'2}. 
\end{lemma}

\begin{figure}[htb!]
    \centering
 \begin{minipage}{.5\linewidth}
 \centering
\subfloat[The $(1,1)$ diagram of $K^{-}_{n_1,n_2}$.]{
  \begin{tikzpicture}[scale=0.8]
  \begin{scope}[thin, black!0!white]
	  \draw  (-5, 0) -- (5,0);
\end{scope}
      \node at (0.81,0.52) {\includegraphics[scale=0.272]{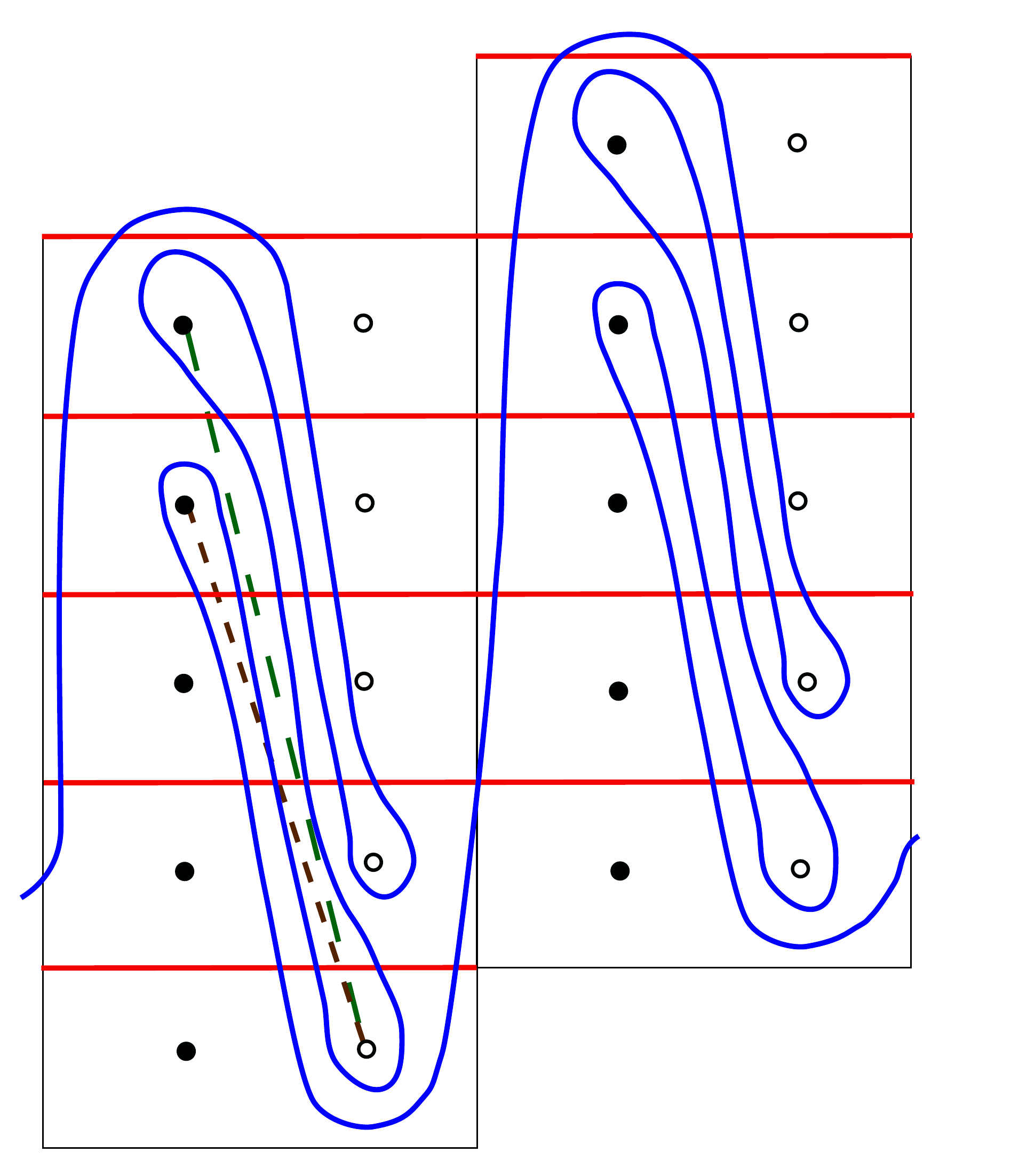}};
 \draw [decorate,decoration={brace,amplitude=4pt},xshift=0.5cm,yshift=0pt]
      (-0.21,-1.2) -- (-0.21,-2.6) node [midway,right,xshift=0.1cm,yshift=0.1cm] {\small$n_1$};
 \draw [decorate,decoration={brace,amplitude=4pt},xshift=0.5cm,yshift=0pt]
      (-.6,2.4) -- (-.6,-0.4) node [midway,right,xshift=0cm,yshift=5pt] {\small$n_2$};
       \draw [decorate,decoration={brace,amplitude=4pt,mirror},xshift=0.5cm,yshift=0pt]
      (-1.9,2.2) -- (-1.9,1) node [midway,left,xshift=0.05cm,yshift=5pt] {\small$n_1$};
       \draw [decorate,decoration={brace,amplitude=4pt,mirror},xshift=0.5cm,yshift=0pt]
      (-1.9,0) -- (-1.9,-2.7) node [midway,left,xshift=0cm] {\small$n_2$};
      \node [text=OliveGreen] at (-1,2) {$\gamma_{-}$};
         \node [text=Brown] at (-1.6,0.7) {$\gamma_{+}$};
  \end{tikzpicture}
  \label{subfig:block'0}}
\end{minipage}%
  \begin{minipage}{.5\linewidth}
  \centering
\subfloat[The first block $H'(1).$]{
\begin{tikzpicture}
    \begin{scope}[thin, black!0!white]
	  \draw  (-5, 0) -- (5,0);
\end{scope}
\node at (0.1,-1.16) {\tiny$\cdots$};
\node at (1.3,-1.16) {\tiny$\cdots$};
     \draw [decorate,decoration={brace,amplitude=4pt,mirror},xshift=0 cm,yshift=-0.5pt]
      (0.9,-1.25) -- (1.65,-1.25) node [midway,below,xshift=0 cm,yshift=-2pt] {\tiny$n_2 - 1$};
     \draw [decorate,decoration={brace,amplitude=4pt,mirror},xshift=0 cm,yshift=-0.5pt]
    (-0.3,-1.25) -- (0.45,-1.25) node [midway,below,xshift=0.15 cm,yshift=-2pt] {\tiny$n_1 - 1$};
    \node at (0,0) {\includegraphics[scale=1.8]{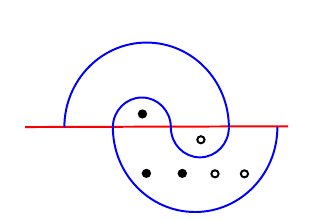}\label{subfig:block_1}};
\end{tikzpicture}
}
\end{minipage}\\
\begin{minipage}{\linewidth}
\centering
\subfloat[The $s$-th block $H'(s)$ with $1\leq s \leq n_1+n_2-1.$]{
\begin{tikzpicture}[scale=1.2]
\begin{scope}
    \begin{scope}[thin, black!0!white]
	  \draw  (-5, 0) -- (5,0);
      \end{scope}
\end{scope}
\node at (0.42,1.1) {$\cdots$};
\node at (-1.42,1.1) {$\cdots$};
\node at (-0.59,-1.05) {$\cdots$};
\node at (1.4,-1.04) {$\cdots$};

 \draw [decorate,decoration={brace,amplitude=4pt},xshift=0 cm,yshift=0.5pt]
      (-0.15,1.25) -- (0.9,1.25) node [midway,above,xshift=0.5cm,yshift=2pt] {\tiny$\max\{s-n_1,0\}$};
     \draw [decorate,decoration={brace,amplitude=4pt},xshift=0 cm,yshift=0.5pt]
    (-2,1.25) -- (-0.95,1.25) node [midway,above,xshift=-0.2cm,yshift=2pt] {\tiny$\max\{s-n_2,0\}$};

 \draw [decorate,decoration={brace,amplitude=4pt},xshift=0 cm,yshift=0.5pt]
      (-1.7,0.45) -- (0.7,0.45) node [midway,above,xshift=-0.1cm,yshift=1pt] {\tiny$\min\{n_1,s\} + \min\{n_2,s\} -s$};
\draw [decorate,decoration={brace,amplitude=4pt,mirror},xshift=0 cm,yshift=-0.5pt]
      (-0.85,-0.36) -- (1.59,-0.36) node [midway,below,xshift=-0.1cm,yshift=-1pt] {\tiny$\min\{n_1,s\} + \min\{n_2,s\} -s$};

     \draw [decorate,decoration={brace,amplitude=4pt,mirror},xshift=0 cm,yshift=-0.5pt]
      (0.89,-1.1) -- (1.9,-1.1) node [midway,below,xshift=0 cm,yshift=-2pt] {\tiny$\max\{n_2-s,0\}$};
     \draw [decorate,decoration={brace,amplitude=4pt,mirror},xshift=0 cm,yshift=-0.5pt]
    (-1.15,-1.1) -- (-0.1,-1.1) node [midway,below,xshift=-0.2cm,yshift=-2pt] {\tiny$\max\{n_1-s,0\}$};
   \node at (-2.1,-0.4) [text=OliveGreen] {$\gamma_-$};
    \node at (0.1,-0.2) [text=OliveGreen] {$\gamma_-$};
 \node at (-0.7,0.2) [text=Brown] {$\gamma_+$};
\node at (2,-0.2) [text=OliveGreen] {$\gamma_-$};
 \node at (1,0.2) [text=Brown] {$\gamma_+$};
    \node at (0,0) {\includegraphics[scale=1.8]{block2}\label{subfig:block'2}};
\end{tikzpicture}
}
\end{minipage}
\caption{}
\label{fig:block'}
\end{figure}

Again due to \eqref{eq:relation1}, we have $K^{-}_{n_1,n_2} = K^{-}_{n_2,n_1}$.  Therefore from now on we only consider  knots of the form $K^{-}_{n,n + k}$ with $n>0, k\geq 0$. Similarly as before, we split off $D_s$ summands iteratively until we obtain a small complex. This process in the case of $K^{-}_{n,n + k}$ is somewhat more straightforward  than that in the previous section. We depicted the two consecutive blocks in Figure \ref{fig:consecutive_}, but in fact we  need only to consider the first block $H'(s)$.

\begin{lemma} \label{le:simul'}
    For $n>0, k \geq 0$, 
    \begin{align*}
        \CFKi(S^3, K^{-}_{n,n+k}) \cong \begin{cases}
            G(n-1) \oplus 2(\bigoplus_{s=1}^{n-1}D_s)   &\hspace{1em} k=0  \\
            G(n) \oplus 2(\bigoplus_{s=1}^{n}D_s) &\hspace{1em} k\geq 1 
        \end{cases}
    \end{align*}
\end{lemma}
\begin{proof}
    For $1\leq s \leq n$, label the generators in $H'(s)$ from left to right by $y_s, x^{(s)}_{2s}, \cdots, x^{(s)}_{0}, y_{s+1} $ where $y_{s+1} =  H'(s) \cap H'(s+1).$ The differentials are
    \begin{align*}
        \d x^{(s)}_{2i + 1} &=  x^{(s)}_{2i } + x^{(s)}_{2i +2 }  \hspace{3em}  0\leq i \leq s-1\\
        \d x^{(s)}_{2i } &= y_s + y_{s+1} \hspace{5.5em}  0\leq i \leq s
        \intertext{with filtration shifts}
        \dij(x^{(s)}_{2i }, y_s) &= (i, s-i)  \\
        \dij(x^{(s)}_{2i }, y_{s+1}) &= (i, s-i) + (n+k-s, n-s)
        \intertext{for $ 0\leq i \leq s.$ Therefore  performing the filtered change of basis}
        y_s &\mapsto y_s + y_{s+1}
    \end{align*}
    splits off $\wH'(s) = H'(s) \setminus \{y_{s+1}\} \cong D_s$ as a summand.  We can split off $\wH'(s)$ and $\wH'(2n+k-s)$ for $1\leq s \leq n-1$ at the same time since $\wH'(s) \cap \wH'(2n+k-s) = \varnothing$. Due to the symmetry of the knot Floer complex $\wH'(2n+k-s)\cong D_s.$ When $k\geq 1,$ we can further split off $\wH'(n) \cup \wH'(n+k) \cong 2 D_n.$
\end{proof}

\begin{figure}[htb!]
   \centering
\begin{tikzpicture}
\node at (1.85,-1.2) {$\cdots$};
\node at (3.72,-1.2) {$\cdots$};

\node at (-3.5,-1.2) {$\cdots$};
\node at (-1.65,-1.2) {$\cdots$};

\node at (-0.9,0.4) {$x^{(s)}_0$};
\node at (-1.8 ,0.4) {$x^{(s)}_1$};
\node at (-3.5 ,0.4) {$\cdots$};
\node at (-5 ,0.4) {$x^{(s)}_{2s}$};

\node at (-0.5 ,-0.2) {$y_{s+1}$};
\node at (-5.6 ,-0.2) {$y_{s}$};
\node at (5.5 ,-0.2) {$y_{s+2}$};
\node at (4.7 ,0.4) {$x^{(s+1)}_0$};
\node at (3.7 ,0.4) {$x^{(s+1)}_1$};
\node at (1.9 ,0.4) {$\cdots$};
\node at (0.5 ,0.42) {$x^{(s+1)}_{2s'}$};

\draw [decorate,decoration={brace,amplitude=4pt,mirror},xshift=0cm,yshift=0pt]
      (1.35,-1.3) -- (2.35,-1.3) node [midway,below,xshift=0.1 cm] {\tiny$n+k-s-1$};
      \draw [decorate,decoration={brace,amplitude=4pt,mirror},xshift=0cm,yshift=0pt]
      (3.15,-1.3) -- (4.15,-1.3) node [midway,below,xshift=0.15 cm] {\tiny $n -s-1$};
    \node at (0,0) {\includegraphics[scale=1.5]{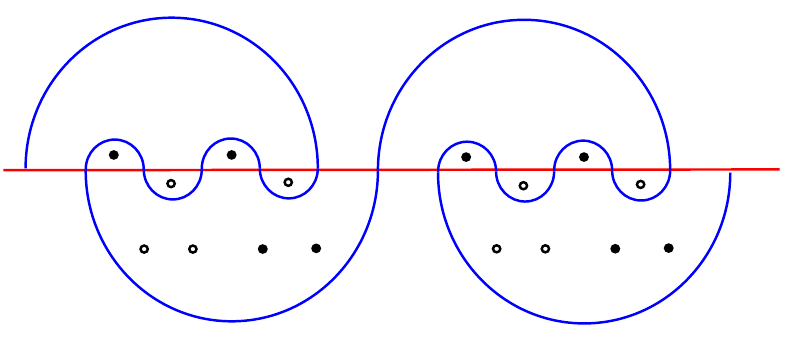}};

    \draw [decorate,decoration={brace,amplitude=4pt,mirror},xshift=0cm,yshift=0pt]
      (-4,-1.3) -- (-3,-1.3) node [midway,below,xshift=-0 cm] {\tiny$n+k-s$};
      \draw [decorate,decoration={brace,amplitude=4pt,mirror},xshift=0cm,yshift=0pt]
      (-2.2,-1.3) -- (-1.2,-1.3) node [midway,below,xshift=-0.1 cm] {\tiny $n -s$};
    \node at (0,0) {\includegraphics[scale=1.5]{consecutive_}};
\end{tikzpicture}
\caption{Two consecutive blocks $H'(s)$(left) and $H'(s+1)$(right) in $G'(s)$ for $1\leq s \leq n.$ There is either one basepoint or one of each kind in the upper half plane of $H'(n+1)$ when $s=n$.}
\label{fig:consecutive_}
\end{figure}
Recall 
    that $C_0$ is the complex generated by one element. 
\begin{definition}\label{def:c'nk}
    For $n>0, k\geq 2,$ let $C'_{n,k}$ be the complex generated by $\{x^{(j)}_{i} ~|~ 1\leq j \leq k-1, 0\leq i \leq 2n \} \cup \{ y_j ~|~ 1\leq j \leq k \}$ with differentials 
    \begin{align*}
          \partial x^{(j)}_{2i+1} &= x^{(j)}_{2i} + x^{(j)}_{2i+2} \hspace{2em} 0 \leq i \leq n-1 \\
    \partial x^{(j)}_{2i} &= y_{j} + y_{j+1}  \hspace{3em} 0 \leq i \leq n
    \intertext{with the filtration shifts}
    \dij (x^{(j)}_{2i+1},x^{(j)}_{2i} ) &= (1,0)\\
    \dij (x^{(j)}_{2i+1},x^{(j)}_{2i+2} ) &= (0,1)  \hspace{7em} 0 \leq i \leq n-1\\
    \dij (x^{(j)}_{2i},y_{j} ) &= (i,n-i + j)\\
    \dij (x^{(j)}_{2i},y_{j+1} ) &= (i + k-j, n-i)  \hspace{2em} 0 \leq i \leq n.
    \end{align*}
\end{definition}
The complex $C'_{1,3}$ is shown below as an example.
\[
\xymatrix{
&x^{(2)}_0 \ar[ddd]_{3} \ar[ld] & x^{(2)}_1 \ar[l] \ar[d] &&\\
y_3&&x^{(2)}_2 \ar[ll] \ar[ldd] &&\\
&&& x^{(1)}_{0} \ar[dd] \ar[lld] & x^{(1)}_{1} \ar[l] \ar[d] \\
&y_2 &&& x^{(1)}_{2}  \ar[lll]_{3}  \ar[ld] \\
&&& y_1 &
}
\]

\begin{proposition}\label{prop:class2}
    For $n>0, k\geq 0$, up to homotopy equivalence the knot Floer complex $ \CFKi(S^3,K^{-}_{n,n+k})$  is given by the following. 
    \begin{itemize}
        \item When $k=0,$ 
        \[
        \CFKi(S^3,K^{+}_{n,n})\cong C_0 \oplus D_n \oplus 2\left(\bigoplus^{n-1}_{s=1}D_s \right);
        \]
        \item When $k=1,$ 
        \[
        \CFKi(S^3,K^{+}_{n,n+1})\cong  C_0 \oplus 2\left(\bigoplus^{n}_{s=1}D_s\right);
        \]
        \item When $k\geq 2,$ 
        \[
        \CFKi(S^3,K^{+}_{n,n+k})\cong C'_{n,k} \oplus 2\left(\bigoplus^{n}_{s=1}D_s\right).
        \]
    \end{itemize}
    Moreover, there exists a choice of basis, such that $C_0$ is always supported in the filtration level $(0,0)$ and each pair of $D_s$ for $1\leq s \leq n$ is supported in filtration levels
    \begin{align*}
        &\left(  \big(\frac{k+1}{2} + n -s \big)k + \frac{(n-s)(n-s+1)}{2} ,  \frac{(n-s)(n-s+1)}{2}   \right) \qquad \text{and}\\
        &\left(   \frac{(n-s)(n-s+1)}{2} ,  \frac{(n-s)(n-s+1)}{2} +  \big(\frac{k+1}{2} + n -s \big)k  \right)
    \end{align*}
    respectively. Under this basis, each $C_0$ and $D_s$ summand are supported in Maslov grading $0$; each $+1$-graded generator has one $-$ marking and each $+2$-graded generator has one $+$ marking.
\end{proposition}
\begin{proof}
      From the proof of Lemma \ref{le:simul'} we see that $y_s$ and $y_{s+1}$ has a filtration difference of 
    \begin{align}
       \label{eq:prooffil'}  (n+k-s, n-s)  \hspace{3em} \text{for } \quad 1\leq s \leq n.
    \end{align} Each $D_s$ is supported by $y_s$ for $1\leq s \leq n.$ For the first two cases, fix a basis such that $C_0$ is supported in $(0,0)$.  For the last case, fix a basis such that $C'_{n,k}$ is supported in $(0,0)$, and then it is straightforward to check that $y_1 \in C'_{n,k}$ has filtration level $((k-1)k/2,0)$. By \eqref{eq:prooffil'} we can determine the support of all the $D_s$ that come from $\wH'(s)$ for $1\leq s \leq n.$ The filtration levels of the remaining complex  follows from the symmetry of the knot Floer complex.

  Each  $y_s$ generator has the same Maslov grading $t$, for some integer $i$, each  $x^{(s)}_{2i}$ generator  has the same Maslov grading $t+1$ and each   $x^{(s)}_{2i+1}$ generator has the same Maslov grading $t+2$. Since the homology is supported in some $y_s$, we have $t=0.$ The statement regarding the marked basis can be read off from Figure \ref{fig:block'}\subref{subfig:block'2}.
\end{proof}

\subsection{Case $K^{+}([2n_1,1,-2n_2])$ with $n_1 >0,n_2>1$}\label{subsec:class3}
For suitable orientations of $\aa$ and $\bb$ curve in the $(1,1)$ diagram, the induced orientation for each bigon is the same. See for example, Figure \ref{fig:class3_1} and \ref{fig:class3_2}. By \cite[Theorem 1.2]{GLV}, $K^{+}([2n_1,1,-2n_2])$ is a (negative) $(1,1)$ L-space knot. Therefore in order to pin down its knot Floer complex, it suffices to record  the length of  (say)  horizontal arrows.
\begin{figure}[htb!]
    \centering
    \begin{minipage}{0.5\linewidth}
     \centering
       \subfloat[The $(1,1)$ diagram when $n_1 = 2 , n_2 = 4 $.]{
       \begin{tikzpicture}
         \begin{scope}[thin, black!0!white]
          \draw  (-4, 0) -- (5.5, 0);
      \end{scope}
                 \node at (0.8,0) {\includegraphics[scale=0.3]{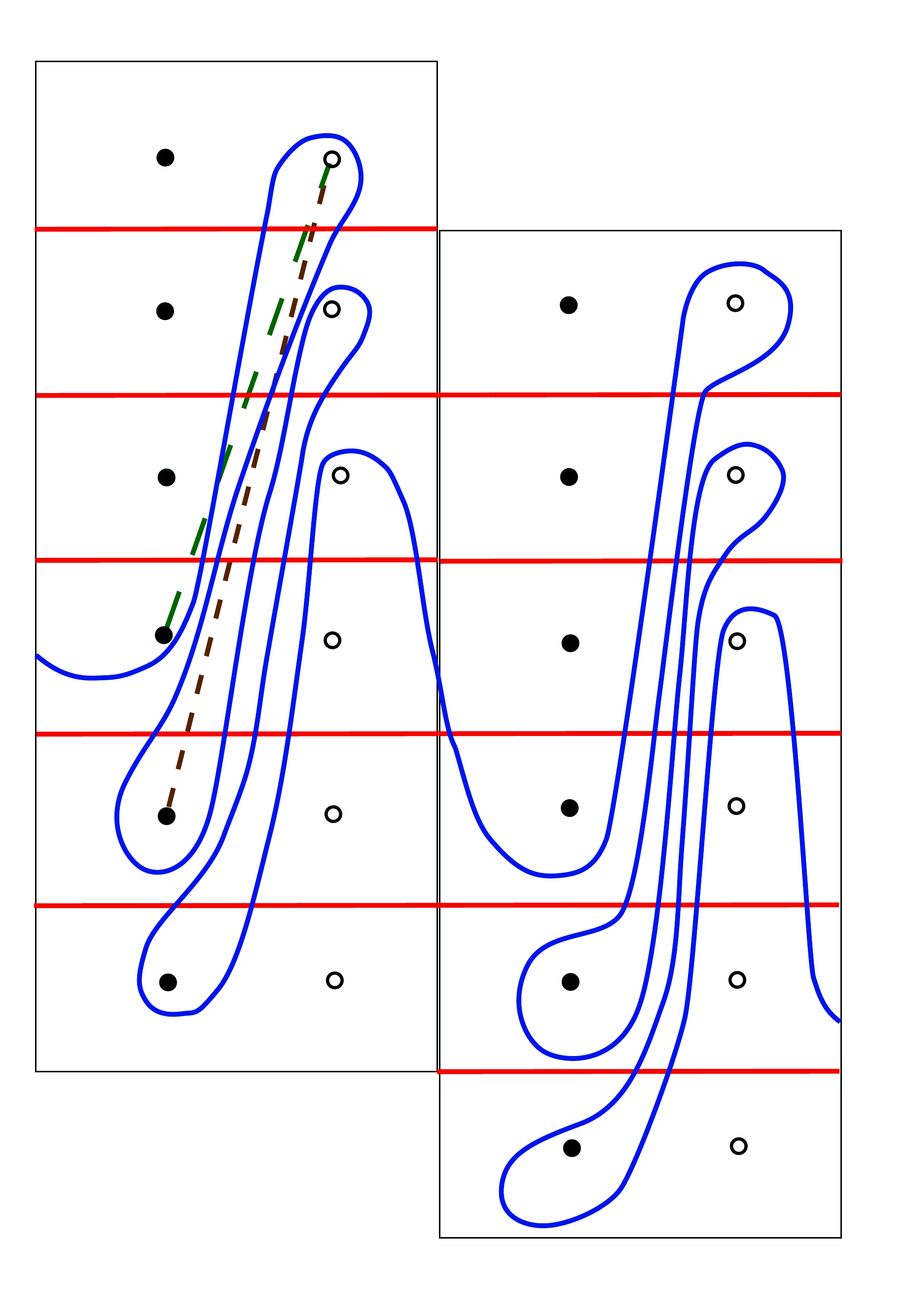}};
 \draw [decorate,decoration={brace,amplitude=4pt},xshift=0.5cm,yshift=0pt]
      (-0.2,1.4) -- (-0.2,-1) node [midway,right,xshift=0cm,yshift=0.1cm] {};
      \node at (1.39,0.45) {\small$n_2 - n_1$};
       \node at (1.1,0) {\small$+1$};
 \draw [decorate,decoration={brace,amplitude=4pt},xshift=0.5cm,yshift=0pt]
      (0,3.4) -- (0,2) node [midway,right,xshift=0cm,yshift=5pt] {\small$n_1$};
       \draw [decorate,decoration={brace,amplitude=4pt,mirror},xshift=0.5cm,yshift=0pt]
      (-1.6,3.2) -- (-1.6,0) node [midway,left,xshift=-0.1cm,yshift=5pt] {\small$n_2$};
       \draw [decorate,decoration={brace,amplitude=4pt,mirror},xshift=0.5cm,yshift=0pt]
      (-1.9,-0.8) -- (-1.9,-2.4) node [midway,left,xshift=0cm] {\small$n_1$};
    \node at (-2.2,2.7) [text=red] {$\alpha_5$};
        \node at (-2.2,1.6) [text=red] {$\alpha_4$};
\node at (-2.2,0.6) [text=red] {$\alpha_3$};
  \node at (-2.2,-0.55) [text=red] {$\alpha_2$};
      \node at (-2.2,-1.6) [text=red] {$\alpha_1$};
       \node at (-0.7,1.8) [text=OliveGreen] {$\gamma_-$};
 \node at (-1.2,-0.75) [text=Brown] {$\gamma_+$};
       \end{tikzpicture}
       \label{subfig:lspaceblock+1}
       }
    \end{minipage}%
     \centering
    \begin{minipage}{0.5\linewidth}
        \subfloat[When $1\leq s \leq n_1 -1.$]{
         \begin{tikzpicture}
           \node at (0,0) {\includegraphics[scale=0.4]{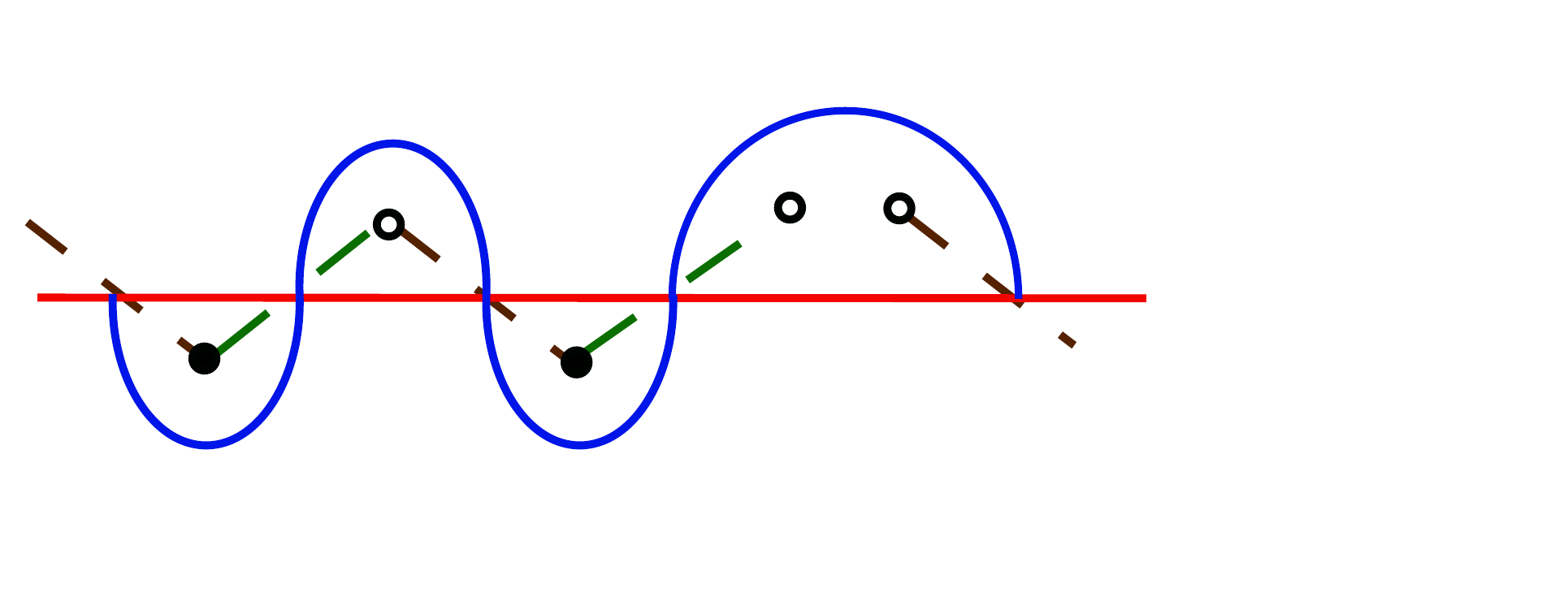}};
            \node at (-4.1,0) [text=red] {$\alpha_s$};
                     \draw [decorate,decoration={brace,amplitude=4pt},xshift=0.5cm,yshift=0pt]
      (-2.9,0.75) -- (-2.05,0.75) node [midway,above,yshift=0.15cm] {\small$s-1$};
             \node at (-1.9,-0.5) {\small$\cdots$};
             \draw [decorate,decoration={brace,amplitude=4pt,mirror},xshift=0.5cm,yshift=0pt]
      (-4,-0.7) -- (-1,-0.7) node [midway,below,yshift=-0.15cm] {\small$s$};
            \node at (0.34,0.42) {\tiny$\cdots$};
             \draw [decorate,decoration={brace,amplitude=4pt},xshift=0.5cm,yshift=0pt]
      (-0.6,0.7) -- (0.2,0.7) node [midway,above,yshift=0.15cm] {\small$n_1 + n_2-2s$};
         \node at (-2.1,-0.2) [text=OliveGreen] {$\gamma_-$};
    \node at (-0.2,-0.2) [text=OliveGreen] {$\gamma_-$};
 \node at (-1.3,0.2) [text=Brown] {$\gamma_+$};
\node at (1.7,-0.2) [text=Brown] {$\gamma_+$};
\node at (-3.2,0.2) [text=Brown] {$\gamma_+$};
       \end{tikzpicture}
       \label{subfig:lspaceblock+2}
        }\\
          \subfloat[When $n_1\leq s \leq n_2 -1.$]{
         \begin{tikzpicture}
          \node at (-4.2,0) [text=red] {$\alpha_s$};
           \node at (0,0) {\includegraphics[scale=0.4]{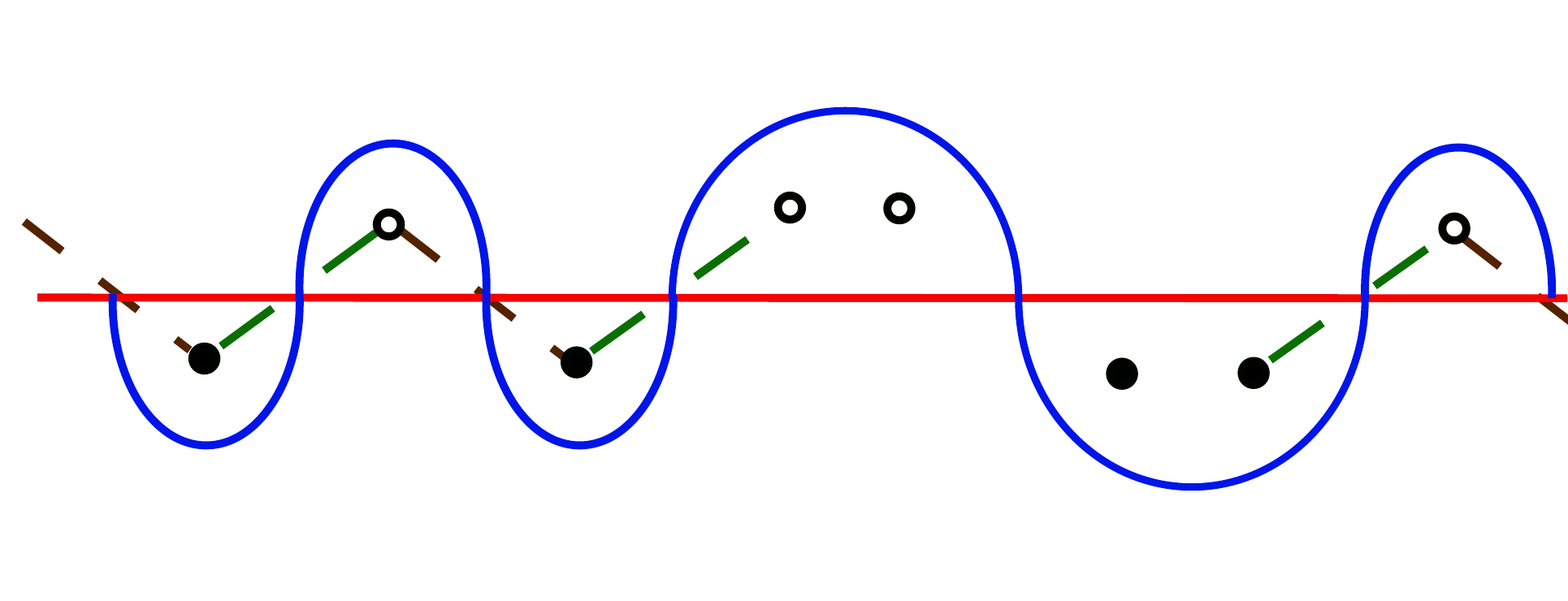}};
                     \draw [decorate,decoration={brace,amplitude=4pt},xshift=0.5cm,yshift=0pt]
      (-2.9,0.75) -- (-2.05,0.75) node [midway,above,yshift=0.15cm] {\small$n_1-1$};
             \node at (-1.9,-0.5) {\small$\cdots$};
             \draw [decorate,decoration={brace,amplitude=4pt,mirror},xshift=0.5cm,yshift=0pt]
      (-4,-0.7) -- (-1,-0.7) node [midway,below,yshift=-0.15cm] {\small$n_1$};
            \node at (0.34,0.42) {\tiny$\cdots$};
             \draw [decorate,decoration={brace,amplitude=4pt},xshift=0.5cm,yshift=0pt]
      (-0.6,0.7) -- (0.2,0.7) node [midway,above,yshift=0.15cm] {\small$n_2-s$};
           \node at (2,-0.4) {\tiny$\cdots$};
             \draw [decorate,decoration={brace,amplitude=4pt,mirror},xshift=0.5cm,yshift=0pt]
      (1.1,-0.6) -- (2,-0.6) node [midway,below,yshift=-0.15cm] {\small$s-n_1+1$};
      \draw [Brown, thick] (1.2,0) circle (0.2cm);
               \node at (-2.1,-0.2) [text=OliveGreen] {$\gamma_-$};
    \node at (-0.2,-0.2) [text=OliveGreen] {$\gamma_-$};
 \node at (-1.3,0.2) [text=Brown] {$\gamma_+$};
\node at (4,-0.2) [text=Brown] {$\gamma_+$};
\node at (-3.2,0.2) [text=Brown] {$\gamma_+$};
\node at (2.6,0.2) [text=OliveGreen] {$\gamma_-$};
       \end{tikzpicture}
       \label{subfig:lspaceblock+3}
        }
    \end{minipage}
    \caption{The case $n_2 \geq n_1.$ The only generators without a marking are given by the circled intersection point in Figure \protect\subref{subfig:lspaceblock+3}.}
    \label{fig:class3_1}
\end{figure}
\subsubsection{When $n_2 \geq n_1$} The $(1,1)$  diagram in this case is depicted in Figure \ref{fig:class3_1}. We proceed as before. For a fixed lift $\bb,$ label all the lifts of $\alpha$ which intersect $\bb$ by $\alpha_1, \cdots, \alpha_{n_1+n_2-1}.$ Identify $\aa$ with $\alpha_s$ for $1\leq s\leq n_1+n_2-1$ in the $s$-th block.  There is an ambiguity in defining the end point of each block. We define the end point of each block to be the first intersection point after $\bb$ travels above a $w$ basepoint in the next block. See the right hand side of Figure \ref{fig:class3_1}.

By Definition \ref{def:gamma}, the slope of $\gamma_+$ is $n_2$ and the slope of  $\gamma_-$ is $-(-n_2+1)=n_2-1$.
Observe also that aside from the intersection point circled in Figure \ref{fig:class3_1}  \subref{subfig:lspaceblock+3},  every intersection point is marked exactly once, where the sign of the marking depends on the parity of the Maslov grading.

We simply count the number of basepoints in the bigons in the upper half plane. This is given by 
\begin{align*}
   \overbrace{1,\cdots,1}^{s-1},n_1+n_2-2s 
   \intertext{for  $1\leq s \leq n_1-1$, and }
   \overbrace{1,\cdots,1}^{n_1-1},n_2 - s, 1 
\end{align*}
for  $n_1\leq s \leq n_2-1$. When $n_1 \neq n_2,$ we have $n_2-1 > (n_1 + n_2 -1)/2$, therefore by the symmetry  we can fill out  the remaining complex. Specifically, the horizontal-vertical arrows of length $(n_2-n_1,1)$ in the $(n_1)$-th block is identified after reflection with the horizontal-vertical arrows of length $(1,n_2-n_1)$ in the $(n_2-1)$-th block.  It follows that the remaining horizontal arrows are all of length $1$, and there are in total 
\[
 \left(\sum^{n_1 }_{i=1}i \right) -1 = \frac{(n_1+2)(n_1-1)}{2}
\]
of them. In summary, we have shown that the sequence of length of horizontal arrows in the knot Floer complex is given by  
\[
\underbrace{\overbrace{1,\cdots,1}^{s-1},n_1+n_2-2s} _{\text{for } 1\leq s \leq n_1-1},  \underbrace{\overbrace{1,\cdots,1}^{n_1-1},n_2 - s, 1 }_{\text{for } n_1\leq s \leq n_2-1},  \overbrace{1,\cdots,1}^{\frac{(n_1+2)(n_1-1)}{2}}.
\]
Moreover, the only generators without a marking are those whose outgoing horizontal arrows are of length $n_2 - s$ in the $s$-th block for $n_1\leq s \leq n_2-1.$ The above analysis does not cover the case when $n_1 = n_2,$ but it is straightforward to check that the conclusion also holds  there. (When $n_1 = n_2,$ the $(n_1)$-th block consists of $n_1-1$ bigons  in each half plane, where each bigon has exactly one basepoint, and the rest of the complex follows from symmetry.)
\subsubsection{When $n_2 < n_1$} The $(1,1)$  diagram in this case is depicted in Figure \ref{fig:class3_2}. This case is parallel to the previous case, and one can similarly work out the sequence of length of horizontal arrows using the right hand side diagrams in  Figure \ref{fig:class3_2}. Putting together the discussion on both cases, we have proved the following.
\begin{figure}[htb!]
    \centering
    \begin{minipage}{0.5\linewidth}
     \centering
       \subfloat[The $(1,1)$ diagram when $n_1 = 4 , n_2 = 2 $.]{
       \begin{tikzpicture}
       \begin{scope}[thin, black!0!white]
          \draw  (-4, 0) -- (5.5, 0);
      \end{scope}
                 \node at (0.8,0) {\includegraphics[scale=0.3]{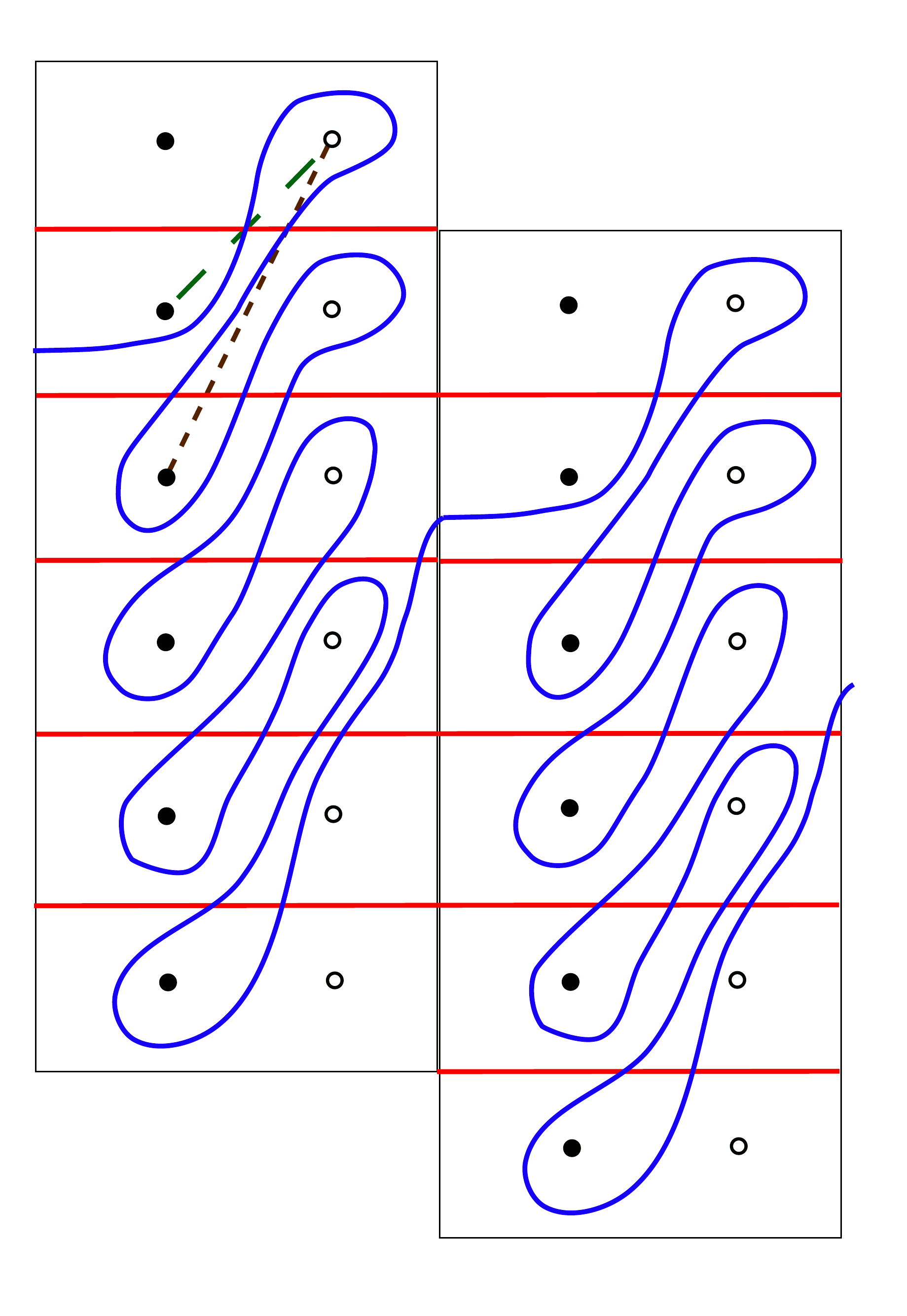}};
 \draw [decorate,decoration={brace,amplitude=4pt},xshift=0.5cm,yshift=0pt]
      (-0.2,-0.8) -- (-0.2,-2) node [midway,right,xshift=0cm,yshift=0.1cm] {\small$n_2$};
 \draw [decorate,decoration={brace,amplitude=4pt},xshift=0.5cm,yshift=0pt]
      (0,3.4) -- (0,0) node [midway,right,xshift=0.2cm,yshift=5pt] {\small$n_1$};
       \draw [decorate,decoration={brace,amplitude=4pt,mirror},xshift=0.5cm,yshift=0pt]
      (-1.6,3.3) -- (-1.6,2) node [midway,left,xshift=-0.1cm,yshift=5pt] {\small$n_2$};
       \draw [decorate,decoration={brace,amplitude=4pt,mirror},xshift=0.6cm,yshift=0pt]
      (-1.9, 1.4) -- (-1.9,-2.4) node [midway,left,yshift=0.15cm] {\small$n_1$};
    \node at (-2.2,2.7) [text=red] {$\alpha_5$};
        \node at (-2.2,1.6) [text=red] {$\alpha_4$};
\node at (-2.2,0.6) [text=red] {$\alpha_3$};
  \node at (-2.2,-0.55) [text=red] {$\alpha_2$};
      \node at (-2.2,-1.6) [text=red] {$\alpha_1$};
          \node at (-0.6,3) [text=OliveGreen] {$\gamma_-$};
 \node at (0.3,2.8) [text=Brown] {$\gamma_+$};
       \end{tikzpicture}
       \label{subfig:lspaceblock+4}
       }
    \end{minipage}%
     \centering
    \begin{minipage}{0.5\linewidth}
        \subfloat[When $1\leq s \leq n_2-1.$]{
         \begin{tikzpicture}
           \node at (0,0) {\includegraphics[scale=0.4]{lspaceblock+2}};
            \node at (-4.1,0) [text=red] {$\alpha_s$};
                     \draw [decorate,decoration={brace,amplitude=4pt},xshift=0.5cm,yshift=0pt]
      (-2.9,0.75) -- (-2.05,0.75) node [midway,above,yshift=0.15cm] {\small$s-1$};
             \node at (-1.9,-0.5) {\small$\cdots$};
             \draw [decorate,decoration={brace,amplitude=4pt,mirror},xshift=0.5cm,yshift=0pt]
      (-4,-0.7) -- (-1,-0.7) node [midway,below,yshift=-0.15cm] {\small$s$};
            \node at (0.34,0.42) {\tiny$\cdots$};
             \draw [decorate,decoration={brace,amplitude=4pt},xshift=0.5cm,yshift=0pt]
      (-0.6,0.7) -- (0.2,0.7) node [midway,above,yshift=0.15cm] {\small$n_1 + n_2-2s$};
            \node at (-2.1,-0.2) [text=OliveGreen] {$\gamma_-$};
    \node at (-0.2,-0.2) [text=OliveGreen] {$\gamma_-$};
 \node at (-1.3,0.2) [text=Brown] {$\gamma_+$};
\node at (1.7,-0.2) [text=Brown] {$\gamma_+$};
\node at (-3.2,0.2) [text=Brown] {$\gamma_+$};
       \end{tikzpicture}
       \label{subfig:lspaceblock+2'}
        }\\
          \subfloat[When $n_2\leq s \leq n_1 -1.$]{
         \begin{tikzpicture}
          \node at (-4.2,0) [text=red] {$\alpha_s$};
           \node at (0,0) {\includegraphics[scale=0.4]{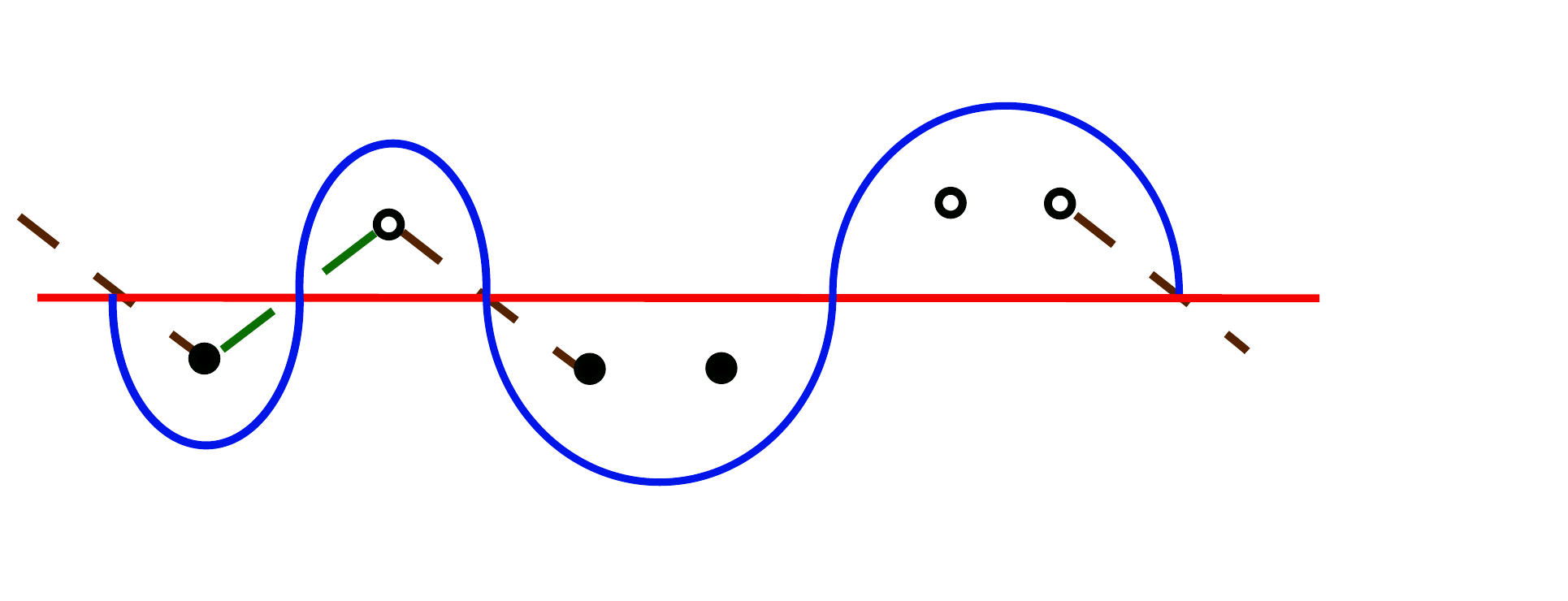}};
                     \draw [decorate,decoration={brace,amplitude=4pt},xshift=0.5cm,yshift=0pt]
      (-2.9,0.75) -- (-2.05,0.75) node [midway,above,yshift=0.15cm] {\small$n_2-1$};
             \node at (-1.9,-0.5) {\small$\cdots$};
             \draw [decorate,decoration={brace,amplitude=4pt,mirror},xshift=0.5cm,yshift=0pt]
      (-3.9,-0.7) -- (-2.9,-0.7) node [midway,below,yshift=-0.15cm] {\small$n_2-1$};
            \node at (1.15,0.42) {\tiny$\cdots$};
             \draw [decorate,decoration={brace,amplitude=4pt},xshift=0.5cm,yshift=0pt]
      (0.2,0.7) -- (1.1,0.7) node [midway,above,yshift=0.15cm] {\small$n_1-s$};
           \node at (-0.63,-0.4) {\tiny$\cdots$};
             \draw [decorate,decoration={brace,amplitude=4pt,mirror},xshift=0.5cm,yshift=0pt]
      (-1.6,-0.6) -- (-0.7,-0.6) node [midway,below,yshift=-0.15cm] {\small$s-n_2+1$};
          \node at (-2.1,-0.2) [text=OliveGreen] {$\gamma_-$};
  \draw [OliveGreen, thick] (0.26,0) circle (0.2cm);
 \node at (-1.3,0.2) [text=Brown] {$\gamma_+$};
\node at (1.9,-0.2) [text=Brown] {$\gamma_+$};
\node at (-3.2,0.2) [text=Brown] {$\gamma_+$};
       \end{tikzpicture}
       \label{subfig:lspaceblock+5}
        }
    \end{minipage}
    \caption{The case $n_2 < n_1.$ The only generators without a marking are given by the intersection point circled in Figure \protect\subref{subfig:lspaceblock+5}.}
    \label{fig:class3_2}
\end{figure}
\begin{proposition}\label{prop:class3}
    For $n_1>0,n_2>1$, $K^{+}([2n_1,1,-2n_2])$ is a negative L-space knot, with the length of horizontal arrows given in order by the follows. \begin{itemize}
        \item When $n_2 \geq n_1,$
        \begin{align*} 
  \underbrace{\overbrace{1,\cdots,1}^{s-1},n_1+n_2-2s} _{\text{for } 1\leq s \leq n_1-1},  \underbrace{\overbrace{1,\cdots,1}^{n_1-1},\overline{n_2 - s}, 1 }_{\text{for } n_1\leq s \leq n_2-1},  \overbrace{1,\cdots,1}^{\frac{(n_1+2)(n_1-1)}{2}}.
\end{align*}
Given a basis, each generator with odd Maslov grading has one $-$ marking.
Each generator whose outgoing horizontal arrow is marked by the overline has no marking. Apart from them, each generator with even Maslov grading has one $+$ marking.
\item When $n_2 < n_1,$
        \begin{align*} 
  \underbrace{\overbrace{1,\cdots,1}^{s-1},n_1+n_2-2s} _{\text{for } 1\leq s \leq n_2-1},  \underbrace{\overbrace{1,\cdots,1}^{n_2-1},\overline{n_1 - s}} _{\text{for } n_2 \leq s \leq n_1-1},  \overbrace{1,\cdots,1}^{\frac{(n_2+2)(n_2-1)}{2}}.
\end{align*}
Given a basis, each generator whose incoming horizontal arrow is marked by the overline has no marking. Apart from them, each generator with odd Maslov grading has one $-$ marking. 
Each generator with even Maslov grading has one $+$ marking.
    \end{itemize}   
\end{proposition}
\subsection{Case $K^{-}([2n_1,1,-2n_2])$ with $n_1>0,n_2>1$} \label{subsec:class4}  By \cite[Theorem 1.2]{GLV}, $K^{-}([2n_1,1,-2n_2])$ is a positive $(1,1)$ L-space knot. It suffices to record (say) the length of vertical arrows. The process to determine such a sequence is entirely parallel to the process in Section \ref{subsec:class3}. Therefore we will skip the proof, giving only the conclusion, as follows.
\begin{proposition}\label{prop:class4}
    For $n_1>0,n_2>1$, $K^{-}([2n_1,1,-2n_2])$ is a positive L-space knot, with the length of vertical arrows given in order by the follows.
    \begin{itemize}
        \item When $n_2 \geq n_1 +2,$
        \begin{align*} 
  \underbrace{\overbrace{1,\cdots,1}^{s-1},n_1+n_2-2s} _{\text{for } 1\leq s \leq n_1},  \underbrace{\overbrace{1,\cdots,1}^{n_1},\overline{n_2 - s- 1}} _{\text{for } n_1+1\leq s \leq n_2-2},  \overbrace{1,\cdots,1}^{\frac{n_1(n_1+3)}{2}}.
\end{align*}
Given a basis, each generator with even Maslov grading has one $-$ marking. Each generator whose outgoing vertical arrow is marked by the overline has no marking. Apart from them, each generator with odd Maslov grading has one $+$ marking. 
\item When $n_2 \leq n_1+1,$
        \begin{align*} 
  \underbrace{\overbrace{1,\cdots,1}^{s-1},n_1+n_2-2s} _{\text{for } 1\leq s \leq n_2-2}, \overbrace{1,\cdots,1}^{n_2-2},\overline{n_1-n_2+1},  \underbrace{\overbrace{1,\cdots,1}^{n_2-1},\overline{n_1 - s+1}} _{\text{for } n_2 \leq s \leq n_1},  \overbrace{1,\cdots,1}^{\frac{n_2(n_2-1)}{2}}.
\end{align*}
Given a basis,  each generator whose incoming vertical arrow is marked by the overline has no marking.  Apart from them, each generator with even Maslov grading has one $-$ marking.  Each generator with odd Maslov grading has one $+$ marking.
    \end{itemize}  
\end{proposition}

\bibliographystyle{amsalpha}
\bibliography{bibliography}

\end{document}